\documentclass[english]{article}
\usepackage[T1]{fontenc}
\usepackage[latin9]{inputenc}
\usepackage{amsmath}
\usepackage{amssymb}
\usepackage{esint}
\usepackage[authoryear]{natbib}

\usepackage{amsmath,amsfonts,latexsym,amssymb,amsthm}
\usepackage{graphicx}

\usepackage{color}
\usepackage{colordvi}

\newcommand{\op}{\mathbf{o_{P_o}}}
\newcommand{\Op}{\mathbf{O_{P_o}}}
\newcommand{\od}{\mathbf{o}}
\newcommand{\Od}{\mathbf{O}}

\newcommand{\fr}{\boldsymbol{|}}

\newcommand{\R}{\mathbb{R}}
\newcommand{\Z}{\mathbb{Z}}
\newcommand{\N}{\mathbb{N}}
\newcommand{\1}{\mathbb{I}}
\newcommand{\tj}{\theta_j}

\newcommand{\tr}{\textrm{tr}}

\newcommand{\bd}{\bar \theta_{d,k}}
\newcommand{\bz}{\bar \theta_{d_o,k}}

\newcommand{\tlo}{\Theta_k(\beta,L)}
\newcommand{\tlt}{\Theta_k(\beta-\frac{1}{2},L)}

\newcommand{\bn}{B_k(\bar \theta_{d,k},2l_0  \delta_n)}

\newcommand{\jkl}{\sum_{j=0}^k}
\newcommand{\xn}{X}             

\newcommand{\kp}{k_{n}^{'}}

\newcommand{\vb}{\bar v_n}

\newcommand{\az}{\alpha}
\newcommand{\be}{\beta}
\newcommand{\dn}{\delta_n}

\newcommand{\an}{\alpha_n}

\newcommand{\fn}{\frac{1}{n}}
\newcommand{\ft}{\frac{1}{2}}
\newcommand{\e}{\epsilon}
\newcommand{\s}{\sigma}
\newcommand{\lb}{\lambda}
\newcommand{\ty}{\theta}
\newcommand{\en}{\epsilon_n}
\newcommand{\te}{\epsilon_n}
\newcommand{\Sto}{\mathcal{S}}
\newcommand{\Det}{\mathcal{D}}

\numberwithin{equation}{section}
\theoremstyle{plain}
\newtheorem{thm}{Theorem}[section]
\newtheorem{lem}{Lemma}[section]
\newtheorem{cor}{Corollary}[section]
\newtheorem{rem}{Remark}[section]

\begin{document}

\title{Bayesian semi-parametric estimation of the long-memory parameter under FEXP-priors.} 

\author{Willem Kruijer and Judith Rousseau}

\maketitle

\begin{abstract}
For a Gaussian time series with long-memory behavior,
we use the FEXP-model for semi-parametric estimation of the long-memory parameter $d$. The true spectral density $f_o$
is assumed to have long-memory parameter $d_o$ and a FEXP-expansion of Sobolev-regularity $\be > 1$.
We prove that when $k$ follows a Poisson or geometric prior, or a sieve prior increasing at rate $n^{\frac{1}{1+2\be}}$, $d$ converges to $d_o$ at a suboptimal rate. When the sieve prior increases at rate $n^{\frac{1}{2\be}}$ however, the minimax rate is almost obtained.
Our results can be seen as a Bayesian equivalent of the result which
Moulines and Soulier obtained for some frequentist estimators.
\end{abstract}




\section{Introduction}

Let $X_t$, $t \in \Z$, be a stationary Gaussian time series with zero mean  and spectral density $f_o(x)$, $x \in [-\pi,\pi]$, which takes the form
\begin{equation} \label{longmem}
|1-e^{i x}|^{-2 d_o} M_o (x), \qquad x \in [-\pi,\pi],
\end{equation} \noindent
where $d_o \in (-\ft,\ft)$ is called the long-memory parameter, and $M$ is a slowly-varying bounded function that describes the short-memory behavior of the series.
If $d_o$ is positive, this makes the autocorrelation function $\rho(h)$ decay polynomially, at rate $h^{-(1-2 d_o)}$, and the time series is said to have long-memory. When $d_o=0$, $X_t$ has short memory,
and the case $d_o < 0$ is referred to as intermediate memory. Long memory time series models are used in a wide range of applications, such as hydrological or financial time series; see for example  \citet{beran94} or \citet{robinson94a}. In parametric approaches, a finite dimensional model is used for the short memory part $M_o$; the most well known example is the ARFIMA(p,d,q) model.
The asymptotic properties of maximum likelihood estimators (\citet{d1} or \citet{lrz1}) and Bayesian estimators (\citet{pr1})
have been established in such models and these estimators are consistent and asymptotically normal with a convergence rate of order $\sqrt{n}$. However when the model for the short memory part is misspecified, the estimator for $d$ can be inconsistent, calling for semi-parametric methods for the estimation of $d$.
A key feature of semi-parametric estimators of the long-memory parameter is that they converge at a rate which depends on the smoothness of the short-memory part, and apart from the case where $M_o$ is infinitely smooth, the convergence rate is smaller than $\sqrt{n}$. The estimation of the long-memory parameter $d$ can thus be considered as a non-regular semi-parametric problem. In \citet{ms1} (p. 274) it is shown that when $f_o$ satisfies \eqref{fo}, the minimax rate for $d$ is $n^{-\frac{2\be-1}{4\be}}$. There are frequentist estimators for $d$ based on the periodogram that achieve this rate (see \citet{hcms02} and \citet{ms1}).

Although Bayesian methods in long-memory models have been widely used (see for instance \citet{ko:qu:vannucci:09}, \citet{Jensen:04} or \citet{holan:mcelroy:chakra:09}), the literature on convergence properties of non- and semi-parametric estimators is sparse. \citet{rcl}  (RCL hereafter) obtain consistency and rates for the $L_2$-norm of the log-spectral densities (Theorems 3.1 and 3.2), but for $d$ they only show consistency (Corollary 1).
No results exist on the posterior concentration rate on $d$, and thus on the convergence rates of Bayesian semi-parametric estimators of $d$. In this paper we aim to fill this gap for a specific family of semi-parametric priors.

We study Bayesian estimation of $d$ within the FEXP-model (\citet{beran93}, \citet{robinson95a}), that contains densities of the form
\begin{equation} \label{fexp}
f_{d,k,\ty}(x) = |1-e^{i x}|^{-2 d} \exp\left\{ \jkl \ty_j \cos(j x)\right\},
\end{equation} \noindent
where $d \in (-\ft,\ft)$, $k$ is a nonnegative integer and $\ty \in \R^{k+1}$. The factor $\exp\{ \jkl \ty_j \cos(j x)\}$ models the function $M_o$ in \eqref{longmem}.
In contrast to the original finite-dimensional FEXP-model (\citet{beran93}), where $k$ was supposed to be known, or at least bounded, $f_o$ may have an infinite FEXP-expansion,
and we allow $k$ to increase with the number of observations to obtain approximations $f$ that are increasingly close to $f_o$. Note that the case where the true spectral density  satisfies $f_o = f_{d_o,k_o, \ty_o}$, is considered in \citet{holan:mcelroy:chakra:09}.
In this paper we will pursue a fully Bayesian semi-parametric estimation of $d$, the short memory parameter being considered as an infinite-dimensional nuisance parameter.
We obtain results on the convergence rate and asymptotic distribution of the posterior distribution for $d$, which we summarize below in section \ref{resultsSummary}.
These are to our knowledge the first of this kind in the Bayesian literature on semi-parametric time series.  First we state the most important assumptions.

\subsection{Asymptotic framework}

For observations $\xn=(X_1,\ldots,X_n)$ from a Gaussian stationary time series with spectral density $f$, let $T_n(f)$ denote the associated covariance matrix and $l_n(f)$ denote the log-likelihood
\begin{equation*}\label{logLikelihood}
l_n(f) = -\frac{n}{2} \log (2\pi) -\frac{1}{2} \log \det(T_n(f)) - \ft  X^t T_n^{-1}(f) X.
\end{equation*} \noindent

We consider semi-parametric priors on $f$ based on the FEXP-model defined by \eqref{fexp}, inducing a parametrization of $f$ in terms of $(d,k,\ty)$.  
Assuming priors $\pi_d$ for $d$, and, independent of $d$, $\pi_k$ for $k$ and $\pi_{\ty|k}$ for $\ty|k$, we study the (marginal) posterior for $d$, given by
\begin{equation}\label{marginalPosterior}
\Pi(d \in D|\xn) =
\frac{\sum_{k=0}^{\infty} \pi_k(k) \int_{D} \int_{\R^{k+1}}
e^{l_n(d,k,\ty)} d\pi_{\ty|k}(\ty) d\pi_d(d)}
{\sum_{k=0}^{\infty} \pi_k(k) \int_{-\ft}^\ft
\int_{\R^{k+1}} e^{l_n(d,k,\ty)} d\pi_{\ty|k}(\ty) d\pi_d(d)}.
\end{equation} \noindent
The posterior mean or median  can be taken as point-estimates for $d$, but we will focuss on the posterior $\Pi(d |\xn)$ itself.

It is assumed that the true spectral density is of the form
\begin{equation} \label{fo}
\begin{split}
f_o(x) &= |1-e^{i x}|^{-2 d_o} \exp\left\{ \sum_{j=0}^\infty \ty_{o,j} \cos(j x)\right\}, \\
\ty_o &\in \Theta(\be,L_o) = \{\ty \in l_2(\N) : \sum_{j=0}^\infty \ty_j^2 (1+j)^{2\be} \leq L_o \},
\end{split}
\end{equation}
for some known $\be>1$.

In particular, we derive bounds on the rate at which $\Pi(d \in D|\xn)$ concentrates at $d_o$, together with a
Bernstein - von -Mises (BVM) property of this distribution.
The posterior concentration rate  for $d$ is defined as the
fastest sequence $\alpha_n$ converging to zero such that
\begin{equation} \label{convRateDefinition}
\Pi(|d-d_o| < K \alpha_n |\xn) \overset{P_o}{\rightarrow} 0, \quad \mbox{ for a given fixed K. }
\end{equation}

\subsection{Summary of the results} \label{resultsSummary}

Under the above assumptions we obtain several results for the asymptotic distribution of $\Pi(d \in D|\xn)$.
Our first main result (Theorem \ref{BVMtheorem}) states that under the sieve prior $k_n \sim (n/\log n)^{1/(2\be)}$, $\Pi(d \in D|\xn)$
is asymptotically Gaussian, and we give expressions for the posterior mean and the posterior variance. A consequence (Corollary \ref{cor2}) of this result is that the convergence rate for $d$ under this prior is at least $\dn = (n/\log n)^{-\frac{2\be - 1}{4\be}}$, i.e. in \eqref{convRateDefinition} $\alpha_n$ is bounded by $\dn$. Up to a $\log n$ term, this is the minimax rate.

By our second main result (Theorem \ref{suboptimalTheorem}), the rate for $d$ is suboptimal when $k$ is given a
a Poisson or a Geometric distribution, or a sieve prior $\kp \sim (n / \log n)^{\frac{1}{1+2\be}}$. More precisely,
there exists $f_o$ such that the posterior concentration rate $\alpha_n$ is greater than $n^{-(\be - 1/2)/(2\be + 1)}$, and thus suboptimal.
Consequently, despite having good frequentist properties for the estimation of the spectral density $f$ itself (see RCL), these priors are much less suitable for the estimation of $d$.
This is not a unique phenomenon in (Bayesian) semi-parametric estimation and is encountered for instance in the estimation of a linear functional of the signal in white-noise models, see \citet{li:zhao:02} or \citet{arbel:10}.

The BVM property means that asymptotically the posterior distribution of $d$ behaves like $\alpha_n^{-1} (d - \hat{d}) \sim \mathcal N(0,1)$, where $\hat{d}$ is an estimate whose frequentist distribution (associated to the parameter $d$) is $\mathcal N(d_o,\alpha_n^2)$. We prove such a property on the posterior distribution of $d$ given $k=k_n$.
In regular parametric long-memory models, the BVM property has been established by \citet{pr1}. It is however much more difficult to establish BVM theorems in infinite dimensional setups, even for independent and identically distributed models; see for instance \citet{freedman:99}, \citet{castillo:10} and \citet{riro10}. In particular it has been proved that the BVM property may not be valid, even for reasonable priors. The BVM property is however very useful since it induces a strong connection between frequentist and Bayesian methods. In particular, it implies that Bayesian credible regions are asymptotically also frequentist confidence regions with the same nominal level. In section \ref{sec2} we discuss this  issue in more detail.

\subsection{Overview of the paper}

In section \ref{sec2}, we present three families of priors based on the sieve model defined by (\ref{fexp}) with either  $k$ increasing at the rate $(n/\log n)^{1/(2\be)}$, $k$ increasing at the rate $(n/\log n)^{1/(2\be+1)}$ or with random $k$. We study the behavior of the posterior distribution of $d$ in each case and prove that the former leads to optimal frequentist procedures while the latter two lead to suboptimal procedures.
In section \ref{sec3} we give a decomposition of $\Pi(d \in D|\xn)$ defined in \eqref{marginalPosterior}, and obtain bounds for the terms in this decomposition in sections \ref{sec3d} and \ref{sec3c}. Using these results we prove Theorems \ref{BVMtheorem} and \ref{suboptimalTheorem} in respectively sections \ref{sec4} and \ref{sec5}.
Conclusions are  given in section \ref{sec6}. In the appendices we give the proofs of the lemmas in section \ref{sec3}, as well as some additional results on the derivatives of the log-likelihood. The proofs of various technical results can be found in the supplementary material. We conclude this introduction with an overview of the notation.

\subsection{Notation}
The $m$-dimensional identity matrix is denoted $I_m$.
We write $\fr A\fr$ for the Frobenius or Hilbert-Schmidt norm of a matrix $A$, i.e. $\fr A\fr = \sqrt{\tr A A^t}$, where $A^t$ denotes the transpose of $A$.
The operator or spectral norm is denoted $\|A\|^2 = \sup_{\|x\|=1} x^tA^t A x$. We also use $\|\cdot\|$ for the
Euclidean norm on $\R^k$ or $l^2(\N)$. The inner-product is denoted $|\cdot|$.
We make frequent use of the  relations
\begin{equation} \label{matrixCalculus}
\begin{split}
& \fr A B \fr = \fr BA \fr \leq \|A\| \cdot \fr B \fr , \quad \|A B\| \leq \|A\|\cdot \|B\|, \quad \|A\| \leq \fr A \fr \leq \sqrt{n} \|A\|, \\
& |\tr(A B)| = |\tr(BA)| \leq \fr A\fr \cdot \fr B\fr, \quad |x^t A x| \leq x^t x \|A\| ,
\end{split}
\end{equation}
see \citet{d1}, p. 1754.
For any function $h \in L_1([-\pi,\pi])$, $T_n(h)$ is the matrix with entries  $\int_{-\pi}^{\pi} e^{i|l-m|x} h(x)dx$, $l,m=1,\ldots,n$.
For example, $T_n(f)$ is the covariance matrix of observations $\xn=(X_1,\ldots,X_n)$ from a time series with spectral density $f$. If $h$ is square integrable on $[-\pi,\pi]$ we note
 $$\|h\|_2 = \int_{-\pi}^{\pi} h^2(x) dx.$$

The norm $l$ between spectral densities $f$ and $g$ is defined as
\begin{equation*}
l(f,g) = \frac{1}{2\pi} \int_{-\pi}^{\pi} (\log f(x) - \log g(x))^2 dx.
\end{equation*} \noindent
Unless stated otherwise, all expectations and probabilities are with respect to $P_o$, the law associated with the true spectral density $f_o$.
To avoid ambiguous notation (e.g. $\ty_0$ versus $\ty_{0,0}$) we write $\ty_o$ instead of $\ty_0$. Related quantities such as $f_o$ and $d_o$ are also denoted with the $o$-subscript.

The symbols $o_P$ and $O_P$ have their usual meaning. We use boldface when they are
uniform over a certain parameter range. Given a probability law $P$, a family of random variables $\{W_d\}_{d \in A}$
and a positive sequence $a_n$, $W_d = \mathbf{o_P}(a_n,A)$ means that
\begin{equation*}
P\left( \sup_{d\in A} |W_d|/a_n > \e \right) \rightarrow 0, (n\rightarrow \infty).
\end{equation*}
When the parameter set is clear from the context we simply write $\mathbf{o_P}(a_n)$.
In a similar fashion, we write $\od(a_n)$ when the sequence is deterministic.
In conjunction with the $o_P$ and $O_P$ notation we use the letters $\delta$ and $\e$ as follows.
When, for some $\tau>0$ and a probability $P$ we write $Z = O_P(n^{\tau -\e})$, this means that $Z = O(n^{\tau +\e})$ \emph{for all} $\e>0$. When, on the other hand, $Z = O_P(n^{\tau -\delta})$, we mean that this is true for  \emph{some } $\delta>0$. If the value of $\delta$ is of  importance it is given a name, for example $\delta_1$ in Lemma \ref{delta1terms}.

The true spectral density of the process is denoted $f_o$.
We denote $k$-dimensional Sobolev-balls by
\begin{eqnarray}
\tlo &=& \left\{\ty \in \R^{k+1} : \jkl \ty_j^2 (1+j)^{2 \be} \leq L\right\} \subset \R^{k+1}.
\end{eqnarray} \noindent
For any real number $x$, let $x_+$ denote $\max(0,x)$.
The number $r_k$ denotes the sum $\sum_{j\geq k+1} j^{-2}$.
Let $\eta$ be the sequence defined by $\eta_j = -2/j$, $j \geq 1$ and $\eta_0=0$. For an infinite sequence $u = (u_j)_{j\geq 0}$, let $u_{[k]}$ denote the vector of the first $k+1$ elements. In particular,
$\eta_{[k]} = (\eta_0,\ldots,\eta_k)$. The letter $C$ denotes
any generic constant independent of $L_o$ and $L$, which are the constants appearing in the assumptions on $f_o$ and the definition of the prior.

\section{Main results} \label{sec2}

Before stating Theorems \ref{BVMtheorem} and \ref{suboptimalTheorem} in section \ref{mainresults}, we state the assumptions on $f_o$ and the prior, and give examples of priors satisfying these assumptions.

\subsection{Assumptions on the prior and the true spectral density} 
We assume observations $\xn=(X_1,\ldots,X_n)$ from a stationary Gaussian time series with law $P_o$, which is a zero mean Gaussian distribution, whose covariance structure is defined by a spectral density $f_o$ satisfying (\ref{fo}), for known $\be > 1$. It is assumed that for a small constant $t>0$, $d_o \in [-\ft+t,\ft-t]$. 

\textbf{Assumptions on $\Pi$}.
We consider different priors, and first state the assumptions that are common to all these priors.
The prior on the space of spectral densities consists of independent priors $\pi_d$, $\pi_k$ and,
conditional on $k$, $\pi_{\ty|k}$. The prior for $d$ has density $\pi_d$ which is strictly positive on $[-\ft+t, \ft-t]$,
the interval which is assumed to contain $d_o$, and zero elsewhere.
The prior for $\ty$ given $k$ has a density $\pi_{\ty|k}$ with  respect to Lebesgue
measure. This density satisfies condition $\rm{Hyp}( \mathcal K,c_0,\beta,L_o)$, by which we mean that for a subset $\mathcal K$ of $\N$,
 \begin{equation*} \label{hyp:pitheta}
\min_{k \in \mathcal K} \inf_{\ty \in \Theta_k(\be,L_o) } e^{c_0 k \log k}  \pi_{\ty | k}(\ty) >1,
\end{equation*}
where $L_o$ is as in \eqref{fo}.
The choice of $\mathcal K$ depends on the prior for $k$ and $\ty|k$. We consider the following classes of priors.
 \begin{itemize}
\item \textbf{Prior A: } $k$ is deterministic and increasing at rate
\begin{equation} \label{kpktDef1}
k_n = \lfloor k_A  (n/\log n)^{\frac{1}{2\be}} \rfloor,
\end{equation} \noindent
for a constant $k_A>0$.
The prior density for $\ty |k$ satisfies $\rm{Hyp}(\{k_n\}, c_0,\beta-\ft, L_o)$ for some $c_0>0$ and has support $\tlt$.
In addition, for all $\ty , \ty' \in  \Theta_k(\be - \ft,L)$ such that $\| \ty - \ty'\| \leq L (n/\log n)^{-\frac{2\be-1}{4\be}}$,
\begin{equation} \label{priorThetaK}
\log  \pi_{\ty|k}(\ty)  - \log  \pi_{\ty|k}(\ty' ) = h_k^t(\ty - \ty') + o(1),
\end{equation}
for constants $C,\rho_0>0$ and vectors $h_k$ satisfying $\| h_k \| \leq C (n/k)^{1 - \rho_0}$.
Finally, it is assumed that $L$ is  sufficiently large compared to $L_o$.
\item \textbf{Prior B: }
$k$ is deterministic and increasing at rate
\begin{equation*} \label{kpktDef2}
\kp = \lfloor k_B (n / \log n)^{\frac{1}{1+2\be}} \rfloor,
\end{equation*} \noindent
where $k_B$ is such that $\kp < k_n$ for all $n$. The prior for $\ty |k$ has density $\pi_{\ty|k}$ with  respect to Lebesgue
measure which satisfies condition $\rm{Hyp}(\{\kp\}, c_0,\beta, L_o)$ for some $c_0>0$ and is assumed to have support $\tlo$. The density also satisfies
\[\log \pi_{\ty|k}(\ty)-\log \pi_{\ty|k}(\ty')=o(1),
\]
for all $\ty , \ty' \in  \Theta_k(\be,L)$ such that $\| \ty - \ty'\| \leq L (n/\log n)^{-\frac{\be}{2\be+1}}$.
This condition is similar to \eqref{priorThetaK}, but with $h_k=0$, and support $\tlo$.

\item \textbf{Prior C: }
$k \sim \pi_k$ on $\N$ with $e^{-c_1 k\log k} \leq \pi_k(k) \leq e^{- c_2 k \log k}$ for $k$ large enough, where $0<c_1<c_2 < + \infty$.
There exists $\be_s> 1$ such that for all $\be \geq \be_s$, the  prior for $\ty |k$ has density $\pi_{\ty|k}$ with  respect to Lebesgue
measure which satisfies condition $\rm{Hyp}(\{k \leq k_0(n/\log n)^{1/(2\be + 1)}\}, c_0,\beta, L_o)$, for all $k_0>0$ and some $c_0>0$, as soon as $n$ is large enough. It has support included in $\tlo$ and satisfies
\[\log \pi_{\ty|k}(\ty)-\log \pi_{\ty|k}(\ty')=o(1),\]
for all $\ty , \ty' \in  \Theta_k(\be,L)$ such that $\| \ty - \ty'\|
\leq L (n/\log n)^{-\frac{\be}{2\be+1}}$.
\end{itemize}
Note that 
\textbf{prior A} is obtained when we take $\be' = \be-\ft$ in prior B. \newline

\subsection{Examples of priors} \label{examples}

The Lipschitz conditions on $\log \pi_{\ty|k}$ considered for the three types of priors are satisfied for instance for the uniform prior on $\Theta_k(\be-\ft,L)$ (resp. $\Theta_k(\be,L)$), and for the truncated Gaussian prior, where, for some constants $A$ and $\alpha>0$,
$$
\pi_{\ty|k}(\ty) \propto \1_{\tlt}(\ty)\exp \left( - A \sum_{j=0}^k j^{\az} \ty_j^2 \right).
$$
In the case of \textbf{Prior A}, the conditions on $\log  \pi_{\ty|k}$ and $h_k$ in \eqref{priorThetaK} are satisfied for $\alpha < 4\be - 2$. To see this, note that
for all $\ty , \ty^{'} \in \Theta_h(\be - 1/2,L)$,
$$ \sum_{j=0}^k j^\alpha | \theta_j^2 - (\theta_j^{'})^2| \leq L^{1/2}\|\theta - \theta'\| k^{\alpha - \be +1/2} = o((n/k)^{1-\delta}).$$
In the case of \textbf{Prior B} and and \textbf{Prior C} we may choose $\alpha < 2 \be$, since for some positive $k_0$
$$ \sum_{j=0}^k j^\alpha | \theta_j^2 - (\theta_j^{'})^2| \leq L^{1/2}\|\ty  - \ty^{'}\| k^{\alpha - \be} = o(1), $$
for all $k \leq k_0(n/\log n)^{1/(2\be + 1)}$ and all $\ty , \ty^{'} \in \Theta_k(\beta, L)$ such that $\| \ty - \ty^{'}\| \leq (n/\log n)^{-\be/(2\be + 1)}$.

Also a truncated Laplace distribution is possible, in which case
$$
\pi_{\ty|k}(\ty) \propto \1_{\tlt}(\ty)\exp \left( - a \sum_{j=0}^k  |\ty_j| \right).
$$
The condition on $\pi_k$ in \textbf{Prior C } is satisfied for instance by Poisson distributions.

The restriction of the prior to Sobolev balls is required to obtain a proper concentration rate or even consistency of the posterior of the spectral density $f$ itself, which is a necessary step in the proof of our results.
This is discussed in more detail in section \ref{prelim}.

\subsection{Convergence rates and BVM-results under different priors} \label{mainresults}

Assuming a Poisson prior for $k$, RCL (Theorem 4.2) obtain a near-optimal convergence rate for $l(f,f_o)$.
In Corollary \ref{cor1} below, we show that the optimal rate for $l$ implies that we have at least a suboptimal rate for $|d-d_o|$. Whether this can be improved to the optimal rate
critically depends on the prior on $k$. By our first main result the answer is positive under \textbf{prior A}.
The proof is given in section \ref{sec4}.
\begin{thm} \label{BVMtheorem}
Under \textbf{prior A}, the posterior distribution has the asymptotic expansion
\begin{eqnarray} \label{post:norma:th1}
\Pi\left[ \sqrt{\frac{n r_{k_n}}{2}} (d-d_o- b_n(d_o)) \leq z |X\right]  = \Phi(z) + o_{P_o}(1),
\end{eqnarray}
where, for $r_{k_n} = \sum_{j\geq k_n+1} \eta_j^{2}$ and some small enough $\delta>0$,
\begin{equation*}
 b_n(d_o) =  \frac{1}{r_{k_n}} \sum_{j=k_n+1}^\infty \eta_j \ty_{o,j}  + Y_n + o( n^{-1/2-\delta} k_n^{1/2} ), \quad Y_n = \frac{ \sqrt{2}}{ \sqrt{nr_{k_n}} } Z_n,
\end{equation*}
$Z_n$ being a sequence of random variables converging weakly to a Gaussian variable with mean zero and variance 1.
\end{thm} \noindent
\begin{cor} \label{cor2}
Under \textbf{prior A}, the convergence rate for $d$ is $\dn = (n/\log n)^{-\frac{2\be - 1}{4\be}}$, i.e.
\begin{eqnarray*}
\lim_{n \rightarrow \infty} E_0^n \left[\Pi(d : |d-d_o| >  \dn | \xn)\right]  = 0. 
\end{eqnarray*} \noindent
\end{cor}
Equation \eqref{post:norma:th1} is a Bernstein-von Mises type of result: the posterior distribution is asymptotically normal, centered at a  point $d_o + b_n(d_o)$, whose distribution is normal with mean $d_o$ and variance $2 / (n r_{k_n})$. The expressions for the posterior mean and variance give more insight in how the prior for $k$ affects the posterior rate for $d$.
The standard deviation of the limiting normal distribution \eqref{post:norma:th1} is $\sqrt{2/(nr_{k_n})}=O(n^{-\frac{2\be-1}{4\be}} (\log n)^{\frac{1}{4\be}})$ 
and $b_n(d_o)$ equals
\begin{equation*}
\frac{1}{r_{k_n}} \sum_{j=k_n+1}^\infty \eta_j \ty_{o,j}  + O_{P_o}(k_n^{\ft} n^{\ft}) + o( n^{-1/2-\delta_1} k_n^{1/2} ). 
\end{equation*}
From the definition of $\eta_j$, $k_n$ and $r_{k_n}$ and the assumption on $\ty_o$, it follows that
\begin{equation} \label{bnBound1}
\frac{1}{r_{k_n}} \left| \sum_{j=k_n+1}^\infty \eta_j \ty_{o,j} \right| \leq \frac{1}{r_{k_n}}  \sqrt{\sum_{l>k_n} \ty_{o,l}^2 j^{2\be}} \sqrt{\sum_{l>k_n} j^{-2\be-2}} = o(k_n^{-\be+\ft}).
\end{equation}
See also (1.9) in the supplement.
Hence, when the constant $k_A$ in \eqref{kpktDef1} is small enough,
\begin{equation}\label{bnBound2}
|b_n(d_o)|\leq \dn,
\end{equation}
and we obtain the $\dn$-rate of Corollary \ref{cor2}. For smaller $k$, the standard deviation is smaller but the bias $b_n(d_o)$ is larger. In Theorem \ref{suboptimalTheorem} below it is shown that this indeed leads to a suboptimal rate.

An important consequence of the BVM-result is that posterior credible regions for $d$ (HPD or equal-tails for instance) will also be asymptotic frequentist confidence regions. Consider for instance one-sided credible intervals for $d$ defined by
$P^\pi(d \leq z_n(\alpha) |\xn) =\alpha$, so that $z_n(\alpha) $ is the $\alpha$-th quantile of the posterior distribution of $d$. Equation (\ref{post:norma:th1}) in Theorem \ref{BVMtheorem} then implies that
\[z_n(\alpha) = d_o + b_n(d_o)  + \sqrt{\frac{2k_n}{n}}\Phi^{-1}(\alpha)(1 + \op(1)).\]
As soon as $\sum_{j\geq k_n} j^{2\be }\theta_{o,j}^2   = o((\log n)^{-1})$, we have that \[z_n(\alpha) = d_o +  \sqrt{2/(nr_{k_n})} Z_n + \sqrt{2/(nr_{k_n})}\Phi^{-1}(\alpha)(1 + \op(1))\] and
\[P_o^n \left(  d_o \leq z_n(\alpha) \right) = P\left(Z_n  \leq \Phi^{-1}(\alpha) (1+ o(1))\right) =  \alpha+ o(1).\]
Similar computations can be made on equal - tail credible intervals or HPD regions for $d$.

Note that in this paper we assume that the smoothness $\be $ of $f_o$ is greater than 1 instead of $1/2$, as is required in \citet{ms1}. This condition is used throughout the proof. Actually had we only assumed that $\be >3/2$, the proof of Theorem \ref{BVMtheorem} would have been greatly simplified as many technicalities in the paper come from controlling terms when $1<\be \leq 3/2$. We do not believe that it is possible to weaken this constraint to $\be >1/2$ in our setup.

Our second main result states that if $k$ is increasing at a slower rate than $k_n$, the posterior on $d$ concentrates at a suboptimal rate. The proof is given in section \ref{sec5}.
\begin{thm} \label{suboptimalTheorem}
Given $\be>5/2$, there exists $\ty_o \in \Theta(\be,L_o)$ and a constant $k_v>0$ such that under prior B and C defined above,
\begin{equation*} 
\Pi(|d - d_o| > k_v w_n  (\log n)^{-1}| \xn) \overset{P_o}{\rightarrow} 1.
\end{equation*} \noindent
with $w_n=  C_w (n/\log n)^{-\frac{2\be-1}{4\be + 2}}$ and $C_w=C_1 (L+L_o)^{\frac{1}{4\be}} l_0^{\frac{2\be-1}{2\be}}$. %
\end{thm}
 The constant $C_w$ comes from the suboptimal rate for $|d-d_o|$ derived in Corollary \ref{cor1}.
Theorem \ref{suboptimalTheorem} is proved by considering the vector $\ty_o$ defined by $\ty_{o,j} = c_0 j^{-(\be+\ft)} (\log j)^{-1}$, for $j \geq 2$. This vector is close to the boundary of the Sobolev-ball $\Theta(\be,L_o)$, in the sense that for all $\be' > \be$,
$ \sum_j j^{2\be'} \ty_{o,j}^2 = + \infty$. The proof consists in showing that conditionally on $k$, the posterior distribution is asymptotically normal as in (\ref{post:norma:th1}), with $k$ replacing $k_n$, and that the posterior distribution concentrates on values of $k$ smaller than $O(n^{1/(2\be + 1)})$, so that the bias $b_n(d_o)$ becomes of order $w_n(\log n)^{-1}$. The constraint $\be > 5/2$ is used to simplify the computations and is not sharp.

It is interesting to note that similar to the frequentist approach, a key issue is a  bias-variance trade-off, which is optimized when $k \sim n^{1/(2\be)}$. This choice of $k$ depends on the smoothness parameter $\be$, and since it is not of the same order as the \textit{optimal } values of $k$ for the loss $l(f,f')$ on the spectral densities, the adaptive (near) minimax Bayesian nonparametric procedure proposed in \citet{kruijer:rousseau:11} does not lead to optimal posterior concentration rate for $d$. While it is quite natural to obtain an adaptive (nearly) minimax Bayesian procedure under the loss $l(.,.)$ by choosing a random $k$, obtaining an adaptive minimax procedure for $d$ remains an open problem. This dichotomy is found in other semi-parametric Bayesian problems, see for instance \citet{arbel:10} in the case of the white noise model or \citet{riro10} for BVM properties.

\section{Decomposing the posterior for $d$} \label{sec3}
To prove Theorems \ref{BVMtheorem} and \ref{suboptimalTheorem} we need to take a closer look at \eqref{marginalPosterior},
to understand how the integration over $\Theta_k$ affects the posterior for $d$. We develop
$\ty \rightarrow l_n(d,k,\ty)$ in a point $\bd$ defined below and decompose the likelihood as
\begin{equation*}
\exp\{l_n(d,k,\ty)\} = \exp\{l_n(d,k)\} \exp\{l_n(d,k,\ty)-l_n(d,k)\},
\end{equation*}
where $l_n(d,k)$ is short-hand notation for $l_n(d,k,\bd)$.
Define
\begin{equation} \label{InDef}
I_n(d,k) = \int_{\Theta_k} e^{l_n(d,k,\ty)-l_n(d,k)} d\pi_{\ty|k}(\ty),
\end{equation}
where $\Theta_k$ is the generic notation for $\tlt$ under \textbf{prior A} and $\tlo$ for priors B and C.
The posterior for $d$ given in \eqref{marginalPosterior} can be written as
\begin{equation} \label{marginalPosterior3}
\Pi(d \in D|\xn) =
\frac{\sum_{k=0}^{\infty} \pi_k(k) \int_{D} e^{l_n(d,k)- l_n(d_o,k)} I_n(d,k) d\pi_d(d)}
{\sum_{k=0}^{\infty} \pi_k(k) \int_{-\ft+t}^{\ft-t} e^{l_n(d,k)- l_n(d_o,k)}
I_n(d,k) d\pi_d(d)}.
\end{equation}

The factor $\exp\{l_n(d,k)-l_n(d_o,k)\}$ is independent of $\ty$, and will under certain conditions dominate the marginal likelihood. In section \ref{sec3d} we give a Taylor-approximation which, for given $k$, allows for a normal approximation to the marginal posterior. However, to obtain the convergence rates in Theorems \ref{BVMtheorem} and \ref{suboptimalTheorem},
it also needs to be shown
that the integrals $I_n(d,k)$ with respect to $\ty$ do not vary too much with $d$. This is the most difficult part of the proof of Theorem \ref{BVMtheorem} and the argument is presented in section \ref{sec3c}. Since Theorem \ref{suboptimalTheorem} is essentially a counter-example and it is not aimed to be as general as Theorem \ref{BVMtheorem}, as far as the range of $\be$ is concerned, we can restrict attention to larger $\be$'s, i.e. $\be > 5/2$, for which controlling $I_n(d,k)$ is much easier.

\subsection{Preliminaries} \label{prelim}

First we define the point $\bd$ in which we develop $\ty \rightarrow l_n(d,k,\ty)$.
Since the function $\log(2-2\cos(x))$ has Fourier coefficients against $\cos j x$, $ j\in \N$ equal to $0,2,\frac{2}{2},\frac{2}{3},\ldots$,
FEXP-spectral densities can be written as
\begin{equation*} \label{fexpDef}
|1-e^{i x}|^{-2 d} \exp\left\{ \sum_{j=0}^\infty \ty_j \cos(j x) \right\}
= \exp\left\{ \sum_{j=0}^{\infty} (\tj + d \eta_j) \cos (j x) \right\}.
\end{equation*} \noindent
Given $f=f_{d,k,\ty}$ and $f'=f_{d',k',\ty'}$ we can therefore express the norm
$l(f,f')$ in terms of $(\ty-\ty')$ and $(d-d')$:
\begin{equation} \label{fexpl}
l(f,f') = \ft \sum_{j=0}^{\infty} ((\ty_j-\ty_j') + \eta_j(d-d'))^2,
\end{equation} \noindent
where $\ty_j$ and $\ty_j'$ are understood to be zero when $j$ is larger than $k$ respectively $k'$.
Equation \eqref{fexpl} implies that for given $d$ and $k$,
$l(f_o,f_{d,k,\ty})$ is minimized by
\begin{equation*} \label{projDef}
\bd := 
\mbox{argmin}_{\ty \in \R^{k+1}} \sum_{j=0}^\infty (\ty_j -\ty_{o,j} + (d-d_o)\eta_j)^2  = \ty_{o[k]} + (d_o-d)\eta_{[k]}.
\end{equation*} \noindent
In particular, $\ty=\ty_{o[k]}$ minimizes $l(f_o,f_{d,k,\ty})$ only when $d=d_o$; when $d\neq d_o$ we need to add
$(d_o-d)\eta_{[k]}$. The following lemma shows that an upper bound on  $l(f_o,f_{d,k,\ty})$ leads to upper bounds on
 $|d-d_o|$ and $\|\ty-\ty_o\|$.
\begin{lem} \label{lem1}
Suppose that $\ty \in \Theta_k(\gamma,L)$ and $\ty_o \in \Theta_k(\be,L_o)$, where $\gamma \leq \be$.
Also suppose that for a sequence $\an \rightarrow 0$, $l(f_o,f_{d,k,\ty}) \leq \an^2$ for all $n$.
Then there are universal constants $C_1, C_2>0$ such that for all $n$,
\begin{eqnarray*} 
|d-d_o| \leq C_1 (L+L_o)^{\frac{1}{4\gamma}} \an^{\frac{2\gamma-1}{2\gamma}}, \qquad \|\ty-\ty_o\| \leq C_2 (L+L_o)^{\frac{1}{4\gamma}} \an^{\frac{2\gamma-1}{2\gamma}}.
\end{eqnarray*} \noindent
\end{lem} \noindent
%
%
\begin{proof}
For all $(d,k,\ty)$ such that $l(f_{d,k,\ty},f_o) \leq \an$, we have, using \eqref{fexpl},
\begin{equation*} 
\begin{split}
2 \an^2  & \geq  2 l(f_{d,k,\ty},f_o) = 2 (\ty_{o,0} - \ty_0)^2 + \sum_{j \geq 1} \left((\ty_{o,j} - \tj) +
\eta_j (d_o-d) \right)^2 \\& \geq \sum_{j \geq 1} (\ty_{o,j} - \tj)^2  + (d-d_o)^2 \sum_{j \geq 1} \eta_j^2
- 2 |d-d_o| \sqrt{\sum_{j \geq 1} \eta_j^2} \sqrt{\sum_{j \geq 1} (\ty_{o,j} - \tj)^2} \\
&= \left(\|\ty - \ty_o\| -|d-d_o| \|\eta\|\right)^2.
\end{split}
\end{equation*} \noindent
The inequalities remain true if we replace all sums over $j\geq 1$ by sums over $j\geq m_n$, for any nondecreasing sequence $m_n$.
Since $\|(\eta_j 1_{j>m_n})_{j\geq 1}\|^2$  is of order $m_n^{-1}$ and $\|(\ty-\ty_o\|_j1_{j>m_n})_{j\geq 1} \|^2 \leq m_n^{-2\gamma} \sum_{j > m_n} (1+j)^{2\be} (\ty_j - \ty_{o,j})^2 < 2(L+L_o) m_n^{-2\gamma}$,
setting $m_n =\an^{-\frac{1}{\gamma}}$ gives the desired rate for $|d-d_o|$ as well as for $\|\ty-\ty_o\|$. 
\end{proof} \noindent
The convergence rate for $l(f_o,f_{d,k,\ty})$ required in Lemma \ref{lem1} can be found in \cite{kruijer:rousseau:11}. For easy reference we restate it here. Compared to a similar result in RCL, the $\log n$ factor is improved.
\begin{lem} \label{rateLemma1}
Under \textbf{prior A}, there exists a constant $l_0$ depending only on $L_o$ and $k_A$ (and not on $L$) such that
\begin{eqnarray*}
\Pi((d,k,\ty) : l(f_{d,k,\ty},f_o) \geq l_0^2 \dn^2 | \xn) \overset{P_o}{\rightarrow} 0,
\end{eqnarray*} \noindent
where $\dn = (n/\log n)^{-\frac{2\be-1}{4\be}}$.
Under \textbf{priors B} and {\bf C}, this statement holds with $\en = (n/\log n)^{-\frac{\be}{2\be + 1}}$ replacing $\dn$.
\end{lem} \noindent
In the proof of Theorem \ref{BVMtheorem}  (resp. \ref{suboptimalTheorem}), this
result allows us to restrict attention to the set of spectral densities $f$ such that $l(f,f_o) \leq l_0^2\dn^2 $ (resp. $l_0^2 \epsilon_n^2$).
In addition, by combination with Lemma \ref{lem1} we can now
deduce bounds on $|d-d_o|$ and $\|\ty-\bd\|$.
These bounds, although suboptimal,  will be important in the sequel for obtaining the near-optimal rate in Theorem \ref{BVMtheorem}.
\begin{cor} \label{cor1}
Under the result of Lemma \ref{rateLemma1} and \textbf{prior A}, we can apply Lemma \ref{lem1} with $\an^2= l_0^2\dn^2$ and $\gamma=\be-\ft$, and obtain
\begin{equation*}\label{vnconvergence3}
\Pi_d(d : |d-d_o| \geq \vb | \xn) \overset{P_o}{\rightarrow} 0, \quad
\Pi(\|\ty-\bd\| \geq 2l_0\dn | \xn) \overset{P_o}{\rightarrow} 0,
\end{equation*} \noindent
where $\vb = C_1 (L+L_o)^{\frac{1}{4\be-2}} l_0^{\frac{2\be-2}{2\be-1}} (n / \log n)^{-\frac{\be-1}{2\be}}$.
Under \textbf{priors B} and {\bf C} we have $\gamma=\be$; the rate for $|d-d_o|$ is then
$w_n = C_w (n / \log n)^{-\frac{2\be-1}{4\be+2}}$ and the rate for $\|\ty-\bd\|$ is $2 l_0 \en$.
The constant $C_w= C_1 (L+L_o)^{\frac{1}{4\be}} l_0^{\frac{2\be-1}{2\be}}$ is as in Theorem \ref{suboptimalTheorem}.
\end{cor} \noindent
\begin{proof}
The rate for $|d-d_o|$  follows directly from Lemma \ref{lem1}. To obtain the rate for $\|\ty-\bd\|$,
let $\an$ denote either $l_0 \dn$
(the rate for $l(f_o,f)$ under prior A) or $l_0\en$ (the rate under priors B and C).
Although Lemma \ref{lem1} suggests that the Euclidean distance from $\ty_o$ to $\ty$
(contained in $\tlo$ or $\tlt$) may be larger
than $\an$, the distance from $\ty$ to $\bd$ is certainly of order $\an$.
To see this, note that Lemma \ref{rateLemma1} implies the existence of $d,k,\ty$ in the model with
$l(f_o, f_{d,k,\ty}) \leq \an^2$.  From the definition of $\bd$ it follows that $l(f_o,f_{d,k,\bd}) \leq \an^2$. The triangle inequality gives $\|\ty-\bd\|^2 = l(f_{d,k,\ty},f_{d,k,\bd}) \leq 4 \an^2$.
\end{proof}

The rates $\vb$ and $w_n$ obtained in Corollary \ref{cor1} are clearly suboptimal; their importance however lies in the fact that they narrow down the set for which we need to prove Theorems \ref{BVMtheorem} and \ref{suboptimalTheorem}.
To prove Theorem \ref{suboptimalTheorem} for example it suffices to show that the posterior mass
on $k_v w_n (\log n)^{-1}< |d-d_o| < w_n$ tends to zero. Note that the lower and the upper bound differ only by a factor $(\log n)$. Hence under priors B and C, the combination of Corollary \ref{cor1} and Theorem \ref{suboptimalTheorem} characterizes the posterior concentration rate (up to a $\log n$ term) for the given $\ty_o$. Another consequence of Corollary \ref{cor1} is that we may neglect the posterior mass on all
$(d,k,\ty)$ for which $\|\ty -\bd\|$ is larger than $2l_0\dn$ (under prior A) or $2l_0\en$ (under priors B and C).

We conclude this section with a result on $\bd$ and $\tlo$.
In the definition of $\bd$ we minimize over $\R^{k+1}$, whereas the support of priors A-C is the Sobolev ball
$\tlo$ or $\tlt$. Under the assumptions of Theorems \ref{BVMtheorem} and \ref{suboptimalTheorem} however,
$\bd$ is contained in $\tlt$ respectively $\tlo$. Also the $l_2$-ball of radius $2l_0\dn$ (or $2l_0\en$) is contained in these Sobolev-balls.
\begin{lem} \label{bdLemma}
Under the assumptions of Theorem \ref{BVMtheorem}, $B_{k}(\bd,2l_0\dn)$ is contained in $\Theta_{k}(\be-\frac{1}{2},L)$
for all $d\in [d_o-\vb,d_o+\vb]$, if $L$ is large enough. In particular, $\bd \in \Theta_{k}(\be-\frac{1}{2},L)$. Similarly, under the assumptions of Theorem \ref{suboptimalTheorem}, $B_{k}(\bd,2l_0\te) \subset \tlo$, for all $d\in [d_o-w_n,d_o+w_n]$.
\end{lem} \noindent
\begin{proof}
Since the constant $l_0$ is independent of $L$, $\ty \in B_{k}(\bd,2l_0\dn)$ implies that for $n$ large enough ,
\begin{equation*} 
\begin{split}
& \sum_{j=0}^{k} \ty_j^2 (j+1)^{2\be-1} \leq 2\sum_{j=0}^{k} (\ty-\bd)_j^2 (j+1)^{2\be-1} + 2 \sum_{j=0}^{k} (\bd)_j^2 (j+1)^{2\be-1} \\
&\quad \leq 8 \delta^2(L_o) (n/\log n)^{\frac{2\be-1}{2\be}} (k_n+1)^{2\be-1} + 4 \sum_{j=0}^{k_n} \ty_{o,j}^2 (j+1)^{2\be-1} \\ &\qquad + 16 (d-d_o)^2 \sum_{j=1}^{k_n} j^{2\be-3}.
\end{split}
\end{equation*}
The first two terms on the right only depend on $L_o$, and are smaller than $L/4$ when $L$ is chosen sufficiently large.
Because $\vb = C_1 (L+L_o)^{\frac{1}{4\be-2}} l_0^{\frac{2\be-2}{2\be-1}} (n / \log n)^{-\frac{\be-1}{2\be}}$, the last
term in the preceding display is at most 
\begin{equation*} \label{inclusion3}
C_1^2 (L+L_o)^{\frac{1}{2\be-1}} l_0^{\frac{4\be-4}{2\be-1}} (n / \log n)^{-\frac{\be-1}{\be}} k_A^{2\be-2} (n/\log n)^{\frac{\be-1}{\be}},
\end{equation*}
which, since $\be>1$, is smaller than $L/2$ when $L$ is large enough. We conclude that $B_{k}(\bd,2l_0\dn)$
is contained in $\Theta_{k}(\be-\frac{1}{2},L)$ provided $L$ is chosen sufficiently large.
The second statement can be proved similarly.
\end{proof}

\subsection{A Taylor approximation for $l_n(d,k)$} \label{sec3d}
Provided that the integrals $I_n(d,k)$ have negligible impact on the posterior for $d$,
the conditional distribution of $d$ given $k$ will only depend on $\exp\{l_n(d,k)-l_n(d_o,k)\}$.
Let $l_n^{(1)}(d,k)$, $l_n^{(2)}(d,k)$ denote the first two derivatives of the map $d \mapsto l_n(d,k)$.
There exists a $\bar d$ between $d$ and $d_o$ such that
\begin{eqnarray} \label{D1development}
l_n(d,k) &=&  l_n(d_o,k) + (d-d_o)l_n^{(1)}(d_o,k) + \frac{(d-d_o)^2}{2} l_n^{(2)}(\bar{d},k).
\end{eqnarray}
Defining
\begin{equation*}
b_n(d) = - \frac{l_n^{(1)}(d_o,k)}{l_n^{(2)}(d,k)},
\end{equation*}
which is the $b_n$ used in Theorem \ref{BVMtheorem}, we can rewrite \eqref{D1development} as
\begin{equation} \label{normalAppr}
\begin{split}
l_n(d,k) - l_n(d_o,k) &=  - \ft \frac{(l_n^{(1)}(d_o,k))^2 }{l_n^{(2)}(\bar d,k)}   + \ft l_n^{(2)}(\bar d,k) \left(d-d_o - b_n(\bar d)\right)^2.
\end{split}
\end{equation} \noindent
Note that each derivative $l_n^{(i)}(d,k)$, $i = 1, 2$, can be decomposed into a centered quadratic form denoted $\mathcal S(l_n^{(i)}(d,k))$ and a deterministic term
$\mathcal D (l_n^{(i)}(d,k))$.
In the following lemma we give expressions for $l_n^{(1)}(d_o,k)$, $l_n^{(2)}(d,k)$ and $b_n$, making explicit their dependence on $k$ and $\ty_o$.
Since $\kp \leq k_n$ and $w_n < \vb$ (see Corollary \ref{cor1}) the result is valid for all priors under consideration.
The proof is given in appendix \ref{appD}.
%
%
\begin{lem} \label{delta1terms}
Given $\be > 1$, let $\ty_o \in \Theta(\be,L_o)$. If $k \leq k_n$ and $|d-d_o|\leq \vb$, then there exists  $\delta_1>0$  such that
\begin{eqnarray*} 
l_n^{(1)}(d_o,k) &:=& \mathcal S(l_n^{(1)}(d_o,k)) + \mathcal D( l_n^{(1)}(d_o,k)) \\
 &=&   \mathcal S(l_n^{(1)}(d_o,k)) + \frac{n}{2}\sum_{j={k+1}}^\infty \ty_{o,j}\eta_j  + o(n^{\e}(  k^{-\be+3/2} + n^{-1/(2\be)})),  \nonumber \\ 
l_n^{(2)}(d,k)  &=& l_n^{(2)}(d_o,k)   \left( 1 + \frac{k^{1/2}}{ n^{1/2+\e}} + \frac{k^{-2\be+1+\e}}{n}   \right)
= - \ft n r_k \left(1+ \op(n^{-\delta_1}) \right), 
\end{eqnarray*} \noindent
where $ \mathcal S(l_n^{(1)}(d_o,k))$ is a   centered quadratic form with variance
\begin{equation*}
Var( \mathcal S(l_n^{(1)}(d_o,k)))  =  \frac{n}{2} \sum_{j>k} \eta_j^2( 1 +o(1)) = \frac{n r_k}{ 2}(1 + o(1)) =  O(n k^{-1}).
\end{equation*} \noindent
Consequently, 
\begin{equation} \label{Bdef}
\begin{split}
b_n(d) &=  - \frac{l_n^{(1)}(d_o,k)}{l_n^{(2)}(d,k)} = \frac{1}{r_k} \sum_{j={k+1}}^\infty \ty_{o,j}\eta_j  ( 1+\op(n^{-\delta}))  \\
& \qquad + \frac{2 \mathcal S(l_n^{(1)}(d_o,k))( 1+\op(n^{-\delta})) }{ nr_k} + \op(n^{\e-1} k^{-\be+ 5/2} + n^{\e -1}),
\end{split}
\end{equation}
with
 $$\frac{2 \mathcal S(l_n^{(1)}(d_o,k)) }{ nr_k}  = \mathbf{O_{P_o}}(n^{-\ft} k^{\ft}).$$

\end{lem}

\noindent
\begin{rem} \label{rem2}
Recall from \eqref{bnBound1} that  $r_k^{-1} \sum_{j={k+1}}^\infty \ty_{o,j}\eta_j$  is $O(k^{-\be + 1/2})$. The term $2 \mathcal S(l_n^{(1)}(d_o,k))/(nr_k)$ is $O_{P_o}(k^{-\be + 1/2})$  whenever $k \sim n^{1/(2\be)}$, which is the case under all priors under consideration.
\end{rem}

Substituting the above results on $l_n^{(1)}$, $l_n^{(2)}$ and $b_n$ in \eqref{normalAppr}, we can give the following informal argument leading to Theorems \ref{BVMtheorem} and Theorem \ref{suboptimalTheorem}.
If we consider $k$ to be fixed and $I_n(d,k)$ constant in $d$, then \eqref{normalAppr} implies that the posterior distribution for $d$ is asymptotically normal with mean $d_o + b_n(d_o)$ and variance of order $k/n$.


\subsection{Integration of the short memory parameter} \label{sec3c}

A key ingredient in the proofs of both Theorems \ref{BVMtheorem} and \ref{suboptimalTheorem} is the control of the integral $I_n(d,k)$ appearing in \eqref{marginalPosterior},
whose dependence on $d$ should be negligible with respect to $\exp\{l_n(d,k)-l_n(d_o,k)\}$. In Lemma \ref{lem:Idk:th1} below we prove this to be the case under the assumptions of Theorems \ref{BVMtheorem} and \ref{suboptimalTheorem}.
For the case of Theorem \ref{suboptimalTheorem} this is fairly simple: the conditional posterior distribution of $\ty$ given $(d,k)$ can be proved to be asymptotically Gaussian by a Laplace-approximation. For smaller $\be$ and larger $k$ the control is technically more demanding.
In both cases the proof is based on the following Taylor expansion of $l_n(d,k,\ty)$ around $\bd$:
\begin{equation} \label{taylor1}
l_n(d,k,\ty) - l_n(d,k) = \sum_{j=1}^J \frac{(\ty - \bd)^{(j)} \nabla^j l_n(d,k) }{ j ! } + R_{J+1,d}(\ty),
\end{equation} \noindent
where
\begin{equation*} 
(\ty - \bd)^{(j)} \nabla^j l_n(d,k) = \sum_{l_1,\ldots,l_j=0}^k (\ty - \bd)_{l_1}\ldots(\ty - \bd)_{l_j} \frac{\partial^j l_n(d,k,\bd) }{ \partial \ty_{l_1}\ldots\partial \ty_{l_j}},
\end{equation*} \noindent
\begin{equation} \label{Rexpression}
R_{J+1,d}(\ty) = \frac{ 1}{(J+1)!} \sum_{l_1,\ldots,l_{J+1}=0}^k (\ty - \bd)_{l_1}\ldots(\ty - \bd)_{l_{J+1}} \frac{\partial^{J+1} l_n(d,k,\tilde \theta) }{ \partial \ty_{l_1}\ldots\partial \ty_{l_{J+1}}}.
\end{equation} \noindent
%
%
The above expressions are used to derive the following lemma, which gives control of the term $I_n(d,k)$.
\begin{lem} \label{lem:Idk:th1}
Under the conditions of Theorem \ref{BVMtheorem}, the integral $I_n(d,k)$ defined in \eqref{InDef} equals
\begin{equation*}
I_n(d_o,k) \exp\left\{\op(1) + \op\left(\frac{|d-d_o| n^{\ft-\delta_2}}{ \sqrt{k}}\right) + \op\left((d-d_o)^2 \frac{ n^{1-\delta_2} }{k}\right)\right\},
\end{equation*}
for some $\delta_2 > 0$. Under the conditions of Theorem \ref{suboptimalTheorem},
\begin{equation*}
I_n(d,k)  =  I_n(d_o,k) \exp\left\{\op(1)\right\}.
\end{equation*}
\end{lem}
The proof is given in Appendix \ref{lem:Idk:th1:Details}, and relies on the expressions for the derivatives $\nabla^j l_n$ given in Appendix \ref{app:deriv:theta}.
Lemma \ref{lem:Idk:th1} should be seen in relation to Lemma \ref{delta1terms} and the expressions for $\Pi(d|X)$ and $l_n(d,k)-l_n(d_o,k)$ in equations \eqref{marginalPosterior3} and \eqref{D1development}.
Lemma \ref{lem:Idk:th1} then shows that the dependence on the integrals $I_n(d,k)$ on $d$ is asymptotically negligible with respect to $l_n(d,k)-l_n(d_o,k)$. This is made rigorous in the following section. 

\section{Proof of Theorem \ref{BVMtheorem}} \label{sec4}
By Lemma \ref{rateLemma1} we may assume posterior convergence of $l(f_o,f_{d,k,\ty})$ at rate $l_0^2\dn^2$, and,
by Corollary \ref{cor1}, also convergence of $|d-d_o|$ at rate $\vb$. By Lemma \ref{bdLemma}, we may restrict the integration over $\ty$ to $\bn$. Let $\Gamma_n(z) = \{d: \sqrt{\frac{n r_k}{2}}(d-d_o-b_n(d_o)) \leq z\}$. Under \textbf{prior A}, it suffices to show that for $k=k_n$,
\begin{equation} \label{posterior}
\begin{split}
\frac{N_n}{D_n} & : =  \frac{\int_{\Gamma_n(z)}
e^{l_n(d,k) - l_n(d_o,k)} \int_{\bn}  
e^{l_n(d,k,\ty) - l_n(d,k)} d\pi_{\ty|k}(\ty) d\pi_d(d)}
{\int_{|d-d_o| < \vb} e^{l_n(d,k) - l_n(d_o,k)} \int_{\bn} 
e^{l_n(d,k,\ty) - l_n(d,k)} d\pi_{\ty|k}(\ty) d\pi_d(d)} \\
&=  \frac{\int_{\Gamma_n(z)} \exp\{l_n(d,k) - l_n(d_o,k) + \log I_n(d,k)\} d\pi_d(d)}
{\int_{|d-d_o| < \vb} \exp\{l_n(d,k) - l_n(d_o,k) + \log I_n(d,k)\} d\pi_d(d)} = \Phi(z) + \op(1).
\end{split}
\end{equation} \noindent
Using the results for $l_n(d,k)-l_n(d_o,k)$ and $I_n(d,k)$ given by Lemmas \ref{delta1terms} and \ref{lem:Idk:th1}, we show that for $A_n \subset \R^n$ defined below such that $P_o^n(A_n) \rightarrow 1$,
\begin{equation} \label{bvm_upper_bound}
\frac{N_n}{D_n} \leq \Phi(z) + o(1), \quad \frac{N_n}{D_n} \geq \Phi(z) + o(1), \quad \forall X \in A_n.
\end{equation}
Since $P_o^n(A_n) \rightarrow 1$ this implies the last equality in \eqref{posterior}.

Note that Lemmas \ref{delta1terms} and \ref{lem:Idk:th1} also hold for all $\delta_1' < \delta_1$ and $\delta_2' < \delta_2$. In the remainder of the proof, let $0< \delta\leq \min(\delta_1,\delta_2)$. For notational simplicity, let $\mathcal D = \mathcal D(l_n^{(1)}(d_o,k)$, the deterministic part of $l_n^{(1)}(d_o,k)$.
For a sufficiently large constant $C_1$ and arbitrary $\e_1>0$,
let $A_n$ be the set of $X \in \R^n$ such that
\begin{equation*} 
\left.
\begin{array}{r}
\left|\log I_n(d,k) - \log I_n(d_o,k)\right| \leq \e_1 + (d-d_o)^2 k^{-1} n^{1-\delta_{}} + |d-d_o| k^{-\ft} n^{\ft-\delta_{}} \\
  \left| l_n^{(1)}(d_o,k) - \mathcal D \right| \leq C_1n^{\ft} k^{-\ft}\sqrt{\log n},
  \qquad \left| l_n^{(2)}(d,k) + \ft n r_k \right| \leq n^{1-\delta_{}} k^{-1}
\end{array} 
\right\}
\end{equation*}
for all $|d-d_o|\leq \vb$.
Since $k=k_n$ and $\be> 1$, 
Lemmas \ref{delta1terms} and \ref{lem:Idk:th1} imply  that $P_o^n(A_n^c) \rightarrow 0$.
We prove the first inequality in \eqref{bvm_upper_bound};
the second one can be obtained in the same  way.
Using \eqref{D1development} and the definition of $A_n$, it follows that for all $X \in A_n$,
\begin{equation} \label{upperBound}
\begin{split}
& l_n(d,k) - l_n(d_o,k) + \log I_n(d,k) - \log I_n(d_o,k)
 \leq \e_1 + (d-d_o)^2 \frac{ n^{1-\delta_{}} }{ k} \\
& \qquad + |d-d_o| \frac{ n^{\ft-\delta_{}} }{ k^{\ft} } + (d-d_o) l_n^{(1)}(d_o,k) -\frac{n r_k}{4}(d-d_o)^2 (1  - n^{-\delta_{}}) \\
& \leq 2\e_1 - \frac{n r_k}{4} \left(1-\frac{2}{n^{\delta_{}}}\right) \left(d-d_o -\frac{2 l_n^{(1)}(d_o,k)}{\left(1-\frac{2}{n^{\delta_{}}}\right)n r_k} \right)^2 + \left|d-d_o\right|  \frac{ n^{\ft-\delta_{}} }{ k^{\ft} } +\frac{(l_n^{(1)}(d_o,k))^2}{\left(1-\frac{2}{n^{\delta_{}}}\right)n r_k} \\
& \leq 3\e_1 - \frac{n r_k}{4} \left(1-\frac{2}{n^{\delta_{}}}\right) \left(d-d_o -\frac{b_n(d_o,k)}{1-\frac{2}{n^{\delta_{}}}} \right)^2   \\
& \qquad + \left|d-d_o -\frac{b_n(d_o,k)}{1-\frac{2}{n^{\delta_{}}}} \right| \frac{ n^{\ft-\delta_{}}}{k^{\ft}} + \frac{(l_n^{(1)}(d_o,k))^2}{\left(1-\frac{2}{n^{\delta_{}}}\right)n r_k},
\end{split}
\end{equation}
 The third inequality follows from \eqref{bnBound2} and Remark \ref{rem2}, by which $b_n(d_o)=O(k^{-\beta+\ft})=O(\dn)$. This implies that $|b_n(d_o)| k^{-\ft} n^{\ft-\delta_{}}  < \e_1$, again for large enough $n$.
Similar to the preceding display, we have the lower-bound
\begin{equation} \label{lowerBound}
\begin{split}
& l_n(d,k) - l_n(d_o,k) + \log I_n(d,k) - \log I_n(d_o,k) \\
& \quad \geq -3\e_1 - \frac{n r_k}{4} (1+2n^{-\delta_{}}) \left(d-d_o -\frac{b_n(d_o,k)}{(1+2n^{-\delta_{}})} \right)^2 \\
& \qquad  - \left|d-d_o-\frac{b_n(d_o,k)}{(1+2n^{-\delta_{}})} \right| k^{-\ft} n^{\ft-\delta_{}} +\frac{(l_n^{(1)}(d_o,k))^2}{(1+2n^{-\delta_{}})nr_k}.
\end{split}
\end{equation}
Note that
\begin{equation}\label{constantTerm}
\exp\left\{\frac{(l_n^{(1)}(d_o,k))^2}{(1-2n^{-\delta_{}})nr_k} -\frac{(l_n^{(1)}(d_o,k))^2}{(1+2n^{-\delta_{}})nr_k} \right\}=
\exp\{o(1)\},
\end{equation}
which follows from the expression for $l_n^{(1)}(d_o,k)$ in Lemma \ref{delta1terms}, the definition of $A_n$ and the assumption that $X \in A_n$.
Therefore,  substituting  \eqref{upperBound} in $N_n$ and \eqref{lowerBound} in $D_n$,  the terms $\frac{(l_n^{(1)}(d_o,k))^2}{4nr_k}$  cancel out and by \eqref{constantTerm} we can neglect the difference between $\frac{(l_n^{(1)}(d_o,k))^2}{(1\pm 2n^{-\delta_{}})nr_k}$ and $\frac{(l_n^{(1)}(d_o,k))^2}{nr_k}$.

To conclude the proof that $N_n/D_n \leq \Phi(z) + o(1)$ for each $X \in A_n$,
we make the change of variables
\begin{equation*}
\begin{split}
u &= \sqrt{\frac{n r_k}{2}(1\pm 2 n^{-\delta_{}})} \left(d-d_o -  \frac{b_n(d_o)}{ 1 \pm 2 n^{-\delta_{}}} \right),
\end{split}
\end{equation*}
where we take $+$ in the lower bound for $D_n$ and $-$ in the upper-bound for $N_n$.
Using once more that $b_n(d_o)=O(\dn)$, we find that for large enough $n$, $|u| \leq \frac{\vb}{4} \sqrt{n r_k}$ implies $|d-d_o| \leq \vb$. Hence we may integrate over $|u| \leq \frac{\vb}{4} \sqrt{n r_k}$ in the lower-bound for $D_n$.
In the upper-bound for $N_n$ we may integrate over $u \leq z + \e_1$.

Combining \eqref{upperBound}-\eqref{constantTerm}, it follows that for all $\epsilon_1$ and all $X \in A_n$, 
\begin{equation*} \label{posterior3}
\begin{split}
&\frac{N_n}{D_n} \leq e^{7 \e_1} \left(\frac{1 + 2n^{-\delta_{}}}{1 - 2n^{-\delta_{}}}\right)^{\ft} \frac{\int_{u<z+\e_1} \exp\{-\ft u^2 + C n^{-\delta_{}} |u|\} du }{\int_{|u| \leq \frac{\vb}{4} \sqrt{n r_k}}\exp\{-\ft u^2 - C n^{-\delta_{}}|u|\}du } \\
 & \leq e^{8\e_1 } \frac{\int_{u<z+\e_1} \exp\{-\ft u^2 + C n^{-\delta_{}} |u|\} du }{\int_{|u| \leq \frac{\vb}{4} \sqrt{n r_k}}\exp\{-\ft u^2 - C n^{-\delta_{}}|u|\}du } \rightarrow \Phi(z+\e_1)e^{8 \e_1 }.
\end{split}
\end{equation*}
Similarly we prove that for all $\epsilon_1$, $N_n /D_n \geq \Phi(z-\e_1)e^{-8 \e_1 }$, when $n$ is large enough,
which terminates the proof of Theorem \ref{BVMtheorem}.

\section{Proof of Theorem \ref{suboptimalTheorem}} \label{sec5}
Let $\be > 5/2$ and $\ty_{o,j} = c_0 j^{-(\be+\ft)} (\log j)^{-1}$.
When the constant $c_0$ is chosen small enough, $\ty_o \in \Theta(\be,L_o)$.
In view of Corollary \ref{cor1}, the posterior mass on the events $\{(d,k,\ty) : \|\ty - \bd\| \geq 2l_0 \te\}$ and $\{(d,k,\ty) : |d-d_o| \geq w_n\}$
tends to zero in probability, and may be neglected. Moreover Lemma \ref{lem1} implies that with posterior probability going to 1, $\| \ty - \ty_0\| \lesssim (n / \log n)^{-(\be -1/2)/(2 \be + 1)}$.
However, within the $(k+1)$-dimensional FEXP-model, $\|\ty - \ty_o\|$ is minimized by setting $\ty_j=\ty_{o,j}$ ($j=0,\ldots,k$), and for this choice of $\ty$ we have
$$\|\ty - \ty_o\|^2 = \sum_{l>k} \ty_{o,l}^2 \gtrsim k^{-2\be} (\log k)^{-2}.$$
Consequently, the fact that $\| \ty - \ty_0\| \lesssim (n / \log n)^{-(\be -1/2)/(2 \be + 1)}$ implies that
$k > k_n'' := k_l (n / \log n)^{ (\be - 1/2)/ ( \be(2 \be + 1))} ( \log n)^{-1/\be}$, for some constant $k_l$. We conclude that
$$ \Pi \left( k \leq k_n''|\xn\right) = \op(1),$$
and we can restrict our attention to $k> k_n''$.

We decompose $\Pi_d(|d-d_o| \leq k_v w_n (\log n)^{-1}, k> k_n''|\xn)$ as
\begin{equation*} \label{posteriorKdecomp}
\begin{split}
& \sum_{m> k_n''} \Pi (|d-d_o| \leq k_vw_n(\log n)^{-1},k=m|\xn) \\
& \quad = \sum_{m> k_n''} \Pi(k=m|\xn) \Pi_m(|d-d_o| \leq k_vw_n(\log n)^{-1}|\xn),
\end{split}
\end{equation*} \noindent
where $\Pi_m(|d-d_o| \leq k_vw_n(\log n)^{-1}|\xn)$ is the posterior for $d$ within the FEXP-model of dimension $m+1$, i.e. $\Pi_m(|d-d_o| \leq k_vw_n(\log n)^{-1}|\xn) := \Pi(|d-d_o| \leq k_vw_n(\log n)^{-1}|k=m,\xn)$.

To prove Theorem \ref{suboptimalTheorem} it now suffices to show that
\begin{eqnarray}
& & \sum_{k_n'' \leq m\leq\kp} \Pi(k=m|\xn) = \Pi(k_n'' \leq k\leq\kp|\xn) \overset{P_o}{\rightarrow} 1, \label{priorK} \\
& & E_0^n \Pi_k(|d-d_o| \leq k_v w_n(\log n)^{-1}|\xn) \overset{P_o}{\rightarrow} 0, \qquad \forall k_n'' \leq k\leq \kp. \label{priorM}
\end{eqnarray} \noindent
The convergence in \eqref{priorK} is a by-product of Theorem 1 in \citet{kruijer:rousseau:11}. 
In the remainder we prove \eqref{priorM}. For every $k\leq\kp$ we can write, using the notation of \eqref{posterior},
\begin{equation} \label{posterior2}
\begin{split}
& \Pi_k(|d-d_o| < k_vw_n(\log n)^{-1} | \xn) \leq \frac{N_{n,k}}{D_{n,k}} \\ \quad &:=
\frac{\int_{|d-d_o| < k_vw_n(\log n)^{-1}} \exp\{l_n(d,k) - l_n(d_o,k) + \log I_n(d,k)\} d\pi_d(d)}
{\int_{|d-d_o| < w_n} \exp\{l_n(d,k) - l_n(d_o,k) + \log I_n(d,k)\} d\pi_d(d)}.
\end{split}
\end{equation} \noindent
Let $\delta_2 >0$ and $A_n$ be the set of $X \in \R^n$ such that
\begin{equation*} \label{thetaBound2}
\left.
\begin{array}{r}
\left|\log I_n(d,k) - \log I_n(d_o,k)\right| \leq \e_1, \\
  \left| l_n^{(1)}(d_o,k) -  \mathcal D(l_n^{(1)}(d_o,k)) \right| \leq n^{\ft} k^{-\ft}\sqrt{\log n},\\
  \left| l_n^{(2)}(d,k) - \mathcal D(l_n^{(2)}(d_o,k)) \right| \leq  \epsilon_1n^{-(2+\delta_2)/(2\be + 1)}
\end{array} 
\right\}
\end{equation*}
for all $|d-d_o|\leq w_n$ and $k_n'' \leq k \leq k_n'$.
Compared to the definition of $A_n$ in the proof of Theorem \ref{BVMtheorem}, the constraints on $l_n^{(2)}(d,k)$ and $I_n$ are different. For the latter, recall from Lemma \ref{lem:Idk:th1} that $\log I_n(d,k)  =  \log I_n(d_o,k)+\op(1)$, uniformly over $d \in (d_o-w_n,d_o+w_n)$.
As in the proof of Theorem \ref{BVMtheorem}, it now follows from Lemmas \ref{delta1terms} and \ref{lem:Idk:th1} that $P_o^n(A_n^c) \rightarrow 0$.
We can write
\begin{equation*} \label{Esplit2}
E_0^n \left[\frac{N_{n,k}}{D_{n,k}}\right] \leq P_o^n(A_n^c) + E_0^n\left[\frac{N_{n,k}}{D_{n,k}} 1_{A_n} \right],
\end{equation*}
and bound $N_{n,k}/D_{n,k}$ pointwise for $X \in A_n$.
Since when $k \in (k_n'', k_n^{'})$,
 $$\frac{(l_n^{(1)}(d_o,k))^2}{2|l_n^{(2)}(d_o,k)|} n^{-(2+\delta_2)/(2\be + 1)} = o(1)$$
 on $A_n$, for all $\delta_2>0$,
analogous to \eqref{upperBound} and \eqref{lowerBound}, we find that for all $X \in A_n$, by definition of $b_n(d_o)$,
\begin{equation*}
\begin{split}
 l_n(d,k) - l_n(d_o,k) + \log I_n(d,k) &\leq 2\e_1 - \frac{|l_n^{(2)}(d_o,k)|}{2} \left(d-d_o -b_n(d_o) \right)^2 + \frac{(l_n^{(1)}(d_o,k))^2}{2|l_n^{(2)}(d_o,k)|}\\
 l_n(d,k) - l_n(d_o,k) + \log I_n(d,k) & \geq  -2\e_1 -  \frac{|l_n^{(2)}(d_o,k)|}{2} \left(d-d_o -b_n(d_o) \right)^2 + \frac{(l_n^{(1)}(d_o,k))^2}{2|l_n^{(2)}(d_o,k)|},
\end{split}
\end{equation*}
when $n$ is large enough since $k > k_n''$.
We now lower-bound $b_n(d_o)$ by bounding the terms on the right in \eqref{Bdef} in Lemma \ref{delta1terms}. By construction of $\ty_o$ it follows that
\begin{equation*}\label{scoreBound1}
r_k^{-1} \sum_{j>k} j^{-1} \ty_{o,j} = c_0 r_k^{-1} \sum_{j>k} j^{-\be-\frac{3}{2}} / (\log j) \geq c k^{-\be +\ft} (\log k)^{-1},
\end{equation*}
for some $c>0$. Since $X \in A_n$, $2 \mathcal S(l_n^{(1)}(d_o,k))/(n r_k) \leq 2\sqrt{k/n} \sqrt{\log n}$. Since $k \leq k_n'$, this bound is $o(k^{-\be +\ft} (\log k)^{-1})$.
The last term in \eqref{Bdef} is $o(n^{\e-1})$ when $\be > 5/2$, and hence this term is also $o(k^{-\be -\ft} (\log k)^{-1})$.
Therefore, the last two terms in \eqref{Bdef} are negligible with respect to $r_k^{-1} \sum_{j>k} j^{-1} \ty_{o,j}$. We deduce that $b_n(d_o) \geq  c k^{-\be+ \ft}(\log k)^{-1} \geq  c n^{-(2\be -1)/(4\be + 2)} (\log n)^{-(2\be + 3)/(4\be + 2)}$ for $n$ large enough.

Consequently, when the constant $k_v$ is chosen sufficiently small, $\sqrt{nr_{k_n'}}(b_n(d_o) - k_vw_n(\log n)^{-1} ) \geq (c-k_v)n^{1/(4\be + 2)} (\log n)^{-(\be +1)/(2\be + 1)} := z_n \rightarrow \infty$.
We now substitute the above bounds on $l_n(d,k) - l_n(d_o,k) + \log I_n(d,k)$ in the right hand side of \eqref{posterior2}, make the change of variables $u = d-d_o -  b_n(d_o)$ and obtain
\begin{equation*} 
\begin{split}
\frac{N_{n,k}}{D_{n,k}} &\leq e^{5 \e_1 } \frac{ \int_{u\leq - k_vw_n(\log n)^{-1} -b_n(d_o) } \exp\{-\frac{n r_k u^2}{4}\} du }{ \int_{|u|<w_n / 2}\exp\{-\frac{n r_k u^2 }{4} \}du } \\ &\leq e^{5 \e_1 } \frac{ \int_{v > z_n } \exp\{-\frac{ v^2}{2 }\} dv }{\int_{|v|<w_n \sqrt{n r_k / 8}} \exp\{-\frac{ v^2}{2 } \} dv } =\op(1).
\end{split}
\end{equation*}
This achieves the proof of Theorem \ref{suboptimalTheorem}.

\section{Conclusion} \label{sec6}

In this paper we have derived conditions  leading to a BVM type of result for the long memory parameter $d\in (-\ft, \ft)$ of a stationary Gaussian process, for the class of FEXP-priors.  To our knowledge such a result has not been obtained before. The result implies in particular that asymptotically credible intervals for $d$ have good frequentist coverage.

A by-product of our results is that the \textit{most natural prior } (Prior C) from a Bayesian perspective, which is also the prior leading to adaptive minimax rates under the loss function $l$ on $f$, leads to sub-optimal estimators in terms of $d$. Prior A leads to optimal estimators for $d$ however it is not adaptive.  An interesting direction for future work would be to define an adaptive- minimax  estimation procedure for $d$.

More broadly speaking, the approach considered here to derive the asymptotic posterior distribution of a finite dimensional parameter of interest in a semi-parametric problems could be used in other non - regular models, hence completing (not exhaustively) the recent works of \citet{castillo:10} and \citet{kleijn:bickel:10}.

\section{Acknowledgements}
This work was supported by the 800-20072010 grant ANR-07-BLAN-0237-01 SP Bayes.

\bibliographystyle{apalike} %
\bibliography{memlong}%

\begin{thebibliography}{}

\bibitem[Arbel, 2010]{arbel:10}
Arbel, J. (2010).
\newblock Bayesian optimal adaptive estimation using a sieve prior, submitted.

\bibitem[Beran, 1993]{beran93}
Beran, J. (1993).
\newblock Fitting long-memory models by generalized linear regression.
\newblock {\em Biometrika}, 80(4):817--822.

\bibitem[Beran, 1994]{beran94}
Beran, J. (1994).
\newblock {\em Statistics for long-memory processes}, volume~61 of {\em
  Monographs on Statistics and Applied Probability}.
\newblock Chapman and Hall, New York.

\bibitem[Bickel and Kleijn, 2010]{kleijn:bickel:10}
Bickel, P. and Kleijn, B. (2010).
\newblock The semiparametric bernstein-von mises theorem.

\bibitem[Castillo, 2010]{castillo:10}
Castillo, I. (2010).
\newblock A semiparametric bernstein von mises theorem for gaussian process
  priors.
\newblock {\em Probability Theory and Related Fields}.

\bibitem[Dahlhaus, 1989]{d1}
Dahlhaus, R. (1989).
\newblock Efficient parameter estimation for self-similar processes.
\newblock {\em Ann. Statist.}, 17(4):1749--1766.

\bibitem[Freedman, 1999]{freedman:99}
Freedman, D. (1999).
\newblock On the bernstein-von mises theorem with infinite-dimensional
  parameters.
\newblock {\em Ann. Statist.}, 27(4):1119--1140.

\bibitem[Grenander and Szeg{\"o}, 1958]{grsz1958}
Grenander, U. and Szeg{\"o}, G. (1958).
\newblock {\em Toeplitz forms and their applications}.
\newblock California Monographs in Mathematical Sciences. University of
  California Press, Berkeley.

\bibitem[Holan and McElroy, 2010]{holan:mcelroy:chakra:09}
Holan, S.~H. and McElroy, T.~S. (2010).
\newblock Tail exponent estimation via broadband log density-quantile
  regression.
\newblock {\em J. Stat. Plann. Inference}, 140(12):3693--3708.

\bibitem[Hurvich et~al., 2002]{hcms02}
Hurvich, C.~M., Moulines, E., and Soulier, P. (2002).
\newblock The {FEXP} estimator for potentially non-stationary linear time
  series.
\newblock {\em Stochastic Process. Appl.}, 97(2):307--340.

\bibitem[Jensen, 2004]{Jensen:04}
Jensen, M.~J. (2004).
\newblock Semiparametric {B}ayesian inference of long-memory stochastic
  volatility models.
\newblock {\em J. Time Ser. Anal.}, 25(6):895--922.

\bibitem[Ko et~al., 2009]{ko:qu:vannucci:09}
Ko, K., Qu, L., and Vannucci, M. (2009).
\newblock Wavelet-based {B}ayesian estimation of partially linear regression
  models with long memory errors.
\newblock {\em Statist. Sinica}, 19(4):1463--1478.

\bibitem[Li and Zhao, 2002]{li:zhao:02}
Li, X. and Zhao, L.~H. (2002).
\newblock Bayesian nonparametric point estimation under a conjugate prior.
\newblock {\em Stat. Probab. Lett.}, 1(4):23--30.

\bibitem[Lieberman and Phillips, 2004]{lp2}
Lieberman, O. and Phillips, P. C.~B. (2004).
\newblock Error bounds and asymptotic expansions for {T}oeplitz product
  functionals of unbounded spectra.
\newblock {\em J. Time Ser. Anal.}, 25(5):733--753.

\bibitem[Lieberman et~al., 2011]{LRR09}
Lieberman, O., Rosemarin, R., and Rousseau, J. (2011).
\newblock Asymptotic theory for maximum likelihood estimation in stationary
  fractional gaussian processes, under short, long and intermediate memory.
\newblock {\em Econometric Theory}.

\bibitem[Lieberman et~al., 2003]{lrz1}
Lieberman, O., Rousseau, J., and Zucker, D.~M. (2003).
\newblock Valid asymptotic expansions for the maximum likelihood estimator of
  the parameter of a stationary, {G}aussian, strongly dependent process.
\newblock {\em Ann. Statist.}, 31(2):586--612.
\newblock Dedicated to the memory of Herbert E. Robbins.

\bibitem[Moulines and Soulier, 2003]{ms1}
Moulines, E. and Soulier, P. (2003).
\newblock Semiparametric spectral estimation for fractional processes.
\newblock In {\em Theory and applications of long-range dependence}, pages
  251--301. Birkh\"auser Boston, Boston, MA.

\bibitem[Philippe and Rousseau, 2002]{pr1}
Philippe, A. and Rousseau, J. (2002).
\newblock Non-informative priors in the case of {G}aussian long-memory
  processes.
\newblock {\em Bernoulli}, 8(4):451--473.

\bibitem[Rivoirard and Rousseau, 2010]{riro10}
Rivoirard, V. and Rousseau, J. (2010).
\newblock Bernstein-von mises theorem for linear functionals of the density.

\bibitem[Robinson, 1994]{robinson94a}
Robinson, P.~M. (1994).
\newblock Time series with strong dependence.
\newblock In {\em Advances in econometrics, {S}ixth {W}orld {C}ongress, {V}ol.\
  {I} ({B}arcelona, 1990)}, volume~23 of {\em Econom. Soc. Monogr.}, pages
  47--95. Cambridge Univ. Press, Cambridge.

\bibitem[Robinson, 1995]{robinson95a}
Robinson, P.~M. (1995).
\newblock Gaussian semiparametric estimation of long range dependence.
\newblock {\em Ann. Statist.}, 23(5):1630--1661.

\bibitem[Rousseau et~al., 2010]{rcl}
Rousseau, J., Chopin, N., and Liseo, B. (2010).
\newblock Bayesian nonparametric estimation of the spectral density of a long
  memory gaussian process.

\bibitem[Rousseau and Kruijer, 2011]{kruijer:rousseau:11}
Rousseau, J. and Kruijer, W. (2011).
\newblock Bayesian semi-parametric estimation of the long-memory parameter
  under fexp-priors.

\end{thebibliography}

\appendix

\section{Proof of Lemma 3.4} \label{appD}

We decompose the first derivative of $l_n(d,k)$ as $l_n^{(1)}(d,k)= \mathcal S (l_n^{(1)}(d,k)) + \mathcal D (l_n^{(1)}(d,k))$, $\mathcal S (l_n^{(1)}(d,k))$ being a centered quadratic form and $\mathcal D (l_n^{(1)}(d,k))$ the remaining deterministic term.  To simplify notations, in this proof we write $\mathcal S =\mathcal S(l_n^{(1)}(d_o,k))$ and $\mathcal D = \mathcal D (l_n^{(1)}(d_o,k))$.
 Using (1.6) (supplement) 
and defining $A = T_n^{-1}(f_{d_o,k}) T_n(H_k f_{d_o,k}) T_n^{-1}(f_{d_o,k})$, we find that
\begin{equation*} 
\begin{split}
\mathcal D = -\ft \tr\left[(T_n(f_{d_o,k}) - T_n(f_o)) A \right], \quad \mathcal S = \ft \left(X^t A X - \tr\left[T_n(f_o) A\right]\right).
\end{split}
\end{equation*}
 From (1.4) and (1.8) in the supplement it follows that
\begin{equation}\label{f0fd0bound}
\begin{split}
f_o - f_{d_o,k} &= f_{d_o,k} (e^{\Delta_{d_o,k}} - 1) = (\Delta_{d_o,k} + \ft e^\xi \Delta_{d_o,k}^2) f_{d_o,k} \\
&   = f_{d_o,k} O(k^{-\be +1/2}), \quad \xi \in (0, (\Delta_{d_o,k})_+)).
\end{split}
\end{equation}
Consequently, we have
\begin{equation*} 
\begin{split}
& \mathcal D = \ft \tr\left[ T_n(f_{d_o,k}(\Delta_{d_o,k} + O(\Delta_{d_o,k}^2))) T_n^{-1}(f_{d_o,k}) T_n(H_k f_{d_o,k}) T_n^{-1}(f_{d_o,k}) \right] \\ & \; = \frac{n}{4 \pi}\int_{-\pi}^{\pi} H_k(x) (\Delta_{d_o,k}(x) + O(\Delta_{d_o,k}^2(x))) dx + \mbox{error}  \\
& \; =  \frac{n}{2} \sum_{j={k+1}}^\infty \eta_j \ty_{o,j} + O(nk^{-2\be -1}) + \mbox{error}.
\end{split}
\end{equation*}
The last equality follows from (1.9) and  (1.11) in the supplement. 
We bound the error term using Lemma 2.4 (supplement) applied to $H_k f_{d_o,k}$ and $f_{d_o,k}$, whose Lipschitz constants are bounded by $O(k)$ and $O(k^{(3/2-\be)_+}$, respectively (see Lemma 3.1 in the supplement). Using that $\| \Delta_{d_o,k}\|_\infty = O(k^{-\be + 1/2})$ (see (1.8) in the the supplement) we then find that the error is $O(k^{3/2- \be} n^\e)$.

The term $\mathcal S$ is a centered quadratic form with variance $\frac{1}{2}\fr T_n^{\ft}(f_o) A  T_n^{\ft}(f_o)\fr^2$. Applying once more \eqref{f0fd0bound}, we find that
\begin{equation*}
\begin{split}
\tr\left[ (T_n(f_o)A )^2\right] &= \tr\left[ \left(T_n^{-1}(f_{d_o,k}) T_n(H_k f_{d_o,k}) \right)^2 \right] (  1 + O(\| \Delta_{d_o,k}\|_\infty )) \\
 &= \frac{ n }{ 2\pi } \int_{-\pi}^\pi H_k^2(x) dx + O(n^\e k+ \| \Delta_{d_o,k}\|_\infty) = nr_k( 1 + o(n^{-\delta})),
\end{split}
\end{equation*}
where the term $n^\e k$ comes from Lemma 2.4 in the supplement, associated to $f_{d_o,k}$ and $f_{d_o,k}H_k$. This proves the first equality in Lemma \ref{delta1terms}.

Similar to the decomposition of $l_n^{(1)}(d_o,k)$, we decompose the second derivative as $l_n^{(2)}(d,k) = \mathcal D(l_n^{(2)}(d,k)) - 2 \mathcal S_1(l_n^{(2)}(d,k)) +  \mathcal S_2(l_n^{(2)}(d,k))$, where
\begin{equation*}
\mathcal S_1(d)  =  X^t A_{1,d} X - \tr[T_n(f_o) A_{1,d}], \quad
\mathcal S_2(d) = X^t A_{2,d} X - \tr[T_n(f_o) A_{2,d}],
\end{equation*}
\begin{equation}\label{Dlabel}
\begin{split}
\mathcal D_2(d)  &:=  \mathcal D(l_n^{(2)}(d,k)) \\
&= - \frac{1}{2} \tr\left[ T_n(f_{d,k})A_{1,d} \right] + \tr\left[  (T_n(f_{d,k})  - T_n(f_{d_o,k})) \left(A_{1,d} -\ft A_{2,d}\right)\right] \\
& \quad + \tr\left[  (T_n(f_{d_o,k})  - T_n(f_o)) \left(A_{1,d} -\ft A_{2,d}\right)\right] ,
\end{split}
\end{equation}
\begin{equation*}
\begin{split}
A_{1,d} &= T_n^{-1}(f_{d,k})( T_n(H_k f_{d,k})  T_n^{-1}(f_{d,k}))^2, \; A_{2,d}  = T_n^{-1}(f_{d,k})T_n(H_k^2f_{d,k})  T_n^{-1}(f_{d,k}).
\end{split}
\end{equation*}

To control $\mathcal D_2(d)$ we use a first order Taylor expansion around $d_o$, implying that $\mathcal D_2(d) = \mathcal D_2(d_o)  + O(|d-d_o) \sup_{|d'-d_o| \leq \vb }|\mathcal D_2'(d')|$.
First we study $\mathcal D_2(d_o)$.
At $d = d_o$,  the right-hand side of \eqref{Dlabel} equals
\begin{equation} \label{dd}
\begin{split}
& -\frac{n}{4\pi} \int_{-\pi}^{\pi} H_k^2(x)\left( 1 + (e^{\Delta_{d_o,k}}-1) \right)dx +  O( kn^\e) \\
& \qquad = - \frac{ -nr_k}{ 2 } \left( 1 +  O(k^{-\be +1/2} + k^2/n^{1-\e}) \right).
\end{split}
\end{equation}
The $O( kn^\e)$  term is obtained from Lemma 2.4 (supplement), applied to $f_{2j} = H_k f_{d_o,k}$  and $f_{2j-1} = f_{d_o,k}$, with Lipschitz constants $O(k)$ for the former and $O(k^{(3/2 - \be)_+} )$ for the latter, together with the bound $\| \Delta_{d_o,k}\|_\infty = O(k^{-\be +1/2})$.
Using
\begin{equation} \label{A1der}
\begin{split}
A_{1,d}'   &= - 3 \left(T_n^{-1}(f_{d,k})T_n(H_kf_{d,k})\right)^3  T_n^{-1}(f_{d,k}) \\
&\quad + 2 T_n^{-1}(f_{d,k}) T_n(H_k^2 f_{d,k})  T_n^{-1}(f_{d,k})T_n(H_k f_{d,k}) T_n^{-1}(f_{d,k})
\end{split}
\end{equation}
and a similar expression for the derivative of $d \rightarrow A_{2,d}$, it follows that
\begin{equation*}
\begin{split}
& | \mathcal D_2'(d')| \lesssim \tr\left[(T_n(f_{d_o,k}) T_n^{-1}(f_{d',k}) + I_n) (T_n(|H_k|f_{d',k})T_n^{-1}(f_{d',k}))^3\right]  \\
& \; +  \tr\left[(T_n(f_{d_o,k}) T_n^{-1}(f_{d',k}) + I_n)(T_n(|H_k|f_{d',k})T_n^{-1}(f_{d',k})) T_n(H_k^2f_{d',k}) T_n^{-1}(f_{d',k})\right]  \\
 & \; + \tr\left[(T_n(f_{d_o,k}) T_n^{-1}(f_{d',k}) + I_n)T_n(|H_k|^3f_{d',k})T_n^{-1}(f_{d',k})\right]
 \end{split}
 \end{equation*}
 We control the first term of the right hand side of the above inequality, the second and third terms are controlled similarly.  Note first that
 \begin{equation} \label{Dprime}
 \begin{split}
&  \tr\left[T_n(f_{d_o,k}) T_n^{-1}(f_{d',k}) (T_n(|H_k|f_{d',k})T_n^{-1}(f_{d',k}))^3\right] \\
 &= \fr T_n^{\ft}(f_{d_o,k}) T_n^{-1}(f_{d',k}) T_n(|H_k|f_{d',k})T_n^{-1}(f_{d',k}) T_n^{\ft}(|H_k|f_{d',k})\fr^2  \\
 &\leq \| T_n^{\ft}(f_{d_o,k})T_n^{-\ft}(f_{d',k}) \|^2 \\
 & \qquad \times \|T_n^{-\ft}(f_{d',k}) T_n^{\ft}(|H_k|f_{d',k})\|^2 \fr T_n^{-\ft}(f_{d',k}) T_n(|H_k|f_{d',k})T_n^{-\ft}(f_{d',k})\fr^2 \\
 &\lesssim n^{\e} \fr T_n^{-\ft}(f_{d',k}) T_n(|H_k|f_{d',k}) T_n^{-\ft}(f_{d',k})\fr^2 ,
 \end{split}
 \end{equation}
 where the last inequality comes from Lemma 2.3 in the supplement. Note also that
  $$|x|^{-2d'} \lesssim f_{d'}(x) \lesssim |x|^{-2d'} \quad \mbox{ and }  T_n(|H_k|f_{d',k}) \lesssim T_n(|H_k| |x|^{-2d'}) , \quad T_n(|f_{d',k})\gtrsim  T_n(|x|^{-2d'})$$
  and replace $f_{d'}$ by $|x|^{-2d'}$ in \eqref{Dprime}, then
$$\fr T_n^{-1/2}(f_{d',k}) T_n(|H_k|f_{d',k})T_n^{-1/2}(f_{d',k})\fr^2\lesssim \left( \frac{ n }{ k } + O(k)\right)$$
using Lemma 2.4 in the supplement associated to $|H_k| |x|^{-2d'}$ which has Lipschitz constant $k$. This leads to
$
\mathcal D_2(d') = O\left(n^\e \frac{ n }{ k } \right),
$
which implies that for all $\be > 1$,
\begin{equation*}
\mathcal D_2(d) = \mathcal D_2(d_o) + \od(|d-d_o|n^{\e+1}k^{-1}) = - \frac{ n r_k }{ 2} ( 1 + \od(n^{-\delta})).
\end{equation*}
For the stochastic terms in $l_n^{(2)}(d,k)$ we need a chaining argument to control the supremum over $d \in (d_o-\vb,d_o+\vb)$.
We show that for all $\e'>0$ and $\gamma_n = n^{\ft+\e'} k^{-\ft}$,
\begin{eqnarray}
P_o^n\left( \sup_{|d-d_o| \leq \vb} |\mathcal S_1(d)| > \gamma_n \right) &=& o(1), \label{w12}
\end{eqnarray}
i.e. that $\mathcal S_1(d)=\op(\gamma_n)$.
The same can be shown for $\mathcal S_2(d)$ using exactly the same arguments. Consider a covering of $(d_o-\vb, d_o + \vb)$ by balls of radius $n^{-1}$
centered at $d_j$, $j=1,\ldots, J_n$ with $J_n \leq 2\vb n$. Then
\begin{equation*}
\sup_{|d- d_o|<\vb} |\mathcal S_1(d)| \leq \max_j |\mathcal S_1(d_j)| +  \sup_{|d-d'|\leq n^{-1}} |\mathcal S_1(d) - \mathcal S_1(d)|,
\end{equation*}
and
\begin{equation} \label{chaining}
\begin{split}
P_o^n\left( \sup_{|d-d_o| \leq \vb} |\mathcal S_1(d)| > \gamma_n \right) & \leq
P_o^n\left( \sup_{|d-d'|\leq n^{-1}} |\mathcal S_1(d) - \mathcal S_1(d')| > \ft \gamma_n \right) \\
& \quad + J_n \max_{1 \leq j\leq J_n}P_o^n\left(|\mathcal S_1(d_j)|  > \ft \gamma_n \right).
\end{split}
\end{equation}
To control the first term on the right in \eqref{chaining}, note that for a standard normal vector $Z$ and some $d^*\in (d,d')$,
\begin{equation*}
\begin{split}
\mathcal S_1(d)-\mathcal S_1(d')
&= (d-d') \left( Z^t T_n^\ft(f_o) A_{1,d^*}' T_n^\ft(f_o) Z - \tr\left[T_n(f_o) A_{1,d^*}' \right] \right),
\end{split}
\end{equation*} \noindent
with $A_{1,d}'$ as in \eqref{A1der}.
Using Lemma 2.3 (supplement) and the fact that $\|AB\|\leq \|A\| \|B\|$ for all matrices $A$ and $B$,
it follows that $\|T_n^\ft(f_o) A_{1,d^*}' T_n^\ft(f_o)\| = O(n^\e )$, and hence $Z^t T_n^\ft(f_o) A_{1,d^*}' T_n^\ft(f_o) Z \leq Z^t Z \|T_n^\ft(f_o) A_{1,d^*}' T_n^\ft(f_o)\| = O(n^\e ) Z^t Z$. Similarly, it follows that $\tr\left[T_n(f_o) A_{1,d^*}' \right] \lesssim n$.
Consequently, when $\e = \e'/2$ we have
\begin{equation*}
| \mathcal S_1(d)-\mathcal S_1(d') | \lesssim n^{-1} \left(  Z^t Z n^{\e} + n\right),
\end{equation*} \noindent
uniformly over all $d,d'$ such that $|d-d'| \leq n^{-1}$.
Since $1 = o(\gamma_n)$,
\begin{equation*}
P_o^n\left( \sup_{|d-d'|\leq n^{-1}} |\mathcal S_1(d) - \mathcal S_1(d')| \geq \ft \gamma_n \right) \leq P\left( Z^t Z >  n^{1-\e}\gamma_n /4 \right) = o(1).
\end{equation*} \noindent

To bound the last term in \eqref{chaining}, 
we apply Lemma 1.3 (supplement) to $( Z^t A Z - \tr[ A]) \fr A \fr^{-1}$, with
$A = T_n^{\ft}(f_o)A_{1,d}T_n^{\ft}(f_o)$ since
as seen previously $\fr A \fr^2 = \mathbf{O}(n/k) = \mathbf{o}(\gamma_n^2 n^{-2\alpha})$ for $\alpha$ small enough,
it follows that
\begin{eqnarray*}
P_o^n\left(\mathcal S_1(d) \geq \ft \gamma_n \right) \leq  e^{-n^{\alpha}/8}.
\end{eqnarray*}
Since $J_n$ increases only polynomially with $n$, this finishes the proof of \eqref{w12}.

\section{Control of  the derivatives in $\theta$ on the log-likelihood }\label{app:deriv:theta}

Before stating Lemma \ref{taylorLemma} we first give  a general expression for the derivatives of $l_n(d,k,\ty)$ with respect to $\ty$.
For all $j \geq 1$ and $l = (l_1,\ldots,l_j) \in \{0,1,\ldots,k\}^j$, let  $\s = (\s(1),\ldots,\s(|\s|))$ be a partition of $\{1,\ldots,j\}$. Let $|\s|$ be the number of subsets in this partition and $\s(i)$ the $i$th subset of $\{1,\ldots,j\}$ in the partition $\s$. Denoting $l_{\s(i)}$  the vector $(l_t, t \in \s(i))$, we can write
\begin{eqnarray*} \label{deriv:theta:d}
\nabla_{l_{\s(i)}} f_{d,k,\ty}(x) = \prod_{t\in \s(i)} \cos (l_t x) f_{d,k,\ty}(x).
\end{eqnarray*}
For notational ease we write $\nabla_{\s(i)} f_{d,k,\ty} := \nabla_{l_{\s(i)}} f_{d,k,\ty}$.
The derivative $\frac{\partial^j l_n(d,k,\ty) }{ \partial \ty_{l_1}\ldots\partial \ty_{l_j}}$ can now be written in terms of the matrices
\begin{equation}\label{bsigma}
B_{\s}(d,\ty) = \prod_{i=1}^{|\s|} B_{\s(i)}(d,\ty), \quad B_{\s(i)}(d,\ty)=  T_n(\nabla_{\s(i)} f_{d,k,\ty}) T_n^{-1}(f_{d,k,\ty}).
\end{equation}
There exist constants $b_\s, c_\s$ and $d_\s$ such that
\begin{equation} \label{dThetaExpression}
\begin{split}
& \frac{\partial^j l_n(d,k,\ty) }{ \partial \ty_{l_1}\ldots\partial \ty_{l_j}} \\
&=\sum_{\s \in \mathcal S_j} b_{\s} \left(X^t T_n^{-1}(f_{d,k,\ty}) B_{\s}(d,\ty) X - \tr\left[ T_n(f_o) T_n^{-1}(f_{d,k,\ty})B_{\s}(d,\ty) \right] \right) \\
& \;
+ \sum_{\s \in \mathcal S_j} c_{\s} \tr\left[B_{\s}(d,\ty)\right] + \sum_{\s \in \mathcal S_j} d_{\s} \tr\left[( T_n(f_o) T_n^{-1}(f_{d,k,\ty}) - I_n)B_{\s}(d,\ty) \right],
\end{split}
\end{equation}
where $\mathcal S_j$ is the set of partitions of $\{1,\ldots,j\}$. For the first two derivatives ($j=1,2$) the values of the constants
$b_\s$, $c_\s$ and $d_\s$ are given below in Lemmas \ref{scoreLemma} and \ref{fisherLemma}. For the higher order derivatives
these values are not important for our purpose; we will only need that for any $j\geq 1$, the constant $c_\s$ is zero if $|\s|=1$.

The following lemma states that $l_n(d,k,\ty) - l_n(d,k)$ is the sum of a Taylor-approximation $\sum_{j=1}^J \frac{(\ty-\bd)^{(j)}\nabla^j l_n(d_o,k) }{j !}$ and terms whose dependence on $d$ can be negligible. Since the proof is involved, some of the technical details are treated in Lemmas \ref{mixedderiv} and \ref{lem:g1}.

\begin{lem} \label{taylorLemma}
Given $\be > 1$, let $k \leq k_n$ and let $d$ and $\ty$ be such that $l(f_o,f_{d,k,\ty}) \leq l_0^2\dn^2 $.
Then there exists an integer $J$ and a constant $\e>0$ such that uniformly over $d \in (d_o - \vb,d_o + \vb)$  and $\ty \in \bn$,
\begin{equation} \label{eq:taylorLem}
\begin{split}
l_n(d,k,\ty) - l_n(d,k) &= \sum_{j=1}^J \frac{(\ty-\bd)^{(j)}\nabla^j l_n(d_o,k) }{j !}  \\ & \quad + (d-d_o) \sum_{j=2}^J \frac{1}{j!} g_{n,j}(\ty-\bd) + S_n(d),
\end{split}
\end{equation}
where, for $u = \ty - \bd$,
\begin{equation} \label{gnDecomposition}
g_{n,j}(u) = \sum_{l_1,\ldots,l_j = 0}^k u_{l_1} \cdots u_{l_j} \sum_{\s \in \mathcal S_j}  \left(c_\s \tr\left[ T_{1,\s} (d_o,k) \right]  + d_\s \tr \left[ T_{2,\s}(d_o,k)\right] \right),
\end{equation}
\begin{eqnarray*}
T_{1,\s} (d_o,k) &=& \sum_{i=1}^{|\s|}\left(  \prod_{l < i} T_n(\nabla_{\s(l)} f_{d_o,k}) T_n^{-1}(f_{d_o,k}) \right) \times  \\
 & & \quad \left[  T_n(\nabla_{\s(i)}  f_{d_o,k}H_k) - T_n(H_k f_{d_o,k}) T_n^{-1}(f_{d_o,k}) T_n (\nabla_{\s(i)}  f_{d_o,k} )  \right] \times \\
  & & \quad  T_n^{-1}(f_{d_o,k}) \left(  \prod_{l > i} T_n(\nabla_{\s(l)} f_{d_o,k}) T_n^{-1}(f_{d_o,k}) \right),
\end{eqnarray*}
\begin{eqnarray*}
T_{2,\s} (d_o,k) &=& - T_n(H_k f_{d_o,k}) T_n^{-1}(f_{d_o,k}) B_{\s}(d_o,\bz), 
\end{eqnarray*}
and $S_n(d)$ denotes any term of order
\begin{equation} \label{Sproperty}
S_n(d) = \op(1)+\op\left(\frac{|d-d_o| n^{\ft-\delta}}{ \sqrt{k}}\right)+\op\left((d-d_o)^2 \frac{ n^{1-\delta} }{k}\right).
\end{equation}
When $\be > 5/2$ and $k \leq k_n'$, we can choose $J=2$, and \eqref{eq:taylorLem} simplifies to
\begin{equation} \label{eq:taylorLem2}
l_n(d,k,\ty) - l_n(d,k) = \sum_{j=1}^2 \frac{(\ty - \bd)^{(j)}\nabla^j l_n(d_o,k) }{j !} + \op(1).
\end{equation}

\end{lem} \noindent
\begin{proof}

Recall that by \eqref{taylor1},
\begin{equation} \label{taylor2}
\begin{split}
& l_n(d,k,\ty) - l_n(d,k) =
\sum_{j=1}^J \frac{(\ty-\bd)^{(j)}\nabla^j l_n(d_o,k) }{j !} \\
&\qquad + \sum_{j=1}^J \frac{(\ty-\bd)^{(j)}\nabla^j (l_n(d,k) - l_n(d_o,k)) }{j !}
+ R_{J+1,d}(\ty).
\end{split}
\end{equation} \noindent
To prove \eqref{eq:taylorLem} we first show that, writing $u=\ty-\bd$,
\begin{equation} \label{taylorDif}
\begin{split}
&\sum_{j=1}^J \frac{u^{(j)}\nabla^j (l_n(d,k) - l_n(d_o,k)) }{j !} \\
& \qquad =\sum_{j=1}^J \frac{1}{j!} \sum_{l_1,\ldots,l_j = 0}^k u_{l_1}\ldots u_{l_j} \left(   \frac{\partial^j l_n(d,k,\bd) }{ \partial \ty_{l_1}\ldots\partial \ty_{l_j}}
- \frac{\partial^j l_n(d_o,k,\bz) }{ \partial \ty_{l_1}\ldots\partial \ty_{l_j}} \right) \\
& \qquad = (d-d_o) \sum_{j=1}^J \frac{1}{j!} g_{n,j}(u) + O(S_n(d)).
\end{split}
\end{equation} \noindent
This result is combined with \eqref{taylor2} and Lemma \ref{lem:g1} below, by which $g_{n,1}(u) = O(S_n(d))$. It then follows that $l_n(d,k,\ty) - l_n(d,k)$ equals
\begin{equation*}
\begin{split}
\sum_{j=1}^J \frac{(\ty-\bd)^{(j)}\nabla^j l_n(d_o,k) }{j !}
+ (d-d_o) \sum_{j=2}^J \frac{1}{j!} g_{n,j}(u) + R_{J+1,d}(\ty) + O(S_n(d)) .
 \end{split}
\end{equation*}
The final step is to prove that $R_{J+1,d}(\ty)$ is $\op(1)$ and hence $O(S_n(d))$; to this end $J$ needs to be sufficiently large.

First we prove \eqref{taylorDif}. For the factors $\frac{\partial^j}{\partial \ty} l_n(d,k,\bd) - \frac{\partial^j}{\partial \ty} l_n(d,k,\bz)$ we substitute \eqref{dThetaExpression}. In Lemma \ref{mixedderiv} below we give expressions for each of the terms therein, which we substitute in \eqref{taylorDif}. The main terms are $(d-d_o) \tr[T_{1,\s}(d_o,k)]$ and $(d-d_o) \tr[T_{2,\s}(d_o,k)]$ in \eqref{detTerm1} and \eqref{detTerm2}, which after substitution in \eqref{taylorDif} give the term $(d-d_o) \sum_{j=1}^J \frac{1}{j!} g_{n,j}(u)$ on the right. The other terms in \eqref{stochTerm}-\eqref{detTerm2} that enter \eqref{taylorDif} through \eqref{dThetaExpression} are $O(S_n(d))$. This is due to the summation over $u_{l_1},\ldots,u_{l_j}$ in \eqref{taylorDif}, and
the Cauchy-Schwarz inequality by which
\begin{equation} \label{cauchySchwarzbound}
\left| \sum_{l_1,\ldots,l_j=0}^k u_{l_{1}}\ldots u_{l_{j }}\right| \leq (\sqrt{k}\|u\|)^j \leq (2 l_0 \sqrt{k} \dn)^j = o(n^{-\delta}),
\end{equation}
for some $\delta >0$, as $\|u\| \leq 2 l_0\dn$ and \eqref{taylorDif} is proved. We now control $R_{J+1,d}(\ty)$.

Combining \eqref{Rexpression} and the first inequality in \eqref{cauchySchwarzbound}, we obtain
\begin{equation*}
|R_{J+1,d}(\ty) | \leq \frac{ 1}{(J+1)!} (\sqrt{k}\dn)^{J+1} \max_{l_1,\ldots,l_{J+1}}\sup_{\|\tilde \ty-\bd\| \leq 2 l_0\dn} \left| \frac{ \partial^{J+1} l_n(d,k,\tilde \ty) }{ \partial \ty_{l_1}\ldots\partial \ty_{l_{J+1}}} (x)\right|.
\end{equation*} \noindent
We give a direct bound on this derivative using \eqref{dThetaExpression}.
For all partitions $\sigma$ of $\{l_1,\ldots,l_{J+1}\}$ and all $(l_1,\ldots,l_{J+1}) \in \{ 1,\ldots,k\}^{J+1}$,
we bound $\| B_{\s(i)}(d,\ty)\|$, using  $\| B_{\s(i)}(d,\ty)\| \leq  \|T_n^{\ft}(\nabla_{\s(i)} f_{d,k,\ty}) T_n^{-\ft}(f_{d,k,\ty})\|^2$  (see \eqref{matrixCalculus}). We bound $\|T_n^{\ft}(\nabla_{\s(i)} f_{d,k,\ty}) T_n^{-\ft}(f_{d,k,\ty})\|$  by application of Lemma  2.3 (supplement) with $f=f_{d,k,\ty}$ and $g=\nabla_{\s(i)} f_{d,k,\ty}$.
The constant $M$ in this lemma is bounded by
\begin{equation*}
\sum_{j=0}^k |\ty_j| \leq \sum_{j=0}^k |(\bd)_j| + \sum_{j=0}^k |\ty_j-(\bd)_j|  \leq 2\sqrt{L} + \sqrt{k} \| \ty  - \bd\| = O(1),
\end{equation*}
since  $\sum_{i=0}^k |(\bd)_i| \leq 2\sqrt{L}$ (by Lemma \ref{bdLemma}) and  $\| \ty  - \bd\| \leq \delta_n$. Consequently, Lemma  2.3 (supplement) implies that
\begin{equation} \label{BsBound}
\| B_{\s(i)}(d,\ty)\| \leq K,
\end{equation}
where $K $ depends only on $L, L_o$ and not on $n$, $d$ nor $\ty$.  From the relations in \eqref{matrixCalculus} and the definition of $B_\s$ it follows that for any $\s, d, \ty$,
\[|X^t B_{\s}(d,\ty) X| \leq X^t T_n^{-1}(f_o) X K^{|\s|} \| T_n^{\ft}(f_o)T_n^{-\ft}(f_{d,k})\|^2 \leq X^t T_n^{-1}(f_o) X K^{|\s|} n^{\e},
\]
\[| \tr\left[ B_{\s}(d,\ty) \right]| \leq n K^{|\s|}\| T_n^{\ft}(f_o)T_n^{-\ft}(f_{d,k})\|^2 \leq n^{1+\e} K^{|\s|}.
\]
Therefore we have the bound
\begin{equation}\label{Rbound3}
|R_{J+1}(d,\ty)| \leq C K^{J+1} n^{\e} (\sqrt{k}\|\ty-\bd\|)^{J+1}\left(X^t T_n^{-1}(f_o) X + n\right).
\end{equation}
Since $k \leq k_n$, $\|\ty-\bd\|\leq \dn$ and  the term $X^t T_n^{-1}(f_o) X$ in \eqref{Rbound3} is the sum of $n$ independent standard normal variables, there is a constant $c>0$ such that
\begin{equation*} \label{Rbound2}
P_o\left( \sup_{|d-d_o|\leq \vb}\sup_{\|\ty -\bd\|\leq 2l_0\dn} |R_{J+1}(d,\ty )| >n^{-\e} \right) \leq e^{-cn},
\end{equation*}
provided we choose $J$ such that $(J+1)(1-1/\be)>2$. This concludes the proof of \eqref{eq:taylorLem}.

To prove \eqref{eq:taylorLem2} we first show that for $J=2$, $|R_{J+1}(d,\ty )|=\op(1)$.
Since $k \leq k_n'$, $\be > 5/2$ and $\| \ty - \bd \| \leq 2l_0\en$, we can choose
$J+1=3 > (2\be+1)/(\be-\ft)$, and the preceding inequality becomes
$$P_o\left( \sup_{|d-d_o|\leq w_n}\sup_{\|\ty -\bd\|\leq 2l_0\en} |R_3 (d,\ty )| >n^{-\e} \right) \leq e^{-cn}.$$
Combining this result with \eqref{taylorDif}, it only remains to be shown that
$(d-d_o)g_{n,1}(u)$ and $(d-d_o)g_{n,2}(u)$ are $\op(1)$.
Recall from Corollary \ref{cor1} that $|d-d_o| = o(n^{\e -(\be - 1/2)/(2\be + 1)}) $ for all $\e>0$. Consequently,
$$ |d-d_o| n^{\ft}k^{-\ft} = o(n^{\e + 1/(2\be + 1)} k^{-\ft})  = o((\sqrt{k}\en)^{-1} ) , \quad (d-d_o)^2 \frac{n}{k} = o( (\sqrt{k}\en)^{-1} )$$
for all $\be >2$. This implies that $S_n(d)=\op(1)$ and that, by Lemma \ref{lem:g1}, $(d-d_o) g_{n, 1}(u)  = \mathbf{o}(1)$.
Also, for all  $ (l_1,l_2) \in \{ 1,\ldots,k\}^j$ and all partitions $\sigma$ of $(l_1,l_2)$,
 the limiting integral of $\tr\left[T_{1,\sigma}\right] $ is equal to 0. Since $\be > 5/2$ the Lipschitz constants of the functions $f_{d_o,k}$ or $f_o$ are $O(1)$, so that Lemma 2.4 (supplement) implies
 $\tr\left[T_{1,\sigma}(d_o,k)\right]  = O(n^\e k)$. Similarly,
 $$\tr\left[T_{2,\sigma}(d_o,k)\right] = \frac{ n}{ 2\pi } \int_{-\pi}^{\pi} H_k(x) \cos(l_1x) \cos(l_jx) dx + O(n^\e k).$$
Thus we have
\begin{equation*}
 (d-d_o) g_{n, 2}(u) = \frac{ n(d-d_o)}{ 2\pi } \sum_{l_1,l_2=0}^k u_{l_1}u_{l_2}\int_{-\pi}^\pi H_k(x) \cos(l_1x) \cos(l_jx) dx + \mathbf{o}(1),
\end{equation*}
which is $\mathbf{o}(1)$. This completes the proof of Lemma \ref{taylorLemma}.
\end{proof}

%

The proof of the following lemma is given in section 4 of the supplement.
\begin{lem} \label{mixedderiv}
Let $W_{\s}(d)$ denote any of the quadratic forms
\[  X^t T_n^{-1}(f_{d,k})B_{\s}(d,\bd) X - \tr\left[T_n(f_o) T_n^{-1}(f_{d,k})B_{\s}(d,\bd)\right]\] \noindent
in \eqref{dThetaExpression}.
For any $j \leq J$, $(l_1,\ldots,l_j)\in\{0,\ldots,k\}^j$ and $\s \in \mathcal S_j$, we have
\begin{equation} \label{stochTerm}
 |W_\s(d) - W_\s(d_o)| = \op(|d-d_o| n^{\ft +\e} k^{-\ft}),
\end{equation}
\begin{equation} \label{detTerm1}
\begin{split}
& \tr\left[ B_{\s}(d,\bd)\right] - \tr\left[ B_{\s}(d_o,\bar{\ty}_{d_o})\right] \\
&\quad = (d-d_o) \tr[T_{1,\s}(d_o,k)] + (d-d_o)^2 \od( n^{\e + \ft} k^{-\ft + (1 - \be /2)_+})\\
&\quad = (d-d_o) \tr[T_{1,\s}(d_o,k)] + (d-d_o)^2 \od( n^{1 -\delta} /k),
\end{split}
\end{equation}
\begin{equation}\label{detTerm2}
\begin{split}
&  \tr\left[( T_n(f_o) T_n^{-1}(f_{d,k}) - I_n)B_{\s}(d,\bd) \right]-\tr\left[( T_n(f_o) T_n^{-1}(f_{d_o,k}) - I_n)B_{\s}(d_o,\bz) \right] \\
&\quad= (d-d_o) \tr[T_{2,\s}(d_o,k)] + (d-d_o)^2 \od( n/k)+(d-d_o) \od( n^{\e + \ft}k^{-\ft}).
\end{split}
\end{equation}
\end{lem}

\begin{lem} \label{lem:g1}
For all $\be >1$ there exists a constant $\delta>0$ such that uniformly over $\|\ty - \bd\| \leq \delta_n $,
$$|g_{n,1}(\ty - \bd ) | = \mathbf{o}(n^{1/2-\delta} k^{-1/2}).$$
\end{lem}
\begin{proof}
For $u = \ty - \bd$, we have
$$g_{n,1}(u) = -\ft \tr\left[ T_n(H_k f_{d_o, \bz})T_n^{-1}(f_{d_o, \bz})T_n( u^t \nabla f_{d_o,\bz}) T_n^{-1}(f_{d_o, \bz})\right].$$
This follows from \eqref{gnDecomposition} and Lemma  below, by which $b_\s=d_\s=\ft$ and $c_\s=0$ (the only partition for $j=1$ being $\s=(\{l\})$).
By Lemma 2.4 (supplement) $g_{n,1}(u)$ converges to zero, but at a rate slower than $n^{1/2-\delta} k^{-1/2}$. To obtain the $\mathbf{o}(n^{1/2-\delta} k^{-1/2})$ term, we write
\begin{equation*}
g_{n,1}(u) = \Delta_1 + \Delta_2 + \Delta_3,
\end{equation*}
and bound the terms on the right using the other lemmas in section 2 of the supplement.  We first prove that
 \begin{equation*} 
\begin{split}
\Delta_1 &=  tr\left[ T_n(H_k f_{d_o,k}) T_n(f_{d_o,k}^{-1}) T_n( u^t \nabla f_{d_o,k})T_n(f_{d_o,k}^{-1})\right] -  \\
& \qquad (16 \pi^4)  tr\left[ T_n(H_k ) T_n( u^t \mathbf{cos} )  \right] =  \mathbf{o}(1),
\end{split}
\end{equation*}
where    $\mathbf{cos}(x) = ( 1, \cos(x),\ldots,\cos(kx))$.
   We then prove that
\begin{eqnarray*} 
\Delta_2 &=& tr\left[ T_n(H_k ) T_n( u^t \mathbf{cos} )  \right] = 0,
 \end{eqnarray*}
 and finally that
 \begin{eqnarray*} 
 \Delta_3 &=& tr\left[ T_n(H_k f_{d_o,k}) T_n^{-1}(f_{d_o,k}) T_n( u^t \nabla f_{d_o,k})T_n^{-1}(f_{d_o,k})\right]  \nonumber \\
  & & \quad - tr\left[ T_n(H_k f_{d_o,k}) T_n\left(\frac{f_{d_o,k}^{-1}}{ 4 \pi^2 }\right) T_n( u^t \nabla f_{d_o,k})T_n\left(\frac{f_{d_o,k}^{-1}}{ 4 \pi^2 }\right)\right] \nonumber \\
   & = &  \mathbf{o}( n^{1/2-\delta}k^{-1/2}).
 \end{eqnarray*}
To bound $\Delta_1$ we use Lemma 2.5 (supplement) with $b_1(x) = H_k(x)$, $b_2(x) = u^t \mathbf{cos}$ and $L= k^{3/2-\be}$. Equation (2.6) 
then implies that
\begin{eqnarray*}
|\Delta_1| &\leq & C \sqrt{k}\|u\| n^\e \left( 1 + k^{3/2-\be}k^{-1/2}\right) = \mathbf{o}(1).
\end{eqnarray*}
To bound $Delta_2$ note that for $l=0,\ldots,k$ and all $j_1,j_2 \leq n$,
 \begin{eqnarray*}
 (T_n(\cos (lx) ))_{j_1,j_2} &=& \1_{|j_1 - j_2 |= l},
\quad  (T_n(H_k) )_{j_1,j_2} = \sum_{j=k+1}^n \eta_j \1_{ j = |j_1 - j_2|}.
 \end{eqnarray*}
Therefore
\begin{eqnarray*}
\tr\left[ T_n(H_k)T_n(\cos(l.)\right] = \sum_{j_1=1}^n \sum_{j_2=1}^n \sum_{j=k+1}^n \eta_j \1_{ j = |j_1 - j_2|}\1_{|j_1 - j_2 |= l} &=& 0,
\end{eqnarray*}
since $l\leq k$ and $j>k$. We now turn to $\Delta_3$.
Following \citet{LRR09}, we consider separately the positive and negative parts of $H_k$ and of $u^t \mathbf{\cos} $. Hence we may treat these functions as if they were positive. We first define, for $\tilde f_{d_o,k}=(4\pi^2 f_{d_o,k} )^{-1}$,
\begin{equation*}
\begin{split}
& A_1 = T_n(H_k f_{d_o,k})T_n^{-1}(f_{d_o,k}), \quad B_1 =T_n(H_k f_{d_o,k}) T_n(\tilde f_{d_o,k}), \\
&  A_2  = T_n( u^t \nabla f_{d_o,k})T_n^{-1}(f_{d_o,k}), \quad B_2 = T_n( u^t \nabla f_{d_o,k}) T_n(\tilde f_{d_o,k}), \\
& \tilde{ A} = T_n^{\ft}(H_k f_{d_o,k}) T_n^{-1}(f_{d_o,k}) T_n^{\ft}( u^t \nabla f_{d_o,k}), \\
& \tilde{B} =T_n^{\ft}(H_k f_{d_o,k}) T_n(\tilde f_{d_o,k}) T_n^{\ft}( u^t \nabla f_{d_o,k}), \\
& \Delta = I_n - T_n(f_{d_o,k}) T_n(\tilde f_{d_o,k}).
\end{split}
\end{equation*}
Using the same computations as in \citet{LRR09}, we find that
 \begin{eqnarray*}
| \Delta_3| &\lesssim& |\tr\left[B_1 B_2 \Delta\right]| + \fr\tilde{A}-\tilde{B}\fr \fr T_n^{\ft}(H_k f_{d_o,k})T_n(\tilde f_{d_o,k}) \Delta T_n^{\ft}( u^t \nabla f_{d_o,k})\fr  \\
 & & + |\Delta|^2 \sqrt{k}\|u\| n^\e\\
  &\lesssim & \sqrt{k}\|u\| n^\e k^{3/2-\be } + |\tr\left[B_1 B_2 \Delta\right]|.
 \end{eqnarray*}
 The first term on the right is $\mathbf{o}( n^{1/2-\delta}k^{-1/2})$. We bound the last term using Lemma 2.5 (supplement) with $b_1 = H_k$, $b_2  = u^t \mathbf{cos}$ and $b_3 = 1$,
which implies that $\tr\left[B_1 B_2 \Delta\right] = 0 + O( \sqrt{k}\|u\|n^\e k^{(3/2-\be)_+}) = o(1)$.
This achieves the proof of Lemma \ref{lem:g1}.
\end{proof}


\begin{lem} \label{scoreLemma}
Suppose that $k\leq k_n$ and that $l(f_o,f_{d_o,k}) \leq l_0^2\dn^2$. Then all elements of $\nabla_l l_n(d_o,k)$ ($l= 0,\ldots,k$) are the sum of a centered quadratic form, $\Sto ( \nabla_l l_n(d_o,k)) $ with a variance equal to $\frac{n}{2}(1+o(1))$ and a deterministic term, $\Det(\nabla_l l_n(d_o,k))$  which is $o(k^{(3/2 - \be)_+ } n^\e)$.
\end{lem}
\begin{proof}
For all   $l= 0,\ldots,k$, we have
\begin{equation*} 
\nabla_l l_n(d_o,k) = \Sto ( \nabla_l l_n(d_o,k)) +\Det(\nabla_l l_n(d_o,k)),
\end{equation*} \noindent
where
\begin{equation*}
\begin{split}
 \Sto ( \nabla_l l_n(d_o,k)) &=\ft X^t T_n^{-1}(f_{d_o,k})T_n(\nabla_l f_{d_o,k})T_n^{-1}(f_{d_o,k})X  \\
  &\quad - \ft \tr\left[T_n(f_o) T_n^{-1}(f_{d_o,k})T_n(\nabla_l f_{d_o,k}) T_n^{-1}(f_{d_o,k}) \right],
  \end{split}
  \end{equation*}
$$\Det(\nabla_l l_n(d_o,k))=\ft \tr\left[\left(T_n(f_o) T_n^{-1}(f_{d_o,k}) - I_n\right) T_n(\nabla_l f_{d_o,k}) T_n^{-1}(f_{d_o,k}) \right].$$
Note that this is a special case of \eqref{dThetaExpression}, with $j=1$, $b_\s=d_\s=\ft$ and $c_\s=0$, the only partition being $\s=(\{l\})$. The variance of $\Sto ( \nabla_l l_n(d_o,k))$ is equal to
\begin{equation*} 
\begin{split}
& \tr[(T_n(f_o) T_n^{-1}(f_{d_o,k}) T_n(\nabla_l f_{d_o,k}) T_n^{-1}(f_{d_o,k}))^2] \\
& \quad = \frac{n}{2 \pi} \int_{-\pi}^\pi \left(\frac{f_o}{f_{d_o,k}}(x)\right)^2 \cos^2(l x) dx
+ \Od(n^\e k),
\end{split}
\end{equation*} \noindent
since Lemma  2.4 (supplement) implies that the approximation error of the trace by its limiting integral  is of order $O(n^\e (k+ k^{2 (3/2-\beta)\vee 0})= O(n^\e k)$.
Since $\frac{f_o}{f_{d_o,k}} = e^{\Delta_{d_o,k}}$ (see \eqref{f0fd0bound}), the integral in the preceding equation is
\begin{equation*}
\begin{split}
& \frac{n}{2 \pi} \int_{-\pi}^\pi \left(1 +2\Delta_{d_o,k} + O(\Delta_{d_o,k}^2) \right) \cos^2(l x) dx \\
& \quad =\frac{ n}{ 2}+  2n a_{2 l}(d_o) + O(n\delta_n^2) = \frac{ n }{2} ( 1 + \od(1)),
\end{split}
\end{equation*} \noindent
where $a_l$ is defined at the beginning of the supplement. Lemma 1.3 (supplement) then implies that the centered quadratic form is of order $\op(n^{\e +1/2})$.
 Similarly, Lemma 2.4 (supplement) implies that
\begin{equation*}
\begin{split}
\Det(\nabla_l l_n(d_o,k)) 
&= \frac{n}{ 2\pi} \int_{-\pi}^{\pi}\frac{(f_o - f_{d_o,k})}{f_{d_o,k}}(x) \cos (l x) dx + \Od(n^\e \| \Delta_{d_o,k}\|_1k ) \\
& =  \frac{n}{2 \pi} \int_{-\pi}^{\pi} \cos(l x)\Delta_{d_o,k}(x) dx + \Od(nk^{-2\be} ) + \Od(n^\e k^{3/2 - \be})\\
&=  \Od(n^\e k^{(3/2 - \be)_+})
\end{split}
\end{equation*}
which completes the proof of Lemma \ref{scoreLemma}.
\end{proof}
\begin{lem} \label{fisherLemma}
Let $A(d)$ be the $(k+1)\times(k+1)$ matrix with entries $A_{l_1,l_2}(d) = a_{l_1 + l_2}(d)$, where
$a_l(d)= 1_{l>k} (\ty_{o,l}- 2 l^{-1}(d_o-d))$. 
Suppose that $k\leq k_n$ and that $l(f_o,f_{d,k}) \leq l_0^2\dn^2$. Then $J_n(d,k)= -\nabla^2 l_n(d,k,\ty)\Big|_{\ty=\bd}$ satisfies
\begin{equation} \label{fisher1}
 \forall l_1, l_2 \leq k, \quad \fr (J_n(d,k) - J_n(d_o,k))_{l_1,l_2}\fr =  \op( |d-d_o|n^\e k + S_n(d)) = \op(n/k) 
\end{equation}
uniformly over $d \in (d_o - \vb, d_o + \vb)$ and $k \leq k_n$. We also have  for all  $l_1,l_2$
\begin{equation} \label{fisher2}
[J_n(d_o,k) - \frac{n}{2 }I_{k+1} - \frac{n}{2} A(d_o)]_{l_1,l_2} := n(R_{2s})_{l_1,l_2} + n(R_{2d})_{l_1,l_2},
\end{equation}
where $(R_{2s})_{l_1,l_2}$ is a centered quadratic form of order $\op(n^{-1/2+\e})$ and $(R_{2d})_{l_1,l_2}$ is a deterministic term of order $o(kn^{\e-1})$.
For the matrix $A$, we have $\| A(d_o)\| = o(1)$ and $\fr A(d_o)\fr = O(1)$.
\end{lem} \noindent
In particular, \eqref{fisher2} implies that $\fr J_n(d_o,k) - \frac{n}{2 }I_{k+1} - \frac{n}{2} A(d_o)\fr = \op(kn^{1/2+\e}) + o(k^2n^\e)$.
\begin{proof}
Let $d$  and $k\leq k_n$  be such that $l(f_o,f_{d,k}) \leq l_0^2\dn^2$ so that $d \in (d_o-\vb,d_o+\vb)$ (see Corollary
\ref{cor1}). Lemma \ref{taylorLemma} implies that for all $l_1, l_2 \leq k$,
 \begin{equation*}
\begin{split}
& (J_n(d,k))_{l_1, l_2} - (J_n(d_o,k))_{l_1,l_2} \\
& : = - (d-d_o) \sum_{\sigma \in \mathcal S(l_1,l_2)}\left(  c_\sigma \tr\left[ T_{1,\sigma}(d_o,k) \right]+ d_\sigma \tr\left[ T_{2,\sigma}(d_o,k) \right] \right)+ \Op(S_n(d)).
\end{split}
\end{equation*}
Lemma 2.4 (supplement) implies that
$$\tr\left[ T_{i,\sigma}(d_o,k) \right] = O(k n^\e ), \quad i=1,2$$
so that \eqref{fisher1} is satisfied since  this term is $\op(n^{1-\delta}/k)$. We then use expression (\ref{dThetaExpression}), with $\sigma \in \{ (\{1\}, \{2\}), (\{1,2\})\}$ and we denote $\sigma_1$ and $\sigma_2$ the first and the second partition respectively. Note that $c_{\sigma_1} = d_{\sigma_2} = 1/2 $, $c_{\sigma_2} = 0 $ and $d_{\sigma_1} = 1$.   From Lemma 2.4 (supplement), the quadratic form in $(J_n(d_o,k))_{l_1,l_2}$ is associated to a matrix whose Frobenius-norm is $O(\sqrt{n})$ and whose spectral norm is $O(n^\e)$. Hence, this quadratic form is $\op(n^{1/2+\e})$. Also by Lemma 2.4 (supplement), the deterministic terms can be written as
\begin{equation*}
\begin{split}
& \frac{ n }{ 4\pi }  \int_{-\pi}^\pi \cos (l_1x)\cos( l_2x) \left( 1 + \Delta_{d_o,k} \right) (x) dx + o(kn^\e)  \\
& \qquad = \frac{ n}{2 } \left(\1_{l_1=l_2} + a_{l_1+l_2}(d_o) \right) + o(kn^\e),
\end{split}
\end{equation*}
and Lemma \ref{fisherLemma} is proved.
\end{proof}
%
%
%
%
%
%

\section{Proof of Lemma 3.5} \label{lem:Idk:th1:Details} 

Under the conditions of Theorem \ref{BVMtheorem} we have $k = k_n$ and $\be > 1$, and we may assume  (by Lemma \ref{rateLemma1})
that $l(f_o,f_{d,k,\ty}) \leq l_0^2 \dn^2$.
Fixing $d$ and $k$, we develop $\ty \rightarrow l_n(d,k,\ty)$ in $\bd$.
From Lemma \ref{taylorLemma} in Appendix \ref{app:deriv:theta} it follows that
\begin{equation} \label{lnDevelopment}
\begin{split}
l_n(d,k,\ty) - l_n(d,k) &= \sum_{j=1}^J \frac{(\ty - \bd)^{(j)}\nabla^j l_n(d_o,k) }{j !} \\
& \quad + (d-d_o) \sum_{j=2}^J \frac{g_{n,j}(\ty - \bd)}{ j!} + S_n(d),
\end{split}
\end{equation}
where $S_n(d)$ is as in \eqref{Sproperty}.
 Substituting \eqref{lnDevelopment} in the definition of $I_n(d,k)$ in \eqref{InDef}, we obtain
\begin{equation} \label{In_integral}
\begin{split}
& I_n(d,k) = \int_{\| \ty - \bd\| \leq 2 l_0\dn} e^{l_n(d,k,\ty) - l_n(d,k)} d\pi_{\ty|k}(\ty) \\
&\; = e^{S_n(d)} \pi_{\ty|k}(\bd) \int_{\| u\|\leq 2 l_0\dn } e^{\sum_{j=1}^J \frac{1}{j!} u^{(j)}\nabla^j l_n(d_o,k)
+(d-d_o) \sum_{j=2}^J \frac{g_{n,j}(u)}{j!} 
 + h_ku } d u \\
&\; = e^{S_n(d)} \pi_{\ty|k}(\bz) \int_{\| u\|\leq 2 l_0\dn } e^{\sum_{j=1}^J \frac{1}{j!} u^{(j)}\nabla^j l_n(d_o,k)
+ (d-d_o)  \sum_{j=2}^J \frac{g_{n,j}(u)}{j!} 
 + h_ku } d u.
\end{split}
\end{equation}
The first equality  follows from the definition of $I_n(d,k)$ and Lemma \ref{bdLemma}, by which we may replace the domain of integration by $\{\ty : \| \ty - \bd\| \leq 2 l_0\dn\}$.
The second equality  follows from the assumptions on $\pi_{\ty|k}$ in \textbf{prior A}, the transformation $u= \ty - \bd$ and substitution of \eqref{lnDevelopment}.
Also the third equality  follows from the assumptions on $\pi_{\ty|k}$:
these imply that
\begin{equation*}
\left| \log \pi_{\ty|k}(\bd) - \log \pi_{\ty|k}(\bz) \right| = |d_o-d| |h_k^t \eta_{[k]}| + o(1) = O(|d-d_o|(n/k)^{\ft-\e}) + o(1),
\end{equation*}
for some $\e >0$.
Thus, the factor $e^{S_n(d)}\pi_{\ty|k}(\bd)$ on the second line of \eqref{In_integral} may be replaced by $e^{S_n(d)}\pi_{\ty|k}(\bz)$.
Because $S_n(d_o) = \mathbf{o_{P_o}}(1)$, \eqref{In_integral} implies that
\begin{equation} \label{Indzero}
I_n(d_o,k) = (1 + \op(1)) \int_{\| u\|\leq 2l_0\dn }\exp\left\{ h_k u + \sum_{j = 1}^J \frac{u^{(j)} \nabla^j l_n(d_o,k)}{j!}  \right\} du.
\end{equation}
The most involved part of the proof is to establish the bounds
\begin{equation} \label{integralBounds}
\begin{split}
& \pi_{\ty|k}(\bz) \int_{\| u\|\leq l_0\dn }\exp\left\{ h_k u + \sum_{j = 1}^J \frac{u^{(j)} \nabla^j l_n(d_o,k)}{j!}  \right\} du \\
&\leq I_n(d,k) = e^{S_n(d)} \pi_{\ty|k}(\bz) \int_{\| u\|\leq 2 l_0\dn } e^{\sum_{j=1}^J \frac{1}{j!} u^{(j)}\nabla^j l_n(d_o,k) + (d-d_o)  \sum_{j=2}^J \frac{g_{n,j}(u)}{j!}
+ h_ku } d u \\
& \leq \pi_{\ty|k}(\bz) \int_{\| u\|\leq 3 l_0\dn }\exp\left\{ h_k u + \sum_{j = 1}^J \frac{u^{(j)} \nabla^j l_n(d_o,k)}{j!} \right\} du.
\end{split}
\end{equation}
Since the posterior distribution of $\ty$ conditional on $k = k_n$ and $d = d_o$ concentrates at $\bz$ at a rate bounded by $l_0 \delta_n$ (this follows from Lemma \ref{rateLemma1}, with the restriction to $d=d_o$),
the left- and right-hand side of \eqref{integralBounds} are asymptotically equal, up to a factor $(1+\op(1))$. By \eqref{Indzero}, the left- and right-hand side are actually equal to $I_n(d_o,k)$. This implies that $I_n(d,k)  =  e^{S_n(d)} I_n(d_o,k)$, which is the required result.

In the remainder we prove \eqref{integralBounds}. To do so we construct below a change of variables $v = \psi(u)$, which satisfies
\begin{equation}\label{changevar1}
\begin{split}
h_k v + \sum_{j = 1}^J \frac{v^{(j)} \nabla^j l_n(d_o,k)}{j!} &= h_k u + \sum_{j = 1}^J \frac{u^{(j)} \nabla^j l_n(d_o,k)}{j!} \\ &\quad + (d-d_o)  \sum_{j=2}^J \frac{g_{n,j}(u)}{j!}+ O(S_n(d)),
\end{split}
\end{equation}
for all $\|u\| \leq 2 l_0 \delta_n$.
We first define the notation required in the definition of $\psi$ in \eqref{psiDef} below. Recall
from \eqref{gnDecomposition} in Lemma \ref{taylorLemma} that $g_{n,j}(u)$ can be decomposed as
\begin{equation*}
g_{n,j}(u) =  n\sum_{\s \in \mathcal S_j} (c_\s - d_\s) \sum_{l_1,\ldots,l_j=0}^k u_{l_1}\ldots u_{l_j} g_{l_1,\ldots,l_j}^{(j)}, \qquad j=2,\ldots J,
\end{equation*}
where $ g_{l_1,\ldots,l_j}^{(j)}$ depends on $\s$. For ease of presentation however we omit this dependence in the notation.
Using Lemma 2.4 (supplement) and \eqref{gnDecomposition} in Lemma \ref{taylorLemma}, it follows that
for all $j\geq 2$ and $(l_1,\ldots,l_j) \in \{0,\ldots,k\}^j$,
\begin{equation} \label{gexpression}
g_{l_1,\ldots,l_j}^{(j)} =  \gamma^{(j)}_{l_1,\ldots,l_j} + r_{l_1,\ldots,l_j}^{(j)},
\end{equation}
\begin{equation*}
\gamma^{(j)}_{l_1,\ldots,l_j} = \frac{1}{2 \pi} \int_{-\pi}^\pi H_k(x) \cos(l_1 x) \cdots \cos(l_j x) dx.
\end{equation*}
Let $\bar G^{(2)}$ denote the matrix with elements $\gamma_{l_1,l_2}^{(2)}$, and $G^{(2)}$ the matrix with elements $g_{l_1,l_2}^{(2)}$. By direct calculation it follows that
\begin{equation} \label{gamma2entries}
\gamma_{l_1,l_2}^{(2)} = \1_{l_1+l_2>k} \frac{1}{2(l_1+l_2)}.
\end{equation}
Similarly, for all $j\geq 3$ and $l_1,\ldots,l_j \in \{0,1,\ldots,k\}$ we define
\begin{equation*}
(G^{(j)}(u))_{l_1,l_2} = \sum_{l_3,\ldots,l_j=0}^k g^{(j)}_{l_1,\ldots,l_j} u_{l_3}\ldots u_{l_j}, \;
(\bar G^{(j)}(u))_{l_1,l_2} = \sum_{l_3,\ldots,l_j=0}^k \gamma^{(j)}_{l_1,\ldots,l_j} u_{l_3} \ldots u_{l_j}.
\end{equation*}
In contrast to $G^{(2)}$ and $\bar G^{(2)}$, $G^{(j)}(u)$ and $\bar G^{(j)}(u)$ depend on $u$. For notational convenience we will also write $G^{(2)}(u)$ and $\bar G^{(2)}(u)$.
Finally, let  $\tilde{I}_{k} = J_n(d_o,k)/n$ be the normalized Fisher information.

We now define the transformation $\psi$:
\begin{equation}\label{psiDef}
\psi(u) = (I_{k+1} - (d-d_o) D(u))u, \quad \mbox{ with }
\end{equation}
\begin{equation*}
D(u) = (\tilde{I}_{k} + L(u))^{-1} G^t(u), \quad G(u)=\sum_{j=2}^J \frac{1}{j!}\sum_{\s \in \mathcal S_j} (c_\s - d_\s)G^{(j)}(u),
\end{equation*}
\begin{equation} \label{L_def}
(L(u))_{l_1,l_2} =-
\sum_{j=3}^J \frac{1}{n(j-1)!} \sum_{l_3,\ldots,l_j=0}^k u_{l_3}\ldots u_{l_j} \nabla_{l_1,\ldots,l_j} l_n(d_o,k).
\end{equation}
The construction of $G(u)$ is such that
\begin{equation}\label{Gconstruction}
n u^t G(u) u = \sum_{j=2}^J g_{n,j} (u).
\end{equation}
Analogous to $G(u)$ and $ D(u)$
we define $\bar G(u)=\sum_{j=2}^J \frac{1}{j!} \sum_{\s \in \mathcal S_j} (c_\s - d_\s) \bar G^{(j)}(u)$ and
$\bar D(u) = (\tilde{I}_{k} + L(u))^{-1} \bar G^t(u)$.
 After substitution of $v= \psi(u)$, and using \eqref{diffj} in Lemma \ref{lem:normes} it follows that 
\begin{equation*}
\begin{split}
 & \sum_{j=3}^J \frac{ (v^{(j)}  - u^{(j)}) \nabla^j l_n(d_o,k) }{ j! }  = \\
 &\quad - (d-d_o) \sum_{j=3}^J \sum_{l_1,\ldots,l_j=0}^k (D(u)u)_{l_1}u_{l_2}\ldots u_{l_j}\frac{ \nabla_{l_1,\ldots,l_j} l_n(d_o,k) }{ (j-1)! } + \mathbf{O}(S_n(d)).
\end{split}
\end{equation*}
The definitions of $D(u)$ and $L(u)$ and \eqref{Gconstruction} imply that
\begin{equation*}
\begin{split}
& - n(v-u)^t \tilde{I}_k u =  n(d-d_o) u^t D^t(u) \tilde{I}_k u
 = n(d-d_o) u^t G(u) (\tilde{I}_k + L(u))^{-1} \tilde{I}_k u \\
& =  n(d-d_o) u^t G(u) \left(I_{k+1}- (\tilde{I}_k + L(u))^{-1}L(u)\right)u\\
& = (d-d_o) \sum_{j=2}^J \frac{1}{j!} g_{n,j} (u) - n (d-d_o) (D(u)u)^t L(u) u \\
&= (d-d_o) \sum_{j=2}^J \frac{1}{j!} g_{n,j} (u) -(d-d_o)  \sum_{j=3}^{J-1} \sum_{l_1,\ldots,l_j=0}^k (D(u)u)_{l_1}u_{l_2}\ldots u_{l_j}\frac{ \nabla_{l_1,\ldots,l_j} l_n(d_o,k) }{ (j-1)! }.
\end{split}
\end{equation*}
At the same time, the definition of $\tilde{I}_k$ implies that
\begin{equation*}
 \frac{ 1}{2}\left(  v^{(2)}-  u^{(2)}\right)  \nabla^2 l_n(d_o,k)  = - n(v-u)^t \tilde{I}_k u - \frac{n(v-u)^t \tilde{I}_k (v-u) }{2}.
\end{equation*}
Combining the preceding results, we find that
\begin{equation*}
\begin{split}
 & h_k v + \sum_{j = 1}^J \frac{v^{(j)} \nabla^j l_n(d_o,k)}{j!} - \left( h_k u + \sum_{j = 1}^J \frac{u^{(j)} \nabla^j l_n(d_o,k)}{j!} +(d-d_o)  \sum_{j=2}^J \frac{g_{n,j}(u)}{j!} \right) \\
  &\quad =
  h_k (v-u) + (v-u)^t \nabla l_n(d_o,k)  - \frac{ n (v-u)^t \tilde{I}_{k} (v-u) }{ 2} + O(S_n(d)) \\
&\quad  =   (v-u)^t\nabla l_n(d_o,k)   + O(S_n(d)), 
\end{split}
\end{equation*}
where the last equality follows from \eqref{uminusv} below in Lemma \ref{lem:normes}, together with   the assumption on $h_k$ in \textbf{prior A} in \eqref{priorThetaK}. 

Apart from the term $(v-u)^t\nabla l_n(d_o,k)$ on the last line, the preceding display implies \eqref{changevar1}.
Hence, to complete the proof of \eqref{changevar1} it suffices to show that
\begin{equation} \label{vu_score}
(v-u)^t \nabla l_n(d_o,k)= -(d-d_o) u^t D^t(u) \nabla l_n(d_o,k)=O(S_n(d)).
\end{equation}

The proof of \eqref{vu_score} consists of the following steps:
\begin{eqnarray}
|u^t (D(u) - \bar{D}(u))^t  \nabla l_n(d_o,k) |  &=& \op( n^{\ft -\delta}k^{-\ft}), \label{step1} \\
(d-d_o) u^t \bar D^t(u) \Det\left(\nabla l_n(d_o,k)\right)&=& O(S_n(d)),  \label{step2} \\
(d-d_o) u^t \bar D^t(u) \Sto\left(\nabla l_n(d_o,k)\right)&=& O(S_n(d)) \label{step3},
\end{eqnarray}
where $\Sto\left(\nabla l_n(d_o,k)\right)$ denotes the centered quadratic form in $\nabla l_n(d_o,k)$, and $\Det\left(\nabla l_n(d_o,k)\right)$ the remaining deterministic term.
We will use the same notation below for $L(u)$.

Equation \eqref{step1} follows from  Lemma \ref{scoreLemma} and
\eqref{AminusAbar} in Lemma \ref{lem:normes} below, which imply that the left-hand side equals
$\op((\sqrt{k}\|u\|)^2 n^{-1+\e} k \sqrt{n} ) = \op( n^{\ft -\delta}k^{-\ft})$, for some $\delta>0$. For the proof of \eqref{step2}, note that Lemma \ref{scoreLemma} implies
\[\fr u ^t \bar D^t (u) \Det (\nabla l_n(d_o,k))\fr \lesssim \|\bar D(u) u\| \sqrt{k}k^{(3/2-\be)_++\e}.\]
Combined with Lemma  \ref{lem:normes}, this implies that the left-hand side is $O(\sqrt{k}k^{5/2-2\be+\e})$, which is $O(S_n(d))$.
The proof of \eqref{step3} is more involved. Recall that $\bar D(u)$ is defined as $\bar D(u) = (\tilde{I}_{k} + L(u))^{-1} \bar G^t(u)$. Using \eqref{fisher2} in Lemma \ref{fisherLemma}, we obtain
 $$ (\tilde{I}_k + L(u))^{-1} = 2 [ I_{k+1} - (A(d_o)+ R_{2s} + R_{2d} + L(u) ) (1 + \op(1) )] .$$
Substituting this in $\bar D(u)$, it follows that \eqref{step3} can be proved by controlling $\bar G^{(j)} \Sto(\nabla l_n(d_o))$, $\bar G^{(j)} A(d_o)\Sto(\nabla l_n(d_o))$, $\bar G^{(j)} R_{2d}\Sto(\nabla l_n(d_o))$, $\bar G^{(j)} \Det(L(u))\Sto(\nabla l_n(d_o))$, $\bar G^{(j)} R_{2s}\Sto(\nabla l_n(d_o))$ and $ \bar G^{(j)}\Sto(L(u))\Sto(\nabla l_n(d_o))$ for all $j = 3 ,\ldots,J$.
To do so, first note that  Lemma \ref{fisherLemma} implies that
$\|\bar G^{(j)} R_{2s}\Sto(\nabla l_n(d_o))\| = \op( n^{\e}\sqrt{k})$.
Hence,  $$|u^t \bar G^{(j)} R_{2s}\Sto(\nabla l_n(d_o))| = \op( k^{-\be + 1+\e}) = \op(1),$$
which clearly is $O(S_n(d))$.
The terms $\bar G^{(j)} \Sto(\nabla l_n(d_o))$, $\bar G^{(j)} A(d_o)\Sto(\nabla l_n(d_o))$, $ \bar G^{(j)} R_{2d} \Sto(\nabla l_n(d_o))$ and  $ \bar G^{(j)} \Det(L(u))\Sto(\nabla l_n(d_o))$
can be written as quadratic forms $ Z^t M Z - tr[M] $, where, for a sequence $(b_l)_{l=0}^k$ and a function $g$ with $\| g \|_\infty < \infty $, $M$ is of the form
$$ T_n^{\ft}(f_o) T_n^{-1}(f_{d_o,k}) T_n \left(g(x) f_{d_o,k}(x) \sum_l b_l \cos ( l x) \right) T_n^{-1}(f_{d_o,k})T_n^{\ft}(f_o),$$
$Z$ being a vector of $n$ independent standard Gaussian random variables.
Using Lemma 2.4 (supplement) it can be seen that $\fr M\fr^2 \leq n( \sum_{l} b_{l}^2 + k/n )$. Lemma 1.3 (supplement) with $\az = \e+1/2$ then implies that
 \begin{equation} \label{arg:score}
 P_o\left(| Z^t M Z - tr[M]| > n^{\e+\ft}\left( \sum_{l} b_{l}^2 + \frac{k}{n} \right)^{\ft} \right) \leq e^{- cn^\e}.
 \end{equation}
For all $j \in \{ 3, \ldots, J\}$, the four terms above can now be bounded for a particular choice of $g$ and $b_l$.
\begin{itemize}
\item Bound on $\bar G^{(j)} \Sto(\nabla l_n(d_o))$. For all $l_2,\ldots,l_j \in  \{0,\ldots,k\},$ set  $b_{l}  = \gamma_{l_2, l,l_4,\ldots,l_j}^{(j)}$ and $g(x) = 1$. Then we have
\begin{equation*}
\sum_{l=0}^k b_{l}  \Sto(\nabla l_n(d_o)))_{l_1} = \op\left(n^{\ft+\e}\left(\sum_{l}b_l^2\right)^{\ft}\right).
\end{equation*}
By induction it can be shown that
\begin{equation}\label{cos}
\begin{split}
& \frac{1}{2\pi}\int_{-\pi}^{\pi}\cos(l_0 x) \cos(l_1 x) \cos( l_2 x)\cdots\cos(l_{j} x) dx \\
& \quad =  2^{-j}\sum_{\e_1,\ldots,\e_j \in \{ -1,1\}} \1_{l_0 + \sum_{i=1}^{j} \epsilon_i l_i=0}
\end{split}
\end{equation}
Consequently, $\sum_{l}b_l^2 = O(k^{-1})$ for all $l_2,l_4,\ldots,l_j \in  \{0,\ldots,k\}$. Using the fact that $\sum_l |u_l | = o(1)$, we obtain that  $( \bar G^{(j)} \Sto(\nabla l_n(d_o)))_{l_2} = \op(n^{1/2+\e}k^{-1/2})$, for all $l_2\in \{0,\ldots,k\}$. This implies that
\begin{equation} \label{part1}
\| \bar G^{(j)} \Sto(\nabla l_n(d_o))\| = \op(n^{1/2+\e})=O(S_n(d)).
\end{equation}
\item Bound on $ A(d_o) \Sto(\nabla l_n(d_o))$.  Set  $b_{l} = \1_{l + l_1\geq k} \theta_{o, l+l_1} $  and $g(x) = 1$, then
\begin{equation*}
( A(d_o) \Sto(\nabla l_n(d_o)))_{l_1} = \op(n^{1/2+\e}k^{-\be}), \quad \forall l_1\in \{ 0,\ldots,k\}.
\end{equation*}
Combined with Lemma \ref{lem:normes} this implies that
\begin{equation*}
\| \bar G^{(j)} A(d_o) \Sto(\nabla l_n(d_o))\| = \op(n^{1/2+\e}k^{-\be+1/2}).
\end{equation*}
\item Bound on $ \bar G^{(j)} R_{2d}\Sto(\nabla l_n(d_o) $. Set $b_{l}   = (R_{2d})_{l_1,l}$ and $g(x) = 1$, for all $l_1 = 0,\ldots,k$, then  Lemmas \ref{fisherLemma} and  \ref{lem:normes}  lead to
\begin{equation*}
\| \bar G^{(j)} R_{2d}\Sto(\nabla l_n(d_o)) \| = \op( n^{1/2+\e} k^{2}n^{-1}) = \op(n^{-1/2+\e} k^{2}).
\end{equation*}
\item Bound on $ \bar G^{(j)} \Det(L(u))\Sto(\nabla l_n(d_o))$. For all $l_1,l_3,\ldots,l_j, l_3^{'},\ldots,l_{j^{'}}^{'} \in \{0,\ldots,k\}$, set
\begin{equation*}
\begin{split}
b_{l} & = \fn \tr\left[ T_n^{-1}(f_{d_o})T_n\left( \sum_{l_2=0}^k \gamma_{l_1,l_2,\ldots,l_j}^{(j)} \cos (l_2 x) g_1(x)  f_{d_o}(x)\right)\times \right. \\
 &\qquad  \left. T_n^{-1}(f_{d_o})T_n(\cos (lx) g_2(x) f_{d_o}(x))\cdots T_n^{-1}(f_{d_o})T_n( g_r (x) f_{d_o}(x))\right]
 \end{split}
 \end{equation*}
where $g_1(x), ...,g_r(x)$ are products of functions of the form $\cos (l_i^{'} x)$ and
 $ g_1(x)....g_r(x) = \cos(l_3^{'}x)...\cos(l_{j^{'}}^{'}x)$. Lemmas 2.1 and 2.6  in the supplement, together with \eqref{cos}, imply that
%
%
\begin{equation}\label{part4}
\sum_l b_l^2 = O(k^{-1}), \quad \mbox{and} \quad  \|\bar G^{(j)} \Det(L(u))\Sto(\nabla l_n(d_o))\|= \op(n^{\e +1/2} ).
\end{equation}
\end{itemize}

Consequently, the contribution to all these terms in $(v-u)^t \nabla l_n(d_o, k) $ is of order  $\Od(S_n(d))$.

We control $u^t \bar G^{(j)}  \Sto(L(u))\Sto(\nabla l_n(d_o,k))$, by bounding $\| \bar G^{(j)}\Sto(L(u))\|$   using a similar idea. Indeed, for all $l_1, l_2 \leq k$, $(\bar G^{(j)} \Sto(L(u)))_{l_1, l_2}$  can be written as a sum of terms of the form
$ (Z^t M_{l_1,l_2} Z  - tr(M_{l_1,l_2}))/n$, where $Z$ is a vector of $n$ independent standard Gaussian random variables, and
$M_{l_1, l_2}$ has the form
\begin{equation*}
\begin{split}
&      T_n^{\ft}(f_o) T_n^{-1}(f_{d_o,k}) \left(\prod_{i < i_0} T_n( \nabla_\sigma(i) f_{d_o,k}) T_n^{-1}(f_{d_o,k}) \right)\times \\
& T_n\left( \sum_{l=0}^k \gamma_{l,l_1,l_2,\ldots,l_{j-1}}^{(j)} \cos (lx) \nabla_{\sigma(i_0)-\{l\}} f_{d_o,k}\right) T_n^{-1}(f_{d_o,k})\prod_{i < i_0} T_n( \nabla_\sigma(i) f_{d_o,k}) T_n^{\ft}(f_o).
\end{split}
\end{equation*}
We can use the same argument as in \eqref{arg:score} since for all $l_1,l_2,\ldots,l_{j-1}$
\begin{equation*}
\begin{split}
\fr M_{l_1,l_2}\fr &\lesssim \fr T_n^{-\ft}(f_{d_o,k}   T_n\big( \sum_{l=0}^k \gamma_{l,l_1,l_2,\ldots,l_{j-1}}^{(j)} \cos (lx) \nabla_{\sigma(i_0)-\{l\}} f_{d_o,k}\big) T_n^{-\ft}(f_{d_o,k})\fr\\
 &= O( n^{1/2+\e}k^{-1/2}).
 \end{split}
 \end{equation*}
Hence, it follows that  $ n^{-1}[Z^t M_{l_1,l_2} Z  - tr(M_{l_1,l_2})] = \op( n^{-1/2+\e} k^{-1/2} )$ and
\begin{equation}\label{part5}
u^t \bar G  \Sto(L(u))\Sto(\nabla l_n(d_o,k))= \op(\| u \|  n^{\e}k) = \op( n^{1/2-\delta } k^{-1/2}).
\end{equation}
Combining \eqref{part5} and \eqref{part1}-\eqref{part4}, we obtain \eqref{step3}. This in turn finishes the proof of \eqref{vu_score}, since
\begin{eqnarray*} 
(v-u)^t \nabla l_n(d_o,k) = \op(|d-d_o| n^{\ft -\delta}k^{-\ft }) = O(S_n(d)).
\end{eqnarray*}

We now prove that $\psi(u)$ is a one-to-one transformation.
First note that $\psi(u)$ is continuously differentiable for all $\| u \| \leq 2l_0 \delta_n$. This follows from the definition $\psi(u) = (I_{k+1} - (d-d_o) (\tilde{I}_{k} + L(u))^{-1} G^t(u))u$, the fact that $G(u)$ and $L(u)$ are  polynomial in $u$ and Lemma \ref{lem:normes}, by which $\|L(u)\| = \op(1)$. To prove that $\psi(u)$ is also one-to-one, we bound the spectral norm of the Jacobian
\begin{equation*}
\psi'(u)  = I_{k+1}  - (d-d_o) D(u) - (d-d_o) (D'(u) u),
\end{equation*}
where $(D'(u) u)$ is the $(k+1) \times (k+1)$ matrix with elements
$$ \sum_{l=0}^k u_l \frac{ \partial (D(u))_{l_1,l} }{ \partial u_{l_2} },\qquad l_1,l_2=0,\ldots,k.$$
For $\psi(u)$ to be one-to-one, it suffices to have $\psi'(u) = I_{k+1}( 1 + \op(1))$.

By \eqref{uminusv} in Lemma \ref{lem:normes} below, we have $|d-d_o|\|D(u)\| = \Od_{P_o}(|d-d_o|)$. Therefore we only need to control the spectral norm of  $D'(u)u$. For all $l_1, l_2$, we have
\begin{equation} \label{D'}
(D'(u)u)_{l_1,l_2} =   \left[-(\tilde{I}_k + L(u))^{-1} \frac{ \partial L(u) }{ \partial u_{l_2} } (\tilde{I}_k + L(u))^{-1} G^t(u) u + (\tilde{I}_k + L(u))^{-1}\frac{ \partial G^t(u) }{ \partial u_{l_2} } u \right]_{l_1}.
\end{equation}
Both $(G(u))_{l_1,l_2}$ and $(L(u))_{l_1,l_2}$ can be written as
 \begin{eqnarray*}
 F_{l_1,l_2}(u; \tau, b) := \sum_{j =2}^J \tau_j \sum_{l_3,\ldots,l_j=0}^ku_{l_3}\cdots u_{l_j} b_{l_1,l_2,\ldots,l_j},
 \end{eqnarray*}
where the constants $\tau_j, b_{l_1,\ldots,l_j}$ are different for $G$ and $L$, and $b$ is symmetric in its indices. In particular, $\tau_2 = 0$ in the case of $L$. Using this generic notation for $G(u)$ and $L(u)$, we find that for all $v \in \R^{k+1}$ and all $l_1, l_2 \leq k$,
\begin{equation*}
\begin{split}
\left( \frac{ \partial F(u; \tau, b) v }{ \partial u_{l_2} } \right)_{l_1} & = \sum_{j=3}^J \tau_j (j-3+1) \sum_{l_3,\ldots,l_j=0}^k v_{l_3}u_{l_4}\cdots u_{l_j} b_{l_1,l_2,\ldots,l_j}: = F(v,u; \tau', b),
\end{split}
\end{equation*}
where $\tau_j' = \tau_j (j-3+1)$, $j=3,\ldots,J$. It therefore has the same form as $F(u; \tau',b)$, with $v$ replacing one of the $u$'s. Applying this to the first term of \eqref{D'}, with $v = (\tilde{I}_k + L(u))^{-1} G^t(u) u$, we find that
 $$\fr (\tilde{I}_k + L(u))^{-1} F(v,u; \tau', b)\fr \lesssim \fr F(v,u; \tau', b)\fr  = \Od(1), $$
where we used \eqref{LnormBound} and \eqref{uminusv} from Lemma \ref{lem:normes}.
The second term of \eqref{D'} is treated similarly with $v= u$ so that we finally obtain
  \begin{equation*}
  \fr D'(u)u \fr = \Od(1),
 \end{equation*}
and $\psi$ is one-to-one on $\{u : \|u\| \leq 2 l_0\delta_n\}$. Using the above bounds we also deduce that the Jacobian is equal to $\exp(O( S_n(d)) )$, since
 \begin{equation*}
 \begin{split}
& \log \mbox{det} [\mbox{Jac}] = \log \det \left[ I_{k+1}-(d-d_o) D(u)  - (d-d_o)D'(u)u\right]  \\
\quad & =   O[ (d-d_o) (|\tr[ D(u)]| + \tr[D'(u)u])+ (d-d_o)^2 (|D(u)|^2 + |D'(u)u|^2))] \\
\quad &= O(\sqrt{k}(d-d_o) + k(d - d_o)^2 ) = O(S_n(d)).
\end{split}
\end{equation*}
This finishes the proof  of \eqref{integralBounds}, and hence the proof of Lemma \ref{lem:Idk:th1}.

\begin{lem} \label{lem:normes}
Let $v=\psi(u)$, with $\psi$ as in \eqref{psiDef}.
Under the conditions of Lemma \ref{lem:Idk:th1}, we have
\begin{eqnarray}
\fr L(u)\fr  &=& \op(n^{-1/2+\e}k) = \op(1), \label{LnormBound}\\
\fr G-\bar G\fr  &=&  \op(n^{-1/2+\e}k)  = \op(1), \label{AminusAbar}\\
\fr D(u)\fr &=& \Od_{P_o}(1), \label{Dbound}\\
\|u - \psi(u)\| &\lesssim & |d-d_o| \Od_{P_o}(\|u\|) \label{uminusv},
\end{eqnarray}
and 
\begin{equation} \label{diffj}
\begin{split}
&  \sum_{j=3}^J \frac{ (v^{(j)}  - u^{(j)}) \nabla^j l_n(d_o,k) }{ j! }  \\
&= - (d-d_o) \sum_{j=3}^J \sum_{l_1,\ldots,l_j=0}^k (D(u)u)_{l_1}u_{l_2}\ldots u_{l_j}\frac{ \nabla_{l_1,\ldots,l_j} l_n(d_o,k) }{ (j-1)! } + O(S_n(d)),
\end{split}
\end{equation}
uniformly over $\| u \| \leq 2 l_0 \delta_n$.
\end{lem}

\begin{proof}

We first prove \eqref{LnormBound}.
From \eqref{dThetaExpression}, we recall that $\nabla_{l_1,\ldots,l_j} l_n(d_o,k)$ is the sum of a centered quadratic form $\Sto(\nabla_{l_1,\ldots,l_j} l_n(d_o,k))$ and a deterministic term $\Det(\nabla_{l_1,\ldots,l_j} l_n(d_o,k))$. For all $l_1,\ldots,l_j$, $\Sto(\nabla_{l_1,\ldots,l_j} l_n(d_o,k))$ equals
\begin{equation*}
\begin{split}
 X^t \left(T_n^{-1}(f_{d_o,k}) \sum_{\s\in \mathcal S_j} b_{\s} B_{\s}(d_o)\right) X - \tr\left[ T_n(f_o) T_n^{-1}(f_{d_o,k})\sum_{\s\in \mathcal S_j } b_{\s} B_{\s}(d_o) \right],
\end{split}
\end{equation*}
with $B_{\s}(d_o) := B_{\s}(d_o,\bz)$ as defined in \eqref{bsigma}. Using Lemma 1.3 (supplement) together with \eqref{BsBound} we obtain that for all $l_1,\ldots,l_j$, $\Sto(\nabla_{l_1,\ldots,l_j} l_n(d_o,k))= \op(n^{\ft+\e})$, and its contribution to $\fr L(u)\fr$  is $\op(k (\sqrt{k} \|u\|)^{j-2} n^{-\ft+\e}) = \op(n^{-1/2-\delta}k)$.
The deterministic term in \eqref{dThetaExpression} is
\begin{equation*}
\begin{split}
\Det(\nabla_{l_1,\ldots,l_j} l_n(d_o,k))=\sum_{\s} c_{\s} \tr\left[B_{\s}(d_o)\right] + \sum_{\s} d_{\s} \tr\left[( T_n(f_o) T_n^{-1}(f_{d_o,k}) - I_n)B_{\s}(d_o) \right].
\end{split}
\end{equation*}
We bound the contribution of the first term to $\fr L(u)\fr$; the second term can be treated similarly.
Let $\tilde L(u)$ be the matrix when in \eqref{L_def}  we replace $\nabla_{l_1,\ldots,l_j} l_n(d_o,k)$ by $\sum_{\s} c_{\s} \tr[B_{\s}(d_o)]$. Hence,
\begin{equation} \label{L:entries}
(\tilde L(u))_{l_1,l_2} =
- \sum_{j=3}^J \frac{1}{(j-1)!}\sum_{\s \in \mathcal S_j} c_{\s}   \sum_{l_3,\ldots,l_j=0}^k u_{l_3}\ldots u_{l_j} \frac{\tr\left[B_{\s}(d_o)\right]}{ n},
\end{equation}
\begin{equation} \label{E_sigma}
\textrm{where} \quad \fn \tr\left[B_\sigma(d_o)\right] = \frac{1}{2\pi} \int_{-\pi}^{\pi} \cos (l_1x) \ldots\cos(l_j x) dx + E_\sigma,
\end{equation}
$E_\sigma$ being the approximation error.
For each $\s$ and $j \geq 4$, the contribution of the integral in \eqref{E_sigma} to $(\tilde L(u))_{l_1,l_2}$ is
$O(\int_{-\pi}^{\pi} |u^t \mathbf{cos}|^{j-2} (x) dx) = O(\|u\|^2)$; hence its contribution to $\fr \tilde L(u)\fr $ is $k\|u\|^2 = o(n^{-1/2-\delta}k)$.
For $j=3$, we have
\[\frac{1}{4\pi}  \sum_{l_3=1}^k u_{l_3}  \int_{-\pi}^{\pi} \cos (l_1x) \cos(l_2x) \cos(l_3 x) dx = \frac{1}{2}    \left( u_{l_1+l_2} \1_{l_3 = l_1+ l_2} + u_{|l_1 - l_2|} \1_{l_3 = |l_1 - l_2|}\right),
\]
and the contribution of this term to $\fr \tilde L(u)\fr $ is of order $\sqrt{k} \|u\| = o(n^{-1/2+\e}k)$.
Next we bound the contribution to $\fr \tilde L(u)\fr $ of the error term $E_\s$ in \eqref{E_sigma}.
Note that we can write the last sum in \eqref{L:entries} as
\begin{equation} \label{uBsi}
 \sum_{l_3,\ldots,l_j=0}^k u_{l_3}\ldots u_{l_j} \frac{\tr\left[B_{\s}(d_o)\right]}{ n} =
 \fn \tr\left[ \prod_{i=1}^{p} T_n\left( b_i(x)  f_{d_o,k} \right) T_n^{-1}(  f_{d_o,k} )\right],
\end{equation}
where
\begin{equation}\label{biDefinition}
b_i(x) = (u^t \mathbf{cos}(x) )^{|\s(i)|-\delta_1(i)-\delta_2(i)}\cos(l_1\cdot)^{\delta_1(i)}\cos(l_2\cdot)^{\delta_2(i)} ,
\end{equation}
$\delta_1(i) = \1_{1 \in \s(i)}$, $\delta_2(i) = \1_{2 \in \s(i)}$ and $p = |\s|$.  If $p\leq 3$, then Lemma 2.4 (supplement) implies that
\begin{equation} \label{Esbound}
E_\s = O( (\sqrt{k}\|u\|)^{j-2} n^{\e -1} [k^{2(3/2-\be)_+} + k] ) = o( n^{-1 -\delta}k).
\end{equation}
If $p\geq 4$, then   Lemma 2.6  (supplement)  together with \eqref{uBsi}, with
$$f=f_{d_o,k}, \quad  f_{2i} =b_i  f_{d_o,k}, i\leq |\s|, $$
$L = k^{(3/2- \be)_+}$, $M, m^{-1}  = O(1)$,
 $M^{(i)} = O((\sqrt{k}\|u\|)^{|\s(i)|-\delta_1(i)-\delta_2(i)}$ and $L^{(i)} = O(k(\sqrt{k}\|u\|)^{|\s(i)|-\delta_1(i)-\delta_2(i)})$, leads to
the bound
 \begin{equation*}
 \begin{split}
 &   \sum_{l_3,\ldots,l_j=0}^k u_{l_3}\ldots u_{l_j} \frac{ \tr\left[B_{\s}(d_o)\right] }{ n }
   -  \fn \tr\left[ \prod_{i=1}^{|\s|} T_n( b_i f_{d_o,k} ) T_n\left(  \frac{1}{4\pi^2 f_{d_o,k}} \right)\right] \\
 &\quad     =  \od( k^{(3/4-\be/2)_+}n^{-1/2+\e} \|u\| (\sqrt{k} \| u \|)^{j-3})
    = \od( n^{-1/2-\delta}).
    \end{split}
\end{equation*}
Using Lemma 2.1 (supplement) we finally obtain that
\begin{equation*}
\begin{split}
    \fn \tr\left[ \prod_{i=1}^{|\s|} T_n( b_i f_{d_o,k} )\right]
 - \frac{ 1 }{2\pi} \int_{-\pi}^\pi ( u^t \mathbf{cos})^{j-2}(x) \cos(l_1x) \cos(l_2x) dx = \od(n^{-1-\delta}k). 
 \end{split}
\end{equation*}
Therefore the contribution of the approximation  error $E_\s $ in $\fr \tilde L(u)\fr $ is of order $\mathbf{o}( n^{-1/2-\delta}k)$. Using a similar argument we control the terms in the form $\tr\left[ T_n(f_o) (T_n^{-1}(d_o,k)  -I_n)B_\s(d_o)\right]$ and  \eqref{LnormBound} is proved.

We now prove \eqref{AminusAbar} and bound
\begin{equation*}
(G - \bar G)_{l_1,l_2} = \sum_{j=2}^J \frac{1}{j!} \sum_{\s \in \mathcal S_j} (c_\s - d_\s) \sum_{l_3,\ldots,l_j=0}^k r^{(j)}_{l_1,\ldots,l_j} u_{l_3}\ldots u_{l_j},
\end{equation*}
with $r^{(j)}_{l_1,\ldots,l_j}$ as in \eqref{gexpression}. These are the approximation errors which occur when replacing $\fn \tr[ T_{1,\s} (d_o,k)]$ and $\fn\tr[ T_{2,\s}(d_o,k)]$ by their limiting integrals (see also \eqref{gnDecomposition}). Therefore, for each $\s\in \mathcal S_j$,
$\sum_{l_3,\ldots,l_j=0}^k r^{(j)}_{l_1,\ldots,l_j} u_{l_3}\ldots u_{l_j}$ is a combination of terms of the form
\begin{equation*}
\fn \tr\left[ \prod_{i=1}^{p} T_n\left( b_i(x)  f_{d_o,k} \right) T_n^{-1}(  f_{d_o,k} )\right] -\frac{1}{2\pi} \int_{-\pi}^{\pi} (u^t \mathbf{cos}(x))^{j-2} H_k(x) \cos(l_1x)\cos(l_2x)dx ,
\end{equation*}
with $p \in \{ |\s|, |\s|+1\}$ and the functions $b_i$ defined as in \eqref{biDefinition}  apart from $b_1(x)  = H_k(x) (\sum_{l=0}^k u_l \cos(lx))^{|\s(1)| - \delta_1(1) - \delta_2(1)}\cos(l_1.)^{\delta_1(1)}\cos(l_2.)^{\delta_2(1)}$. Therefore, using the same construction as in \eqref{uBsi}-\eqref{Esbound}, we obtain that
\begin{equation*} 
|(G - \bar G)_{l_1,l_2}| = O( n^{-1/2+\e}), \quad \fr G - \bar G\fr = O( n^{-1/2+\e}k) = o(1).
\end{equation*}

To prove \eqref{Dbound}, we use the just obtained bound on $\fr G-\bar G\fr$, and in addition establish a bound $\fr \bar G \fr $.
We treat each term $G^{(j)}$ in $G(u)=\sum_{j=2}^J \frac{1}{j!}\sum_{\s \in \mathcal S_j} (c_\s - d_\s) G^{(j)}(u)$ separately.
First we show that $\fr \bar G^{(2)}\fr =O(1)$, which follows from definition \eqref{gamma2entries}, by which
\begin{eqnarray*}
\tr\left[ (\bar G^{(2)})^2 \right]  = \sum_{l_1, l_2, l_1+l_2\geq k}^{k} \frac{ 1}{ (l_1+l_2)^2}\leq 1.
 \end{eqnarray*}
Consequently, $\fr \bar G^{(2)} \fr \leq 1$.
For $j \geq 3$, note that for all $0\leq l_1,l_2 \leq k$,
\begin{equation*} 
\begin{split}
\left|\bar G_{l_1,l_2}^{(j)}(u)\right| &= \left|  \sum_{l_3,\ldots,l_j=0}^k \gamma^{(j)}_{l_1,l_2,\ldots,l_j} u_{l_3}\ldots u_{l_j}\right|  \\
&\leq  \sum_{l_4,\ldots,l_j=0}^k |u_{l_4}\ldots u_{l_j}| \int_{-\pi}^\pi |H_k(x) | \left|  \sum_{l_3=0}^k \cos(l_3 x) u_{l_3} \right| dx \\
 &\leq (\sqrt{k}\|u\|)^{j-3} \frac{ \|u \|}{\sqrt{k}}=(\sqrt{k})^{j-4} (\|u\|)^{j-2}.
\end{split}
\end{equation*}
Therefore, $\fr \bar G^{(j)}(u)\fr  \leq k (\sqrt{k})^{j-4} (\|u\|)^{j-2} = (\sqrt{k}\|u\|)^{j-2} = o(1)$, for all $j \geq 3$. Hence $\fr \bar G(u)\fr = \Od(1)$, which combined with $\|(\tilde{I}_k^+ L(u)){-1}\| = \mathbf{O}_{P_o}(1)$ (see Lemma \ref{fisherLemma}) and \eqref{LnormBound}), imply that
\begin{eqnarray*} 
\fr  \bar{D}(u)\fr  =  \fr (\tilde{I}_{k} + L(u))^{-1} \bar G^t(u) \fr = O(1),
\end{eqnarray*}
uniformly over $\| u \| \leq 2 l_0 \delta_n$.
It follows that
\begin{equation*}
\fr D(u)\fr \leq \|(\tilde{I}_{k} + L(u))^{-1}\| \left( \fr  \bar G\fr + \fr G - \bar G \fr \right) = \Od_{P_o}(1).
\end{equation*}
This concludes the proof of \eqref{Dbound}; \eqref{uminusv} directly follows from this result since $\|u-\psi(u)\| \leq |d-d_o| \fr D(u)\fr \|u\|$.
Finally, we prove \eqref{diffj}. We have
\begin{equation*}
\begin{split}
& \sum_{j=3}^J \frac{ (v^{(j)}  - u^{(j)}) \nabla^j l_n(d_o,k) }{ j! }  = -
(d-d_o) \sum_{j=3}^J \sum_{l_1,\ldots,l_j=0}^k (D(u)u)_{l_1}u_{l_2}\ldots u_{l_j}\frac{ \nabla_{l_1,\ldots,l_j} l_n(d_o,k) }{ (j-1)! } \\
 & \qquad + (d-d_o)^2 \sum_{j=3}^J { j \choose 2 } \sum_{l_1,\ldots,l_j=0}^k (D(u)u)_{l_1}(D(u)u)_{l_2}\ldots u_{l_j}\frac{ \nabla_{l_1,\ldots,l_j} l_n(d_o,k) }{ (j-1)! } \\
& \qquad +  \ldots +(-1)^J \sum_{l_1,\ldots,l_J=0}^k (D(u)u)_{l_1}(D(u)u)_{l_2}\ldots (D(u)u)_{l_J}\frac{ \nabla_{l_1,\ldots,l_J} l_n(d_o,k) }{ (J-1)! }
\end{split}
\end{equation*}
 Using the same argument as in the proof of \eqref{LnormBound}, we find that for all for all $j \geq 3$
\begin{equation*}
\begin{split}
& \sum_{l_1,\ldots,l_j=0}^k (D(u)u)_{l_1}(D(u)u)_{l_2}u_{l_3}\ldots u_{l_j}\frac{ \nabla_{l_1,\ldots,l_j} l_n(d_o,k) }{ j! } \\ &\quad = n \int_{-\pi}^{\pi} (D(u) u)^t \mathbf{\cos}(x) )^2 (u^t \mathbf{\cos}(x))^{j-2} dx \\
&\quad \quad + (\sqrt{k}\|u\|)^{j-1}O( \sqrt{n}n^\e (\sqrt{k}\|u\|)  + k + \sqrt{n} \|u\|k^{1/2(3/2-\be)_+}) \\
  &\quad = \mathbf{o} (n^{1-\delta}k^{-1}), \quad \mbox{ for some } \delta >0.
\end{split}
\end{equation*}
Similarly, the higher-order terms in the above expression for $\sum_{j=3}^J \frac{ (v^{(j)}  - u^{(j)}) \nabla^j l_n(d_o,k) }{ j! }$ can be shown to be $O(S_n(d))$, which terminates the proof of Lemma \ref{lem:normes}.
\end{proof}

The rest of the paper corresponds to the suppelmentary material
\section{Technical results} \label{integrals}

Let $\eta_j = -1_{j>0} 2/j$ and recall that
$\bd=\ty_{o[k]} + (d_o-d)\eta_{[k]}$. Let the sequence $\{a_j\}$ be defined as $a_j=\ty_{o,j}+ (d_o-d)\eta_j$ when $j>k$ and $a_j=0$ when $j\leq k$. In addition, define
\begin{eqnarray}
H_k(x) &=& \sum_{j=k+1}^\infty \eta_j \cos (j x), \qquad G_k(x) = \sum_{j=1}^k \eta_j \cos (j x), \label{HkGkDef}\\
\Delta_{d,k}(x) &=& \sum_{j=k+1}^\infty(\ty_{o,j} + (d_o-d)\eta_j) \cos(j x) =  \sum_{j=k+1}^\infty a_j \cos(j x). \label{deltakDef}
\end{eqnarray} \noindent
%
Using this notation we can write
\begin{equation} \label{GkHkRelation}
-2 \log |1-e^{i x}| = -\log(2-2\cos(x)) = G_k(x) + H_k(x),
\end{equation}
\begin{equation} \label{fdDef}
\begin{split}
f_{d,k}(x) &= f_{d,k,\bd}(x) =
f_o (x)\exp\left\{-\sum_{j=k+1}^\infty a_j \cos(j x) \right\} \\ &= f_o(x) e^{-\Delta_{d,k}(x)} = f_o(x) e^{(d-d_o) H_k(x) - \Delta_{d_o,k}(x)}.
\end{split}
\end{equation} \noindent
Given $d,k$ and $\ty_o$, the sequence $\{a_j\}$ represents the closest possible distance between
$f_o$ and $f_{d,k,\ty}$, since
\begin{equation} \label{ajBound}
l(f_o,f_{d,k}) = l(f_o,f_{d,k,\bd}) = \frac{1}{2\pi} \int_{-\pi}^\pi \Delta_{d,k}^2(x) dx = \sum_{j>k} a_j^2 .
\end{equation}
From \eqref{fdDef} it also follows that for all $d$,
\begin{equation}\label{fdkDer}
\frac{\partial}{\partial d} f_{d,k} = H_k f_{d,k}.
\end{equation}
\begin{lem} \label{integralLemma}
When $\ty_o \in \Theta(\be,L_o)$, there exist constants 
such that for any positive integer $k$,
\begin{eqnarray}
k^{-1} \lesssim \int_{-\pi}^\pi H_k^2(x) dx & \lesssim & k^{-1}, \label{i1} \\
\sum_{l>k} |\ty_{o,l}| &=& O(k^{-\be+\ft}), \quad \sum_{l\geq 0} |\ty_{o,l}| = O(1), \label{s1} \\
\int_{-\pi}^\pi \Delta_{d_o,k} (x) H_k(x) dx &=& \sum_{j>k} \eta_j \ty_{o,j} = O\left(k^{-\frac{1+2\be}{2}}\right), \label{i2}\\
\int_{-\pi}^\pi \Delta_{d_o,k}^2 (x) dx &=& \sum_{l>k} \ty_{o,l}^2 = O\left(k^{-2 \be}\right), \label{i3} \\
\int_{-\pi}^\pi \Delta_{d_o,k}^2 (x) H_k(x) dx &=&  O\left(k^{-2 \be-1}\right), \label{i4} \\
\int_{-\pi}^\pi H_k^4(x) dx & \lesssim & \frac{\log k}{k}. \label{i5}
\end{eqnarray} \noindent
When $k\rightarrow \infty$, the big-O in \eqref{s1}-\eqref{i4} may be replaced by a small-o, since $\sum_{l>k} \ty_{o,l}^2 l^{2\be}$ then tends to zero.
%
\end{lem} \noindent
\begin{proof}
The result for $\int H_k^2(x) dx$ follows directly from the definition of $H_k$.
The assumption that $\ty_o \in \Theta(\be,L_o)$ and the Cauchy-Schwarz inequality imply that
\begin{equation*}
\sum_{l>k} |\ty_{o,l}| \leq \sqrt{\sum_{l>k} \ty_{o,l}^2 l^{2\be}} \sqrt{\sum_{l>k} l^{-2\be}} = O(k^{-\be+\ft}),
\end{equation*}
proving the first result in \eqref{s1}. Similarly, one can prove \eqref{i2}. For \eqref{i3}, note that $\sum_{l>k} \ty_{o,l}^2 \leq k^{-2 \be} \sum_{l>k} \ty_{o,l}^2 l^{2 \be}$.
For the other bounds we omit the details of the proof. They follow from the fact that 
for all sequences $a$, $b$ and $c$,
\begin{equation*} 
\begin{split}
& 2 \sum_{l,m,n > k} a_l b_m c_n \int_{-\pi}^{\pi} \cos(l x) \cos(m x) \cos(n x) dx \\
&\qquad = \sum_{m,n > k} b_m c_n \sum_{l > k} a_l \int_{-\pi}^{\pi} \cos(l x) \left(\cos((m+n)x) + \cos((m-n)x)\right) dx \\
&\qquad = \sum_{m,n > k}  a_{m+n} b_m c_n + \sum_{m,n > k; m-n>k} a_{m-n} b_m c_n.
\end{split}
\end{equation*} \noindent

\end{proof} \noindent

Before stating the next lemma we give bounds for the functions $H_k$ and $G_k$.
Since $-2\log|1-e^{ix}| = -\log( x^2 + O(x^4))$, there exist positive constants $c$, $B_0$, $B_1$ and $B_2$ such that
\begin{equation} \label{Hbound0}
|H_k(x)| \geq B_0 |\log x|, \quad |x| \leq c k^{-1},
\end{equation} \noindent
\begin{equation} \label{Hbound}
|H_k(x)| \leq B_1 |\log x| + B_2 \log k, \quad x \in [-\pi,\pi].
\end{equation} \noindent
\begin{lem} \label{integralLemma2}
Let $a_j=(\ty_{o,j}-(d-d_o)\eta_j) 1_{j>k}$, as in \eqref{deltakDef}.
Then for $p \geq 1$ and $q=2,3,4$ there exist constants $c(p,q)$ such that for all $d \in (-\ft,\ft)$ and $k \leq \exp(|d-d_o|^{-1})$,
\begin{equation}\label{i10}
\begin{split}
& \int_{-\pi}^\pi \left(\frac{f_o(x)}{f_{d,k}(x)}\right)^p |H_k|^q(x) dx =
O\left(\frac{(\log k)^{c(p,q)}}{k}\right) \\ & \qquad + O((\log k)^{q + p B_2 |d-d_o|} |d-d_o|^{-\frac{q}{2}}e^{-|d-d_o|^{-1}}),
\end{split}
\end{equation}
\begin{equation} \label{i9}
\frac{1}{2\pi} \int_{-\pi}^\pi  \left( \frac{f_o}{f_{d,k}}(x) - 1 \right) \cos(i x) \cos(j x) dx =
\ft a_{i+j} 1_{i+j > k} + O\left(\sum_{j>k}a_j^2\right),
\end{equation}
\begin{equation} \label{i11}
\frac{1}{2\pi} \int_{-\pi}^\pi  \left( \frac{f_o}{f_{d,k}}(x) - 1 \right) H_k^2(x) dx =
O(|d-d_o| k^{-1} \log k),
\end{equation}
where the constant $B_2$ in \eqref{i10} is as in \eqref{Hbound}, and the constants in \eqref{i9} and \eqref{i11} are uniform in $d$. The constant $c(p,q)$ in \eqref{i10} equals $0,\ft,1$ when respectively $q=2,3,4$.
\end{lem} \noindent
\begin{proof}
When $d=d_o$, \eqref{i10} directly follows from \eqref{i1} and \eqref{i5}, because of the boundedness of $(f_o/f_{d_o,k})^p = \exp\{p \Delta_{d_o,k}\}$. Now suppose $d \neq d_o$.
Let $C_k = \max_{x \in [-\pi,\pi]} \exp\{|\Delta_{d_o,k}(x)|\}$ and $b_m = \max_{x \in [m,\pi]} |(d-d_o) H_k(x)|$,
for $m = e^{-\frac{1}{|d-d_o|}} < e^{-1}$. Since $\sum_{j=0}^{\infty} |\ty_{o,j}| < \infty$, the sequence $C_k$ is bounded by some constant $C$. To prove \eqref{i10} we write
\begin{equation} \label{i10split}
\begin{split}
& \ft \int_{-\pi}^\pi \left(\frac{f_o(x)}{f_{d,k}(x)}\right)^p |H_k|^q(x) dx \\
& \quad= \int_{0}^m \left(\frac{f_o(x)}{f_{d,k}(x)}\right)^p |H_k|^q(x) dx + \int_{m}^\pi \left(\frac{f_o(x)}{f_{d,k}(x)}\right)^p |H_k|^q(x) dx.
\end{split}
\end{equation} \noindent
We first bound the last integral in the preceding display, by substitution of $(f_o/f_{d,k})^p = \exp\{p\Delta_{d,k}\} = \exp\{-p (d-d_o) H_k + p \Delta_{d_o,k}\}$.
From \eqref{Hbound} it follows that
\[b_m \leq |d-d_o| (B_1 |d-d_o|^{-1} + B_2 \log k) \leq B_1 + B_2,\] \noindent
as $k \leq \exp(|d-d_o|^{-1})$.
Hence we obtain $(f_o/f_{d,k})^p \leq C e^{b_m}$ on $(m,\pi)$. For $q=2$ and $q=4$ the bound on the last integral in \eqref{i10split} therefore follows
from \eqref{i1} and \eqref{i5}; for $q=3$ the bound follows from the Cauchy-Schwarz inequality.

Next we bound the first integral in \eqref{i10split}. Because the function $x^{|d-d_o|}(\log x)^2$ has a local maximum of $4 |d-d_o|^{-2} e^{-2}$ at $x = e^{-2/|d-d_o|}$,
$(\log x)^2 \leq 4 x^{-|d-d_o|} |d-d_o|^{-2} e^{-2}$ for all $x \in [0,m]$.
Again using \eqref{Hbound} we find that
\begin{equation*}
\begin{split}
& \int_0^m \left(\frac{f_o(x)}{f_{d,k}(x)}\right)^p |H_k|^q(x) dx \lesssim
\sum_{j=0}^q \binom{q}{j} \int_0^m \left(B_1 |\log x|\right)^j \left(B_2 \log k\right)^{q-j}
e^{- p (d-d_o) H_k(x)} dx \\
& \quad \lesssim \sum_{j=0}^q \binom{q}{j} (\log k)^{q-j + p B_2 |d-d_o|}
\int_0^m \left(B_1 |\log x|\right)^j x^{-p B_1 |d-d_o|} dx \\
& \quad \leq \sum_{j=0}^q \binom{q}{j} (\log k)^{q-j + p B_2 |d-d_o|} \left(\frac{2 B_1^2}{e |d-d_o|}\right)^\frac{j}{2}
\int_0^m  x^{-(j/2 + p B_1) |d-d_o|} dx  \\
& \quad \lesssim (\log k)^{q + p B_2 |d-d_o|} |d-d_o|^{-\frac{q}{2}} e^{-1/|d-d_o|}.
\end{split}
\end{equation*} \noindent

We now prove \eqref{i9}.
\begin{eqnarray*}
\lefteqn{  \frac{1}{2\pi} \int_{-\pi}^\pi \left( \frac{f_o}{f_{d,k}}(x) - 1 \right) \cos( i x) \cos(j x)dx } \\
 &=&
  \frac{1}{2\pi} \int_{-\pi}^\pi \left( e^{\Delta_{d,k}(x)} - 1 \right) \cos( i x) \cos(j x) dx \\
& \leq &  \frac{1}{2\pi} \int_{-\pi}^\pi \left(\Delta_{d,k}(x)
+ \ft \Delta_{d,k}^2(x) e^{(\Delta_{d,k}(x))_{+}}\right) \cos( i x) \cos(j x) dx.
\end{eqnarray*} \noindent
The linear term equals
\[ \begin{split}
& \frac{1}{2\pi} \int_{-\pi}^\pi \Delta_{d,k}(x) \cos ( i x) \cos ( j x) dx
\\ \quad &= \frac{1}{4\pi} \int_{-\pi}^\pi \left(\sum_{l > k} a_l \cos(l x)\right)
\left( \cos((i + j) x) + \cos((i - j) x) \right) dx = \frac{1}{2} a_{i+j} 1_{i+j > k}.
\end{split}
\] \noindent
For the quadratic term we have
\begin{equation} \label{remainderTerms}
\begin{split}
& \left|\frac{1}{2\pi} \int_{-\pi}^\pi \Delta_{d,k}^2(x)
e^{(\Delta_{d,k}(x))_{+}} \cos( i x) \cos(j x) dx \right|
\leq \frac{1}{2\pi} \int_{-\pi}^\pi \Delta_{d,k}^2(x) e^{(\Delta_{d,k}(x))_{+}} dx \\
& \quad \leq  \frac{1}{2\pi} \int_{0}^m \Delta_{d,k}^2 (x) e^{-\Delta_{d,k}(x)} dx +
\frac{(1 + C e^{b_m})}{2 \pi} \int_{-\pi}^\pi  \Delta_{d,k}^2 (x) dx.
\end{split}
\end{equation} \noindent
This is $O(\sum_{j>k}a_j^2)$, which follows from \eqref{ajBound} and integration over $(0,e^{-\frac{1}{\vb}})$ and $(e^{-\frac{1}{\vb}}, \pi)$ as above.

To prove \eqref{i11}, write $\exp(\Delta_{d,k})-1 = \Delta_{d,k} + \Delta_{d,k}^2 e^{\xi}$ with $\Delta_{d,k} = -(d-d_o) H_k(x) + \Delta_{d_o,k}(x)$ and $|d-d_o|\leq \vb$, substitute \eqref{Hbound} and proceed as in the proof of \eqref{i10} above. The biggest term is a multiple of $|d-d_o| \int_{-\pi}^\pi |H_k(x)|^3 dx$, which is $O(\vb k^{-1})$.
This is larger than the approximation error when $\be > \ft (1+\sqrt{2})$.
\end{proof} \noindent
\begin{lem} \label{deviationBound}
Let $A$ be a symmetric matrix
matrix such that $\fr A\fr=1$ and let $Y=(Y_1,\ldots,Y_n)$ be a vector of independent standard normal random variables.
Then for any $\az >0$,
\begin{equation*}
P\left(Y^t A Y - \tr(A) >  n^{\az}\right) \leq  \exp\{- n^{\az}/8\}.
\end{equation*} \noindent
\end{lem} \noindent

\begin{proof}
 Note that $\|A\| \leq \fr A\fr=1$ so that for all $s \leq 1/4$, $s y^t A y \leq s_0 y^ty \|A\| \leq y^ty/4$
and $\exp\{s Y^t A Y\}$ has finite expectation. Choose $s = 1/4$, then
by Markov's inequality, 
\begin{eqnarray*} \label{chernoffBound}
 P\left(Y^t A Y - \tr(A) > n^{\az}\right) &\leq& e^{- n^{\alpha }/4} E e^{(Y^tA Y - \tr(A))/4} \nonumber \\
&  =&  \exp\left\{- n^{\az}/4 -\ft \log  \mbox{det}[I_n -  A/2] - \tr(A)/4 \right\} \nonumber \\
&  \leq&  \exp\left\{- n^{\az}/4 + \tr(A^2)/4\right\}.
\end{eqnarray*} \noindent
 The last inequality
follows from the fact that $A (I_n -  \tau A/2)^{-1}$ has eigenvalues $\lb_j(1- \tau \lb_j/2)^{-1}$, where $\lb_j$ are the eigenvalues of $A$ for all $\tau \in (0,1)$.
Hence, $\tr(A^2 (I_n -  \tau A/2)^{-2})$ is bounded by $4 \tr(A^2)$. The result follows from the fact that when $n$ is large enough $n^{\az} >2 \tr(A^2) = 2$.
\end{proof} \noindent
\section{Convergence of the trace of a product of Toeplitz matrices} \label{toeplitz}

Suppose $T_n(f_j)$ ($j=1,\ldots,p$) are covariance matrices associated with spectral densities $f_j$.
According to a classical result by Grenander and Sz\'eg\"o (\citet{grsz1958}),
\begin{equation*}
\fn \tr\left[\prod_{j=1}^p T_n(f_j) \right] \rightarrow (2\pi)^{2 p -1} \int_{-\pi}^\pi \prod_{j=1}^p f_j(x) dx.
\end{equation*}
In this section we give a series of related results.  We first recall a result from \citet{rcl}.
\begin{lem} \label{ProductBound}
Let $1/2>t>0$  and $L^{(i)}, M^{(i)} >0$,  $\rho_i \in (0,1]$, $d_i \in [-1/2 + t, 1/2-t]$ for all $i=1,..., 2p$ and let $f_i$, ($i \leq 2p$) be functions on $[-\pi, \pi]$
satisfying
\begin{equation} \label{fgbound}
| f_i (x)|  = |x|^{-2d_i} g_i(x), \, |g_i(x)|\leq   M^{(i)},  \, | g_i (x) - g_i(y) | \leq \frac{  M^{(i)} |x-y|}{|x|\wedge |y|} + L^{(i)}|x-y|^{\rho_i}
\end{equation} \noindent
and assume  that $ \sum_{i=1}^p (d_{2i-1} + d_{2i}) < \ft$.
Then for all $\e>0 $ there  exists a constant $K$ depending only on $\e, t$ and $q = \sum_{j=1}^p (d_{2j-1}+d_{2j})_{+}$ such that
\begin{equation} \label{ProductBoundResult}
\begin{split}
 & \left| \fn \tr\left[\prod_{j=1}^p T_n(f_{2j-1}) T_n(f_{2j}) \right]
- (2\pi)^{2 p -1} \int_{-\pi}^\pi \prod_{j=1}^{2p} f_{j}(x)  dx  \right| \nonumber \\
 &\leq  K \sum_{j=2}^{2p} \left(\prod_{i\neq j} M^{(i)}\right)
L^{(j)} n^{-\rho_j + \e + 2 q} + K\prod_{i=1}^{2p} M^{(i)} n^{-1+q+\e}.
\end{split}
\end{equation} \noindent
\end{lem}

To prove a similar result involving also inverses of matrices, we need the following two lemmas.
They can be found elsewhere, but as we make frequent use of them they are included for easy reference and are formulated in a way better suited to our purpose.
The first lemma can be found on p.19 of \citet{rcl}, and is an extension of Lemma 5.2 in \citet{d1}.
\begin{lem} \label{inversionBound}
Suppose that for $0 < t < 1/2$ and $d \in [-1/2 + t , 1/2 - t ]$
\begin{equation} \label{fbound1}
| f (x)|  = |x|^{-2d} g(x), \quad m \leq |g(x)|\leq   M,  \quad | g (x) - g(y) | \leq  L|x-y|^{\rho}
\end{equation} \noindent
and assume that $0 <m \leq 1 \leq M < +\infty$ and $L\geq 1$.
Then, for all $\e>0$,  there exists a constant $K$ depending on $t$ and $\e$ only such that
\begin{equation*}
\fr I_n - T_n^{\ft}(f) T_n\left(\frac{1}{4 \pi^2 f}\right) T_n^{\ft}(f)\fr^2\leq K L\frac{M^2}{m^2} n^{1-\rho + \e}.
\end{equation*} \noindent
\end{lem}
\begin{proof}
By Lemma \ref{ProductBound},
\begin{equation*}
\begin{split}
&\fr I_n - T_n^{\ft}(f) T_n\left(\frac{1}{4 \pi^2 f}\right) T_n^{\ft}(f) \fr^2
= \tr\left\{I_n - 2  T_n^{\ft}(f) T_n\left(\frac{1}{4 \pi^2 f}\right)  T_n^{\ft}(f) \right.\\
& \quad \left.+ T_n^{\ft}(f) T_n\left(\frac{1}{4 \pi^2 f}\right) T_n(f) T_n\left(\frac{1}{4 \pi^2 f}\right) T_n^{\ft}(f) \right\}
\end{split}
\end{equation*}
converges to zero, the approximation error being bounded by $ K [ L ( 1 + M^2/m^2) + M^2 / m^2 ]$.
\end{proof}
The next result can be found as Lemma 3 in \citet{LRR09}, and is an extension of Lemma 5.3 in \citet{d1}.
\begin{lem} \label{traceBound}
Suppose that $f_1$ and $f_2$ are such that $|f_1(x)| \geq m |x|^{-2 d_1}$ and $|f_2(x)| \leq M |x|^{-2 d_2}$ for constants $d_1, d_2 \in (-\ft,\ft)$ and $m,M>0$. Then
\begin{equation*}
\|T_n^{-\ft}(f_1) T_n^{\ft}(f_2)\| \leq C \frac{M}{m} n^{(d_2-d_1)_{+} + \e}.
\end{equation*} \noindent
\end{lem}
\begin{proof}
In the proof of Lemma 5.3 on p. 1761 in \citet{d1}, the first inequality only depends
on the upper and lower bounds $m$ and $M$.
\end{proof} \noindent
Using the preceding lemmas, the approximation result given in Lemma \ref{ProductBound} for
traces of matrix products can be extended to include matrix inverses.
\begin{lem} \label{QuotientBound}
Suppose that $f$ satisfies \eqref{fbound1} with constants $d$, $\rho$, $L$, $m$ and $M$. For $f_{2j}$, $j=1,\ldots,p$, assume that \eqref{fgbound} holds with constants
$d_{2j}$, $\rho_{2j}$, $L^{(2j)}$ and $M^{(2j)}$ ($j=1,\ldots,p$). For convenience, we denote $M^{(2j-1)}=m^{-1}$ , $\rho_{2j-1} = \rho$ and $L^{(2j-1)} = L$ ($j=1,\ldots,p$). Suppose in addition that  $d, d_{2j} \in [-\ft+t,\ft-t]$ satisfy $ \sum_{j=1}^p (d_{2j} - d)_+ < \ft(\rho - \ft)$, and let $q = \sum_{j=1}^p (d_{2j} - d)_{+}$. Then for all
$\e>0 $  there exists a  constant  $K$ such that
\begin{equation} \label{quotientResult}
\begin{split}
& \left| \fn \tr\left\{\prod_{j=1}^p T_n^{-1}(f) T_n(f_{2j}) \right\}
- \frac{1}{2 \pi} \int_{-\pi}^\pi \prod_{j=1}^p \frac{f_{2j}(x)}{f(x)} dx \right|  \\
& \qquad \leq K\left[ \sum_{j=2}^{2p} \left(\prod_{i\neq j}^{2p} M^{(i)} \right)
L^{(j)} n^{-\rho_j}  + n^{-1}\prod_{i\leq 2p} M^{(i)} \right]n^{ \e + 2 q} \\
& \qquad \quad  + \left(\prod_{j=1}^p M^{(2j)} \right)
\left(L\frac{M}{m}\right)^{\frac{(p+1)}{2}}n^{(1-\rho)\frac{(p+1)}{2}-1 + \e + 2 q},
\end{split}
\end{equation} \noindent
and setting $\tilde{f} = 1/ (4 \pi^2 f)$,
\begin{equation} \label{quotientResult2}
\begin{split}
 &\fn \left| \tr\left\{\prod_{j=1}^p T_n^{-1}(f) T_n(f_{2j}) \right\} - \tr\left\{\prod_{j=1}^p T_n(\tilde f) T_n(f_j) \right\}\right| \\
 & \leq \left(\prod_{j=1}^p  M^{(2j)} \right)
\left(L\frac{M}{m}\right)^{\frac{(p+1)}{2}}n^{(1-\rho)\frac{(p+1)}{2}-1 + \e + 2 q}.
\end{split}
\end{equation}
\end{lem}
\begin{proof}
Without loss of generality, we consider the $f_{2j}$'s to be nonnegative 
When this is not the case, we write $f_{2j} = f_{2j}^{+} - f_{2j}^{-}$
and treat the positive and negative part separately; see also \citet{d1} , p. 1755-56.
To prove \eqref{quotientResult2}, we use the construction of Lemma 5 from \citet{LRR09}, who treat the case $\rho=1$ and $d_{2j}=d'$. Inspection of their proof shows that this extends to $\rho \neq 1$ and $d_{2j}$ that differ with $j$.
To prove \eqref{quotientResult}, we use the construction of Dahlhaus' Theorem 5.1 (see also the remark on p. 744 of \citet{lp2}, after (28)), and apply Lemma \ref{ProductBound} with $f_{2j-1} = \tilde f = \frac{1}{4 \pi^2 f}$, $j=1,\ldots,p$. This gives the first term on the right in \eqref{quotientResult}. The last term in \eqref{quotientResult} follows from \eqref{quotientResult2}.
\end{proof}
Although the bound provided by Lemma \ref{QuotientBound} is sufficiently tight for most purposes, certain applications require sharper bounds. These can only be obtained if we exploit specific properties of $f$ and $f_{2j}$.
In Lemma \ref{interBound} below we improve on the first term on the right in \eqref{quotientResult}. This is useful when for example $b_i(x)=cos(jx)$; the Lipschitz constant $L$ is then of order $O(k)$, but the boundedness of $b_i$ actually allows a better result. In Lemma \ref{alternative_quotient_bound} we improve on the last term of \eqref{quotientResult}.
%
%
\begin{lem}\label{interBound}
Let $f(x) = |x|^{-2d} g(x)$ with $-1/2<d<1/2$ and $g$ a bounded Lipschitz function satisfying $m < g < M$, with Lipschitz constant $L$.
\begin{itemize}
\item 	Let $b_1,\ldots,b_p$ be  bounded functions and  let  $\|b\|_\infty$ denote a common upper bound for these functions. Then for all $\e>0$, 
\begin{equation} \label{inter:1}
\begin{split}
& \left|\tr\left[ \prod_{i=1}^pT_n(b_i f) T_n(f^{-1})\right] - (2\pi)^p \tr\left[ \prod_{i=1}^p T_n(b_i) \right]\right| \\ & \qquad \leq
C n^\e  \left( \frac{M}{m}\right)^{p} \|b\|_\infty^{p-1} \left( \|b\|_\infty + L\sum_{j=1}^p\|b_j\|_2\right).
\end{split}
\end{equation}
\item Let $b_j$ $(j\geq 2)$ be bounded  functions. Let $b_1$ be such that $\|b_1\|_2 < +\infty $, and assume that for all $a>0$ there exists $M'(a)>0$ such that
$$\int_{-\pi}^\pi |b_1(x)| |x|^{-1+a}dx \leq M'(a).$$
Then for all $a>0$
\begin{equation} \label{inter:2}
\begin{split}
&\left|\tr\left[ \prod_{i=1}^pT_n(b_i f) T_n(f^{-1})\right] - (2\pi)^p \tr\left[ \prod_{i=1}^p T_n(b_i) \right]\right| \\ \qquad &\leq
C \left( \frac{M}{m}\right)^p\prod_{i\geq 2} \|b_i\|_\infty \left(n^{3pa}  M'(a)  + L(\log n)^{2p-1} \|b_1\|_2\right).
\end{split}
\end{equation}
\end{itemize}
\end{lem}
\begin{proof}
We prove \eqref{inter:1}; the proof of \eqref{inter:2} follows exactly the same lines.
We define $\Delta_n(x) = e^{ix}$ and $L_n(x) = n \wedge |x|^{-1}$ where the latter is an upper bound of the former.
Using the decomposition as on p. 1761 in \citet{d1} or as in the proof of  we find that
\begin{equation*}
\begin{split}
& \left|\tr\left[ \prod_{i=1}^pT_n(b_i f) T_n(f^{-1})\right] - (2\pi)^p \tr\left[ \prod_{i=1}^p T_n(b_i) \right]\right|  \\
 &\leq
  C \left| \int_{[-\pi,\pi]^{2p}} \prod_{i=1}^pb_i(x_{2i-1}) \left( \prod_{i=1}^p \frac{ f(x_{2i-1}) }{ f(x_{2i}) } - 1 \right) \Delta_n(x_1-x_2)  \ldots \Delta_n(x_{2p}-x_1)dx \right|\\
   &\leq C \left| \int_{[-\pi,\pi]^{2p}} \prod_{i=1}^pb_i(x_{2i-1})\frac{ g( x_{2i-1})}{ g(x_{2i}) } \left( \prod_{i=1}^p \frac{ |x_{2i-1}|^{-2d} }{ |x_{2i}|^{-2d}  } - 1 \right) \Delta_n(x_1-x_2)  \ldots \Delta_n(x_{2p}-x_1)dx \right| \\
    & \; + C \left| \int_{[-\pi,\pi]^{2p}} \prod_{i=1}^pb_i(x_{2i-1}) \left( \prod_{i=1}^p \frac{ g(x_{2i-1}) }{ g(x_{2i}) } - 1 \right) \Delta_n(x_1-x_2)  \ldots \Delta_n(x_{2p}-x_1)dx \right|\\
     & \leq C \left( \frac{M\|b\|_\infty}{m}\right)^{p}\sum_{j=1}^p \int_{[-\pi,\pi]^{2p}} \prod_{i=1}^j \frac{ |x_{2i-1}-x_{2i}|^{1-3a} }{ (|x_{2i}|\wedge |x_{2i-1}|)^{1-a}  }  L_n(x_1-x_2) \ldots L_n(x_{2p}-x_1)dx \\
      & \; +  C L\left( \frac{M\|b\|_\infty}{m}\right)^{p-1} \sum_{j=1}^p \int_{[-\pi,\pi]^{2p}} |b_j(x_{2j-1})| |x_{2j-1} - x_{2j}| L_n(x_1-x_2) \ldots L_n(x_{2p}-x_1)dx \\
      &\leq
      C \left( \frac{M\|b\|_\infty}{m}\right)^{p} n^{3pa}\left(\int_{[-\pi,\pi]}|x|^{-1+a}dx\right)^p+
          C L\left( \frac{M \|b\|_\infty}{m} \log n \right)^{p-1} (\log n)^{2p-1} \sum_{j=1}^p\|b_j\|_2.
\end{split}
\end{equation*}

\end{proof}

\begin{lem} \label{alternative_quotient_bound}
Let $\tilde{f} =1/(4\pi^2 f)$, and let $\rho>1/2$ and $L>1$, then under   the conditions of Lemma \ref{QuotientBound} 
 we have the following alternative bound for \eqref{quotientResult2}:
\begin{equation} \label{quotientResult3}
\begin{split}
 & \left| \tr\left\{\prod_{j=1}^p T_n^{-1}(f) T_n(f_{2j}) \right\} - \tr\left\{\prod_{j=1}^p T_n(\tilde f) T_n(f_{2j}) \right\}\right| \\
 & \lesssim \sqrt{ L} n^{(1-\rho/2) + 2q + \e} \left\{ \sum_{j=1}^{p-1}  \left(\sqrt{M_{2p}}\prod_{l=j+1}^{p-1}   M^{(2l)} \right)   \times  \left( \int_{-\pi}^{\pi}  \frac{ |f_{2p}(x)| }{f (x) }\prod_{l=1}^{j}\frac{f_{2l}^2}{f^2}(x)  dx \right)^{\ft}  \right. \\
  & \left. +\prod_{l=2}^{p}   M^{(2l)}\left( \int_{-\pi}^{\pi} \frac{f_2^2}{f^2}(x) dx\right)^{\ft}
 +error \right\}
\end{split}
\end{equation}
where
\begin{equation*}
\begin{split}
error &\leq   L^{3/4}n^{(1-3\rho)/4} \prod_{l=1}^{p}   M^{(2l)}+
 \sum_{j=1}^p \sqrt{L^{(2j)}n^{-\rho_{2j}}M^{(2j)} }\prod_{l\neq j} M^{(2l)}
  \end{split}
\end{equation*}
\end{lem}
\begin{rem}
The constant appearing on the right hand side of \eqref{quotientResult3} depends on $M$ and $m$,
but in all our applications of Lemma \ref{alternative_quotient_bound}, the constants $M$ and $m$ will be bounded and of no consequence.
\end{rem}

\begin{proof}
Following the construction of \citet{d1}, equation (13), we  write 
$| \tr\{\prod_{j=1}^p T_n^{-1}(f) T_n(f_{2j}) \} - \tr\{\prod_{j=1}^p T_n(\tilde f) T_n(f_{2j})\}|$ as
\begin{equation}\label{decomp}
\begin{split}
& \left| \tr \left\{\prod_{j=1}^p A_j - \prod_{j=1}^p B_j \right\}\right| \\ & =
\left| \tr \left\{(A_1 - B_1) \prod_{l=2}^p A_l + \sum_{j=2}^p  \left(\prod_{l=1}^{j-1} B_l \right) (A_j - B_j)\prod_{l=j+1}^p A_l \right\}\right|,
\end{split}
\end{equation}
where $A_j = T_n^{\ft}(f_{2j-2})T_n^{-1}(f)T_n^{\ft}(f_{2j})$, $B_j = T_n^{\ft}(f_{2j-2})T_n(\tilde f) T_n^{\ft}(f_{2j})$ and $f_{0} := f_{2p}$ (similarly for $\rho_0$, $L^{(0)}$ and $M^{(0)}$).
When $j=p$, the factor $\prod_{l=j+1}^p A_l$ is understood to be the identity.
Without loss of generality,  the functions $f_{2j}$ are assumed  to be positive (it suffices to write $f_{2j} = f_{2j+} - f_{2j-}$).
Lemma \ref{traceBound} implies that for each $j$,
\begin{equation} \label{lemma2.3.bound}
\|T_n^{-\ft}(f) T_n^{\ft}(f_{2j})\| \lesssim \frac{M^{(2j)}}{m} n^{(d_{2j}-d)_{+} + \e}.
\end{equation}
Using the relations in (1.6) (main paper) it then follows that
\begin{equation} \label{AlBound}
\begin{split}
\Big\| \prod_{l=j+1}^p A_l \Big\| &\leq \prod_{l=j+1}^p \| T_n^{\ft}(f_{2l}) T_n^{-\ft}(f)\|^2 \\
 &\lesssim  \left(\prod_{l=j+1}^{p-1} M^{(2l)}  \right)  n^{2\sum_{l=j+1}^{p-1} (d_{2l}-d)_{+} + (d_{2j}-d)_+ + (d_{2p}-d)_+} \sqrt{M^{(2p)}M^{(2j)}}.
 \end{split}
\end{equation}
First we treat the term $(A_1 - B_1) \prod_{l=2}^p A_l$ on the right in \eqref{decomp}. Writing $R =I_n- T_n^{\ft}(f) T_n(\tilde f) T_n^{\ft}(f)$, it follows that
\begin{equation}  \label{A1B1}
\begin{split}
& \left|\tr\left[ (A_1-B_1) \prod_{l= 2}^p A_l\right]\right| \\
&\quad = \left|\tr\left[T_n^{\ft}(f_{2p}) T_n^{-\ft}(f) R T_n^{-\ft}(f) T_n(f_{2}) T_n^{-\ft}(f)  T_n^{-\ft}(f) T_n^{\ft}(f_{4})  \prod_{l=3}^p A_l \right] \right|  \\
 &\quad \leq  \fr R \fr  \fr T_n^{-\ft}(f) T_n(f_{2}) T_n^{-\ft}(f)\fr  \|T_n^{\ft}(f_{2p}) T_n^{-\ft}(f) \| \| T_n^{-\ft}(f) T_n^{\ft}(f_{4}) \| \Big\|\prod_{l=3}^p A_l\Big\| \\ 
 &\quad \lesssim  L^{\ft}n^{(1 - \rho)/2+\e+\ft +2q} \prod_{l=2}^p M^{(2l)} \left( \int_{-\pi}^{\pi} \frac{ f_2^2(x) }{ f^2(x)}dx + \mbox{error} \right)^{\ft} \\
  &\quad \lesssim  L^{\ft}n^{(1 - \rho)/2+\e+\ft +2q}  \prod_{l=2}^p M^{(2l)}\times \\
   & \quad \quad  \left( \int_{-\pi}^{\pi} \frac{ f_2^2(x) }{ f^2(x)}dx + n^{\e + 2q} \left(L^{3/2}\left( M^{(2)}\right)^2 n^{(1-3\rho)/2} +  M^{(2)} L^{(2)}n^{-\rho_2} \right) \right)^{\ft}.
 \end{split}
 \end{equation}
The first inequality follows from the relations in (1.6) (main paper). The second inequality follows after writing $\fr T_n^{-\ft}(f) T_n(f_{2}) T_n^{-\ft}(f)\fr$ as the sum of a limiting integral and an approximation error; in addition we use \eqref{lemma2.3.bound} and Lemma \ref{inversionBound}, by which
\begin{equation}\label{R2bound}
\fr R \fr^2 \leq KL (M/m)^2 n^{1-\rho+\e}\lesssim L n^{1-\rho +\e}.
\end{equation}
This follows from Lemma \ref{QuotientBound}, which we use to bound the approximation error. The second term within the brackets in \eqref{A1B1} constitutes  part of the term $error$.

Next we bound the term $(\prod_{l=1}^{j-1} B_l )(A_j - B_j)\prod_{l=j+1}^p A_l$ in \eqref{decomp} for $j=2$.
Similar to the preceding decomposition, we have
\begin{equation*} \label{A2B2}
\begin{split}
 \left|\tr\left[ B_1(A_2-B_2) \prod_{l=3}^p A_l\right]\right| &= \left|\tr\left[ B_1 T_n^{\ft}(f_{2}) T_n^{-\ft}(f) R T_n^{-\ft}(f) T_n^{\ft}(f_{4})\prod_{l=3}^p A_l \right] \right| \\
 &\leq \fr B_1 T_n^{\ft}(f_{2}) T_n^{-\ft}(f) \fr \fr R\fr \| T_n^{-\ft}(f) T_n^{\ft}(f_{4}) \| \Big\|\prod_{l=3}^p  A_l\Big\|.
\end{split}
\end{equation*}
The terms $\fr R\fr$, $\| T_n^{-\ft}(f) T_n^{\ft}(f_{4}) \|$ and $\| \prod_{l=3}^p A_l\|$ are bounded as in \eqref{lemma2.3.bound}, \eqref{AlBound} and \eqref{R2bound}. For the term $\fr B_1 T_n^{\ft}(f_{2}) T_n^{-\ft}(f) \fr$ we have the decomposition
\begin{equation*}
\begin{split}
& \fr B_1 T_n^{\ft}(f_{2}) T_n^{-\ft}(f) \fr^2 = \tr\left[T_n^{-\ft}(f) T_n^{\ft}(f_{2}) B_1^t B_1  T_n^{\ft}(f_{2}) T_n^{-\ft}(f) \right] \\
&=  \tr\left[B_1^t B_1  T_n^{\ft}(f_{2}) T_n^{-1}(f) T_n^{\ft}(f_{2}) \right] = \tr\left[B_1^t B_1  T_n^{\ft}(f_{2}) T_n(\tilde f) T_n^{\ft}(f_{2}) \right] \\
& \qquad \quad + \tr\left[B_1^t B_1  T_n^{\ft}(f_{2}) T_n^{-\ft}(f) R T_n^{-\ft}(f) T_n^{\ft}(f_{2}) \right] \\
&\leq \fr B_1 T_n^{\ft}(f_{2}) T_n^{\ft}(\tilde f) \fr^2 + \fr B_1^t B_1\fr    \fr R\fr  \| T_n^{\ft}(f_2) T_n^{-\ft}(f) \|^2.
\end{split}
\end{equation*}
Using again Lemmas \ref{ProductBound}, \ref{inversionBound} and \ref{traceBound}, we find that the first term on the right is bounded by
\begin{equation*}
n \left\{ \int_{-\pi}^{\pi} \frac{ |f_{2p}(x)| f_2^2(x) }{ f^3(x)} dx  + M^{(2p)}M^{(2)} n^{4(d_2 - d)_+ + 2(d_{2p}-d)_++ \e}   \left( M^{(2)} L n^{-\rho} +  L^{(2)}n^{-\rho_2} \right)\right\}
\end{equation*}
and the second term by
\begin{equation*}
\begin{split}
& n^{\ft-\frac{\rho}{2}+(d_2-d)_{+} + \e} \sqrt{L} M^{(2)}  \left[ n \int_{-\pi}^{\pi} \frac{ f_{2p}^2(x) f_2^2(x)}{ f^4(x)} dx + M^{(2p)}  M^{(2)} n^{4(d_2 - d)_+ + 2(d_{2p}-d)_++ \e} \times  \right. \\
   & \qquad \left. \times   \left( (M^{(2)})^2(M^{(2p)})^2  L n^{1-\rho} + (M^{(2p)})^2M^{(2)}  L^{(2)}n^{1-\rho_2}  + (M^{(2)})^2M^{(2p)}L^{(2p)} n^{1-\rho_{2p}}  \right)\right]^{\ft}.
\end{split}
\end{equation*}
Consequently, 
\begin{equation*} \label{A2B2}
\begin{split}
& \left|\tr\left[ B_1(A_2-B_2) \prod_{l=3}^p A_l\right]\right| \\
& \quad \lesssim
 L ^{\ft}n^{\ft + (1 - \rho)/2+\e+2q}\sqrt{M^{(2p)}}\prod_{l=2}^{p-1}  M^{(2l)}\left[\int_{-\pi}^{\pi} \frac{  f_2^2(x) |f_{2p}(x)| }{ f^3(x)}dx + M^{(2p)}(M^{(2)})^2 Ln^{-\rho} \right.  \\
& \quad +  n^{-\rho/2}( L)^{1/2} \sqrt{M^{(2)}}  \left( \int_{-\pi}^{\pi} \frac{ f_2^2(x)f_{2p}^2(x) }{ f^4(x)} dx \right)^{\frac{1}{2}}+ L^{(2)}M^{(2)}M^{(2p)} n^{-\rho_2}  \\
& \quad \left.+(LL^{(2)})^{1/2} M^{(2p)}(M^{(2)})^{3/2}  n^{-(\rho + \rho_2)/2} +(LL^{(2p)}M^{(2p)})^{1/2}(M^{(2)})^{2}  n^{-(\rho + \rho_{2p})/2} \right]^{\ft}.
  \end{split}
 \end{equation*}
Note that
$$ \left( \int_{-\pi}^{\pi} \frac{ f_2^2(x)f_{2p}(x)^2 }{ f^4(x)} dx \right)^{\frac{1}{2}} \lesssim M^{(2)}M^{(2p)},$$
$$(LL^{(2)}M^{(2p)})^{1/2} (M^{(2)})^{3/2}  n^{-(\rho + \rho_2)/2} \leq L(M^{(2)})^2 n^{-\rho} + L^{(2)}M^{(2)}M^{(2p)}n^{-\rho_2}$$
and $L n^{-\rho}  \lesssim L^{3/2} n^{(1-3\rho)/2}$. Therefore the terms on the right are of the same order as the right hand side of \eqref{A1B1}.
A similar argument applies to the term $(LL^{(2p)}M^{(2p)})^{1/2}(M^{(2)})^{2}  n^{-(\rho + \rho_{2p})/2}$.

Finally, we bound the term $(\prod_{l=1}^{j-1} B_l ) (A_j - B_j)\prod_{l=j+1}^p A_l$ in \eqref{decomp} for $j\geq 3$.
For $j \geq 3$, Lemma  \ref{ProductBound} implies that  
\begin{equation*}
\fr \prod_{l=1}^{j-1}  B_l \fr^2 =
n \left(  \int_{-\pi}^{\pi} \frac{f_{2p}(x)f_{2j-2} (x)}{f^2(x) }\prod_{l=1}^{j-2}\frac{ f_{2l}^2 (x)}{ f^2 (x)} dx + error_j \right),
\end{equation*}
where
$$ error_j \lesssim n^{ \e +2 \sum_{l=1}^{j-1}(d_{2l} - d)_+} \prod_{l=1}^{j-1}  M^{(2l)}M^{(2l-2)} \left(  L n^{-\rho} + \sum_{l=1}^{j-1}
\frac{L^{(2l)}}{M^{(2l)}} n^{ - \rho_{2l}})  \right). $$
Consequently, we have for all $j\geq 2$
\begin{equation*}
 \begin{split}
 &\left|\tr\left[ \left(\prod_{l=1}^{j-1} B_l \right)(A_j-B_j)  \prod_{l=j+1}^p A_l\right]\right| \\
 &\quad \leq
 \fr \prod_{l=1}^{j-1} B_l  \fr \fr R\fr \prod_{l=j+1}^p \| A_l\| \|T_n^{\ft}( f_{2j} )T_n^{-\ft }(f)\|  \|T_n^{\ft}( f_{(2j-2)})T_n^{-\ft }(f)\| \\
 &\quad \lesssim  L ^{\ft}n^{\ft + (1 - \rho)/2+2q +\e}\sqrt{M^{(2p)}} \prod_{l= j+1}^{p-1}  M^{(2l)}\left(   \int_{-\pi}^{\pi} \frac{f_{2p}(x)f_{2j} (x)}{f^2(x)}  \prod_{l=1}^{j-1}\frac{ f_{2l}^2 (x)}{ f^2 (x)} dx + error_j \right)^{\ft}.
 \end{split}
\end{equation*}
for all $j \geq 3$, which finishes the proof of Lemma \ref{alternative_quotient_bound}.
\end{proof}

%
\section{H\"older constants of various functions} \label{holderConstants}
%
%
%
\begin{lem} \label{constantsLemma}
Let $\ty_o \in \Theta(\be,L_o)$. Then $f_o$ satisfies condition \eqref{fbound1} with
$\rho=1$ when $\be>\frac{3}{2}$, and with any $\rho < \be-\ft$ when $\be\leq\frac{3}{2}$. The H\"older-constant only depends on $L_o$. When $\ty \in \Theta_k(\be,L)$, $f_{d,k,\ty}$ satisfies \eqref{fbound1} with $\rho=1$, regardless of $\be$. The H\"older-constant is of order $k^{\frac{3}{2}-\be}$. The  function $-\log(2-2\cos(x)) f_{d,k,\ty}$ satisfies condition \eqref{fgbound} with $\rho=1$ and H\"older-constant of order $k^{\frac{3}{2}-\be}$.  The functions $G_k f_{d,k,\ty}$ and $H_k f_{d,k,\ty}$, with $G_k$ and $H_k$ as in \eqref{HkGkDef}, satisfy  \eqref{fgbound} with $\rho=1$ and H\"older-constant of order $k$.
\end{lem}
\begin{proof}
The function $\sum_{j=0}^\infty \ty_{o,j} \cos(j x)$ (i.e. the logarithm of the short-memory part of $f_o$), has smoothness $\rho < \be - \ft$, since
\begin{equation*}
\begin{split}
\sum_{j=0}^\infty |\ty_{o,j}| |\cos(jx) - \cos(jy)| &\leq \left(\sum_{j=0}^\infty |\ty_{o,j}| j^\rho \right) |x-y|^\rho \\
&\leq \left(\sum_{j=0}^\infty \ty_{o,j}^2 j^{2 \be} \right)^\ft \left(\sum_{j=0}^\infty j^{-2(\be - \rho)} \right)^\ft |x-y|^\rho,
\end{split}
\end{equation*} \noindent
which is finite only when $\rho < \be - \ft$.
Since  $\sum_{j=0}^\infty |\ty_j| \lesssim \sqrt{L}$ when $\ty \in \Theta(\be -1/2, L)$ and $\be >1$, the functions $\sum_{j=0}^\infty \ty_j \cos(j x)$ and $\exp\{\sum_{j=0}^\infty \ty_j \cos(j x)\}$ have the same smoothness; only the values of $L$ and $M$ differ.
The same calculation can be made when the FEXP-expansion is finite:
when $\ty \in \Theta_k(\be,L)$, then for all $x,y \in [-\pi,\pi]$,
\begin{equation} \label{smoothnessCalculation}
\begin{split}
& \left|\sum_{j = 0}^k \ty_j (\cos(jx) - \cos(jy))\right| \leq |x-y| \sum_{j=0}^k j |\ty_j| \lesssim \sqrt{L} k^{\frac{3}{2}-\be} |x-y|.
\end{split}
\end{equation} \noindent
Since
\begin{equation} \label{ghol}
|G_k(x) - G_k(y)| \leq 2 \sum_{j=1}^k \eta_j |\cos(jx)-\cos(jy)| = O(k) |x-y|,
\end{equation}
$G_k f_{d,k,\ty}$ has H\"older-smoothness $\rho=1$, its H\"older-constant being $O(k)$.
The same result holds for $H_k f_{d,k,\ty}$, since $H_k(x)= -\log(2-2\cos(x))-G_k(x)$ (see \eqref{GkHkRelation}) and $k^{\frac{3}{2}-\be} = o(k)$ for all $\be>1$.
\end{proof}
\section{Proof of Lemma B.2} \label{mixedderivsection}
For easy reference we first restate the result.
Let $W_{\s}(d)$ denote any of the quadratic forms
\[  X^t T_n^{-1}(f_{d,k})B_{\s}(d,\bd) X - \tr\left[T_n(f_o) T_n^{-1}(f_{d,k})B_{\s}(d,\bd)\right]\] \noindent
in (B.2) (in the main paper). 
Then for any $j \leq J$, $(l_1,\ldots,l_j)\in\{0,\ldots,k\}^j$ and $\s \in \mathcal S(l_1,\ldots,l_j)$, we have
\begin{equation} \label{stochTerm}
 |W_\s(d) - W_\s(d_o)| = \op(|d-d_o| n^{\ft +\e} k^{-\ft}),
\end{equation}
\begin{equation} \label{detTerm1}
\begin{split}
& \tr\left[ B_{\s}(d,\bd)\right] - \tr\left[ B_{\s}(d_o,\bar{\ty}_{d_o})\right] \\
&\quad = (d-d_o) \tr[T_{1,\s}(d_o,k)] + (d-d_o)^2 \od( n^{\e + \ft} k^{-\ft + (1 - \be /2)_+})\\
&\quad = (d-d_o) \tr[T_{1,\s}(d_o,k)] + (d-d_o)^2 \od( n^{1 -\delta} /k),
\end{split}
\end{equation}
\begin{equation}\label{detTerm2}
\begin{split}
&  \tr\left[( T_n(f_o) T_n^{-1}(f_{d,k}) - I_n)B_{\s}(d,\bd) \right]-\tr\left[( T_n(f_o) T_n^{-1}(f_{d_o,k}) - I_n)B_{\s}(d_o,\bz) \right] \\
&\quad= (d-d_o) \tr[T_{2,\s}(d_o,k)] + (d-d_o)^2 \od( n/k)+(d-d_o) \od( n^{\e + \ft}k^{-\ft}).
\end{split}
\end{equation}
\begin{proof}[Proof of Lemma B.2]

 We first prove \eqref{detTerm1}.
Developing the left-hand side in $d$ we obtain, for all $j$, $(l_1,\ldots,l_j)\in\{0,\ldots,k\}^j$ and $\s \in \mathcal S_j$,
\begin{equation} \label{taylorB}
\begin{split}
&\tr\left[ B_{\s}(d,\bd)\right] - \tr\left[ B_{\s}(d_o,\bar{\ty}_{d_o})\right] \\
&\quad =(d-d_o ) \tr\left[  B'_{\s}(d_o,\bz)\right] + \frac{(d-d_o )^2}{2} \tr\left[  B^{''}_{\s}(\bar{d},\bar{\ty}_{\bar{d}})\right],
\end{split}
\end{equation}
where $\bar{d} \in (d,d_o)$, and $B'$ and $B^{''}$ denote the first and second derivative with respect to $d$, respectively.
Writing
\begin{equation*}
\begin{split}
\tilde{B}_{\sigma(i)}(d,k) & = T_n(H_k \nabla_{\s(i)} f_{d,k}) T_n^{-1}(f_{d,k}) \\
& \quad - T_n(\nabla_{\s(i)} f_{d,k}) T_n^{-1}(f_{d,k}) T_n(H_k f_{d,k}) T_n^{-1}(f_{d,k}),
\end{split}
\end{equation*}
it follows that $B_\s^{'}(d,\bd)$ equals
\begin{equation*}
 \begin{split}
 B_\s^{'}(d,\bd)&=\sum_{i=1}^{|\s|} \prod_{j < i} T_n(\nabla_{\s(j)} f_{d,k}) T_n^{-1}(f_{d,k})
   \tilde{B}_{\sigma(i)}(d,k) \prod_{j > i} T_n(\nabla_{\s(j)} f_{d,k}) T_n^{-1}(f_{d,k}).
  \end{split}
\end{equation*}
We recall the definition of $T_{1,\sigma}$ in Lemma B.1 (main paper), and conclude that $ B_\s^{'}(d,\bd) = T_{1,\sigma}(d,k)$. Consequently, the first term on the right in \eqref{taylorB} equals $(d-d_o)\tr[T_{1,\sigma}(d_o,k)]$.

The second derivative $B_\s^{''}(d,\bd)$ equals
\begin{equation*} \label{Bsigma2nd}
\begin{split}
  &2\sum_{i_1<i_2}^{|\s|} \prod_{j < i_1} T_n(\nabla_{\s(j)} f_{d,k}) T_n^{-1}(f_{d,k})\tilde{B}_{\sigma(i_1)}(d,k)
 \prod_{i_1<j < i_2} T_n(\nabla_{\s(j)} f_{d,k}) T_n^{-1}(f_{d,k})  \\
 & \qquad \times \tilde{B}_{\sigma(i_2)}(d,k) \prod_{i_2<j } T_n(\nabla_{\s(j)} f_{d,k}) T_n^{-1}(f_{d,k}) \\
  & \quad +
\sum_{i=1}^{|\s|} \prod_{j < i} T_n(\nabla_{\s(j)} f_{d,k}) T_n^{-1}(f_{d,k})\tilde{B}_{\sigma(i)}^{'}(d,k)
 \prod_{i<j} T_n(\nabla_{\s(j)} f_{d,k}) T_n^{-1}(f_{d,k}).
\end{split}
\end{equation*}
We now show that $\tr\left[  B_\s^{''}(d,\bd)\right] = \mathbf{o}(n^{\e+ \ft} k^{-\ft + (1 - \be/2)_+})$.
From Lemma \ref{QuotientBound} and the above expression for $B_\s^{''}(d,\bd)$, it can be seen that $\tr\left[ B_\s^{''}(d,\bd)\right]$ converges to zero.
To bound the approximation error, we cannot use directly  Lemma \ref{QuotientBound} 
because the bound in  \eqref{quotientResult2} becomes too large when $\be < 2$ and $|\sigma|$ is larger than 1. We therefore use Lemmas \ref{ProductBound} and \ref{alternative_quotient_bound}.  Let $A^{''}_{\s}(d,\bd)$ be the matrix obtained after replacing every factor $T_n^{-1}(f_{d,k})$ in $B^{''}_{\s}(d,\bd)$ by $T_n(\tilde f_{d,k})$, for $\tilde f_{d,k} = f^{-1}_{d,k}/(4 \pi^2)$.
We recall from Lemma \ref{constantsLemma} that the Lipschitz constant of $f_{d,k}$ is $O(k^{(2-\be)_+})$, and for $H_k^j f_{d,k}$  and $H_k^j \nabla_{\sigma(m)} f_{d,k}$ ($m\leq |\sigma|$, $j=1,2$)  it is $O(k \log k)$. Consequently, Lemma \ref{ProductBound} implies that
\begin{eqnarray*} \label{B2:A2}
\left| \tr\left[  A^{''}_{\s}(d,\bd)\right]\right| &=& O(kn^\e )= \od ( n^{\e + \ft}k^{-\ft + (1 - \be/2)_+} )
\end{eqnarray*}
when $k \leq k_n$ and $\be>1$. It follows from Lemma  \ref{alternative_quotient_bound} that
\begin{equation*}
\begin{split}
\left| \tr\left[  A^{''}_{\s}(d,\bd)\right] - \tr\left[  B^{''}_{\s}(d,\bd)\right]\right| &= O(n^{1/2+\e} k^{(1 - \be/2)_+} ) \left(\int_{-\pi}^{\pi} H_k^2(x) dx\right)^{\ft} \\
& =\od ( n^{\e + \ft}k^{-\ft + (1 - \be/2)_+} ) .
\end{split}
\end{equation*}
Note that in the case where $B^{''}_{\s}(d,\bd)$ contains a Toeplitz matrix of the form $T_n(H_k^2 f_{d,k}) $ or $T_n(H_k^2 \nabla_{\sigma(m)} f_{d,k})$ then it contains no other Toeplitz matrix involving $H_k$ and we can set $f_2=H_k^2 f_{d,k}$ or $f_2=H_k^2 \nabla_{\sigma(m)} f_{d,k}$ and use Remark 2.1;  this leads to the above error rate.
Combining the preceding results for $| \tr[  A^{''}_{\s}(d,\bd)]|$ and $| \tr[  A^{''}_{\s}(d,\bd)] - \tr[  B^{''}_{\s}(d,\bd)]|$ we obtain that
\begin{eqnarray*}
\left|\tr\left[  B^{''}_{\s}(d,\bd)\right] \right| &=& \od ( n^{\e + \ft}k^{-\ft + (1 - \be/2)_+} )  = \od( n^{1-\delta} /k),
\end{eqnarray*}
which completes the proof of \eqref{detTerm1}.

Next, we prove \eqref{detTerm2}.
Writing 
$f_o - f_{d,k}=    f_o-f_{d_o,k} + f_{d_o,k} - f_{d,k}$, it follows that the left-hand side of  \eqref{detTerm2} equals
\begin{equation*} \label{deterministicterm:1b}
\begin{split}
& \tr\left[T_n( f_o-f_{d,k}) T_n^{-1}(f_{d,k})B_{\s}(d,\bd) \right] - \tr\left[T_n(f_o - f_{d_o,k}) T_n^{-1}(f_{d_o,k}) B_{\s}(d_o,\bz) \right] \\
&= \tr\left[T_n( f_o-f_{d_o,k}) \left\{T_n^{-1}(f_{d,k})B_{\s}(d,\bd)- T_n^{-1}(f_{d_o,k}) B_{\s}(d_o,\bz) \right\}\right] \\
&+ \tr\left[T_n( f_{d_o,k}-f_{d,k}) T_n^{-1}(f_{d,k})B_{\s}(d,\bd) \right]\\
 & := C_1 + C_2.
\end{split}
\end{equation*}
Using (\ref{fdDef}) we write $f_{d,k}= f_{d_o,k} e^{(d-d_o)H_k}$ and $f_o= f_{d_o,k}e^{ \Delta_{d_o,k}}$, and we develop $C_{\s}(d,\bd) = T_n^{-1}(f_{d,k})B_{\s}(d,\bd)$ around $d=d_o$. It follows that
\begin{equation*}
\begin{split}
C_1  &=  \tr\left[T_n( f_o-f_{d_o,k}) \left\{T_n^{-1}(f_{d,k})B_{\s}(d,\bd)- T_n^{-1}(f_{d_o,k}) B_{\s}(d_o,\bz) \right\}\right]   \\
&= (d-d_o)  \tr\left[T_n( f_o-f_{d_o,k}) C_\sigma^{'}(d_o,\bz)\right] \\
 & \quad + (d-d_o)^2 \int_0^1 (1 - u)  \tr\left[T_n( f_o-f_{d_o,k}) C_\sigma^{''}(d_u,\bar{\theta}_{d_u,k})\right]du,
\end{split}
\end{equation*}
with $d_u = u d + (1 - u) d_o$. For the first term on the right, we write,  using Lemmas \ref{ProductBound} and \ref{alternative_quotient_bound},
\begin{equation*}
\begin{split}
& \tr\left[T_n( f_o-f_{d_o,k}) C_\sigma^{'}(d_o,\bz)\right]  \\ &\quad = \frac{n }{2\pi } \int_{-\pi}^{\pi} \frac{ f_o - d_{d_o,k} }{ f_{d_o,k} } H_k(x) \cos(l_1x) \ldots \cos (l_{|\sigma|}x) dx +\mbox{error},
\end{split}
\end{equation*}
 where $\sigma$ is a partition of $\{1,...,j\}$ and the error term is
 $$O\left( \|\Delta_{d_o,k}\|_\infty n^\e ( k + k^{0.5( 3/2 - \be)_+} \frac{\sqrt{n}}{\sqrt{k}} + k^{0.5 (3/2 - \be)_+}\left( \frac{n}{ k^{2\be} } + k \|\Delta_{d_o,k}\|_\infty\right)^{\ft}  \right),
 $$
which is $o(k^{-1/2} n^{1/2-\delta})$.
Similarly, Lemmas \ref{ProductBound} and \ref{alternative_quotient_bound} imply that there exists $c\in \R$ such that  for al $d \in (d_o - \vb , d_o + \vb)$
\begin{equation*}
\begin{split}
& \tr\left[T_n( f_o-f_{d_o,k}) C_\sigma^{''}(d,\bar{\theta}_{d,k})\right] \\
& \quad = \frac{ c n}{  2 \pi }  \int_{-\pi}^{\pi} \frac{ f_o - d_{d_o,k} }{ f_{d_o,k} } H_k^2(x) \cos(l_1x) \ldots \cos (l_{|\sigma|}x) dx +\mbox{error},
\end{split}
\end{equation*}
 where the error term is of order
 $$O\left( \|\Delta_{d_o,k}\|_\infty n^\e ( k + k^{\ft( 2 - \be)_+} \frac{\sqrt{n}}{\sqrt{k}} + k^{\ft (2 - \be)_+}( \frac{n}{ k^{2\be} } + k \|\Delta_{d_o,k}\|_\infty)^{\ft}  \right) = o\left(\frac{n^{1-\delta}}{k}\right).
 $$
This implies that $C_1 = O(S_n(d))$.

Using a Taylor expansion of $C_\sigma(d,\bd)$ and of $e^{-(d-d_o)H_k}$ around $d_o$, it follows that
\begin{equation*} 
\begin{split}
C_2 &= -(d -d_o) \tr\left[T_n( f_{d_o,k}H_k ) T_n^{-1}(f_{d_o,k}) B_{\s}(d_o,\bz) \right] \\
& -\ft(d-d_o)^2 \tr\left[T_n(f_{d_o,k}H_k^2e^{- t(d^{'}-d_o)H_k } )C_\sigma(d',\bar{\theta}_{d',k}) +2 T_n( f_{d_o,k}H_k )C_\sigma^{'}(d',\bar{\theta}_{d',k})\right],
\end{split}
\end{equation*}
for some $d^{'}$ between $d$ and $d_o$.
The first term equals $\tr[T_{2,\sigma}]$. The second equals
\begin{equation*}
\begin{split}
&-\ft(d-d_o)^2 \tr\left[T_n(f_{d',k}H_k^2)C_\sigma(d', \bar{\theta}_{d',k}) + 2T_n( f_{d_o,k}H_k )C_\sigma^{'}(d^{'}, \bar{\theta}_{d',k})  \right]  \\
 &\qquad = -\frac{n (d-d_o)^2}{ 2 \pi} \int_{-\pi}^{\pi} H_k^2 (x)  \cos(l_1x) \ldots \cos (l_{|\sigma|}x) dx +\mbox{error},
\end{split}
\end{equation*}
where the error term is $O\left(n^\e \left( k + k^{0.5 (2 -\be)_+} ( nk^{-1} + kn^\e)^{1/2}\right)\right) =  o(k^{-1} n^{1-\delta})$. Therefore
\begin{equation*}
C_2 = (d-d_o) \tr[T_{2,\sigma}] + O(n/k).
\end{equation*}

Finally, to prove \eqref{stochTerm}, let $Z = T_n^{-\ft}(f_o)X$ and let
$A_d = T_n^{\ft}(f_o) T_n^{-1}(f_{d,k}) B_{\s}(d,\bd) T_n^{\ft}(f_o)$.
Then for any $|d-d_o| \leq \vb$, we have
\[W_{\s}(d) - W_{\s}(d_o) = Z^t   (A_d- A_{d_o})  Z - \tr\left( A_d- A_{d_o}  \right).\]
Writing $A_d'$ for the derivative of $A_d$ with respect to $d$, it follows that
\begin{equation} \label{aDerivative}
A_d - A_{d_o} = (d-d_o) A_{\bar{d}}',
\end{equation}
for some $\bar{d}$ between $d$ and $d_o$.
Using \eqref{fdkDer}, we find that
  \begin{equation*}
  \begin{split}
  A_d' &=  T_n^{\ft}(f_o) T_n^{-1}(f_{d,k}) T_n(H_k f_{d,k}) T_n^{-1}(f_{d,k}) T_n(B_{\s}(d,\bd)) T_n^{\ft}(f_o) \\
  & + T_n^{\ft}(f_o) T_n^{-1}(f_{d,k}) B_{\s}'(d,\bd) T_n^{\ft}(f_o).
  \end{split}
  \end{equation*}
Therefore, Lemma 2 of \citet{LRR09} and the inequalities in (1.6) (main paper)
imply that
\begin{equation} \label{aPrimeBound}
\begin{split}
& \fr A_d - A_{d_o} \fr \leq |d-d_o| \fr A_{\bar{d}}^{'} \fr \\
& \quad \leq  C |d-d_o| \| T_n^{\ft}(f_o) T_n^{-\ft}(f_{\bar{d},k}) \|^2 \prod_{i=1}^{|\s|}  \|T_n^{-\ft}(f_{\bar{d},k}) B_{\s(i)}(\bar{d},\bar{\ty}_{\bar{d},k})  T_n^{\ft}(f_{\bar{d},k})\| \\
 & + \fr T_n^{-\ft}(f_{\bar{d},k}) T_n(H_k \nabla_{\s(i)} f_{\bar{d},k}) T_n^{-\ft}(f_{\bar{d},k})\fr  \\
& = |d-d_o| n^\e O\left( \fr T_n^{-\ft}(f_{\bar{d},k}) T_n(H_k f_{\bar{d},k}) T_n^{-\ft}(f_{\bar{d},k})\fr  + \fr T_n^{-\ft}(f_{\bar{d},k}) T_n(H_k \nabla_{\s(i)} f_{\bar{d},k}) T_n^{-\ft}(f_{\bar{d},k})\fr\right),
\end{split}
\end{equation}
where $\s(i)$ can also be the empty set, in which case $ \nabla_{\s(i)} f_{\bar{d},k} =  f_{\bar{d},k}$.
We bound the terms between brackets using Lemma \ref{QuotientBound}, with $p=2$, $f=f_{\bar{d},k}$ and $g_1=g_2$ equalling either $H_k f_{\bar{d},k}$ or $H_k \nabla_{\s(i)} f_{\bar{d},k}$. The H\"older constants of these functions are given by Lemma \ref{constantsLemma}. Hence we find that
\begin{equation} \label{frobBound1}
\begin{split}
& \fr T_n^{-\ft}(f_{\bar{d},k}) T_n(H_k f_{\bar{d},k}) T_n^{-\ft}(f_{\bar{d},k})\fr^2  =\tr\left[  \left(T_n^{-1}(f_{\bar{d},k}) T_n(H_k f_{\bar{d},k})\right)^2\right] \\
  &\quad = \frac{n}{2\pi} \int_{-\pi}^{\pi} H_k^2(x) dx +  O( n^\e ( k + k^{2 - \be} ))
= O( n^{1 -1/(2\be)} (\log n)^{1/(2\be)}).
\end{split}
\end{equation}
The last inequality follows from equation \eqref{i1} in Lemma \ref{integralLemma} and the fact that $k=k_n$ and $\be > 1$. Similarly, it follows  that
\begin{equation} \label{frobBound2}
\fr T_n^{-\ft}(f_{\bar{d},k}) T_n(H_k \nabla_{\s(i)} f_{\bar{d},k}) T_n^{-\ft}(f_{\bar{d},k})\fr^2 = O( n^{1 -1/(2\be) } (\log n)^{1/(2\be)}).
\end{equation}
Inserting \eqref{aPrimeBound}, \eqref{frobBound1} and \eqref{frobBound2} in \eqref{aDerivative}, we find that $\fr A_d-A_{d_o}\fr \leq |d-d_o| n^{1/2-1/(4\be) + \e} $, for all $|d-d_o| \leq \vb$ and all $\e>0$, when $n$ is large enough.
Consequently, we can apply Lemma 1.3 with $A = (A_d-A_{d_o})/\fr A_d - A_{d_o} \fr$, so that when $n$ is large enough
\begin{eqnarray} \label{chain:1}
\sup_{|d-d_o| \leq \vb} P_o\left( |W_{\s}(d) - W_{\s}(d_o)|  > |d-d_o| n^{2\e + \frac{1}{2}- \frac{1}{4\be}} \right) \leq e^{-n^{\e}/8}.
\end{eqnarray}
Using the above computations with $|d-d'| \leq n^{-2}$, we obtain
$$| W_{\s}(d) - W_{\s}(d') |\leq n^{-2+\e} \left( n + Z^t Z \right).$$
Hence, for all $\e<\ft$ and $c>0$,
\begin{equation}\label{chain:2}
P_o \left( \sup_{|d'-d|\leq n^{-2}} | W_{\s}(d) - W_{\s}(d') | > n^{-\e} \right) \leq
P_o\left( Z^t Z > n^{2-2\e} \right) \leq e^{-c n},
\end{equation}
provided $n$ is large enough. Hence, we obtain \eqref{stochTerm} by combining (\ref{chain:1}) and (\ref{chain:2}) in a simple chaining argument over the interval $(d_o-\vb,d_o+\vb)$.

\end{proof}

\end{document}